\documentclass[letter, 11 pt, reqno]{article}

\usepackage{xr}
    \usepackage{graphicx}
    \usepackage{styleset}
    \usepackage{macros}
\graphicspath{ {./images/} }

\begin{document}

\title{Surgery and Excision for Furuta--Ohta Invariants on Homology $S^1 \times S^3$}
\author{Langte Ma \thanks{ltma@sjtu.edu.cn}}
\address{Shanghai Jiao Tong University, Shanghai 200240, China}

\maketitle

\begin{abstract}
We prove a surgery formula and an excision formula for the Furuta--Ohta invariant $\lambda_{FO}$ defined on homology $S^1 \times S^3$, which provides more evidence on its equivalence with the Casson--Seiberg--Witten invariant $\lambda_{SW}$. These formulae are applied to compute $\lambda_{FO}$ of certain families of manifolds obtained as mapping tori under diffeomorphisms of $3$-manifolds. In the course of the proof, we establish the existence of asymptotic values for finite energy instantons over end-cylindrical manifolds perturbed by exponentially decay holonomy perturbations, and give a complete description of the structure of the moduli space of charge-zero instantons over any non-compact $4$-manifold whose end is modeled on $[0, \infty) \times T^3$, and whose homology is given by $H_*(D^2 \times T^2; \Z)$. 
\end{abstract}

\tableofcontents

\section{Introduction}

Gauge theory has been seen as an effective method of producing invariants that detect exotic structures on $4$-manifolds since the celebrated work of Donaldson \cite{D83, D87b, D90}. In contrast to the existence of non-smoothable closed topological $4$-manifolds, Freedman\textendash Quinn \cite{FQ90} proved that every open topological $4$-manifold is smoothable. So It is a natural question to ask how many distinct smooth structures a given open topological $4$-manifold supports. In the study of this problem, Furuta\textendash Ohta \cite{FO93} introduced the Furuta\textendash Ohta invariant $\lambda_{FO}(X)$ by means of Yang\textendash Mills theory over any closed $4$-manifold $X$ which shares the same integral homology as $S^1 \times S^3$ and whose infinite cyclic cover $\tilde{X}$ has the same homology as $S^3$. In the same paper Furuta\textendash Ohta conjectured that their invariant $\lambda_{FO}(X) \mod 2$ reduces to the Rohlin invariant $\mu(X)$ which is defined as the Rohlin invariant of any connected $3$-manifold Poincaré dual to a primitive class in $H^1(X; \Z)$. The strength of this conjecture, if proved to be true, would imply that every once punctured topological $4$-manifold admits uncountably many distinct smooth structures, which was proved independently by Gompf \cite{G93} with the help of Taubes' work \cite{T87}.

Instead of approaching the above conjecture of Furuta--Ohta  directly, Mrowka, Ruberman, and Saveliev \cite{MRS11} introduced another invariant $\lambda_{SW}(X)$ using Seiberg--Witten theory, and proved that $\lambda_{SW}(X) \mod 2$ reduces to the Rohlin invariant $\mu(X)$. So the left problem is to establishing the equivalence between the Casson--Seiberg--Witten invariant $\lambda_{SW}$ and the Furuta--Ohta invariant $\lambda_{FO}$. The Casson--Seiberg--Witten invariant $\lambda_{SW}$ has been proved by the author \cite{M1} to satisfy a surgery relation. So the purpose of this paper is to understand how the Furuta--Ohta invariant $\lambda_{FO}$ changes under torus surgery and excision, which in turn provides more evidence on its equivalence with the Casson--Seiberg--Witten invariant $\lambda_{SW}$. 

We start by pinning down the the type of manifolds we will work with and the definition of the Furuta--Ohta invariant. 

\begin{dfn}
Let $X$ be a closed oriented smooth $4$-manifold satisfying the homological assumption
\[
H_*(X; \Z) \cong H_*(S^1 \times S^3; \Z).
\]
We call $X$ an admissible homology $S^1 \times S^3$ if it further satisfies the following property: for any non-trivial $U(1)$-representation $\rho: \pi_1(X) \to U(1)$, one has 
\begin{equation}
H^1(X; \C_{\rho}) = 0,
\end{equation}
which $\C_{\rho}$ is the local system given by the unitary representation $\rho$ with fiber $\C$. 
\end{dfn}

Let $X$ be an admissible homology $S^1 \times S^3$ and $E$ the trivial $\C^2$-bundle over $X$. After fixing a metric on $X$, one can consider anti-self-dual (ASD) $SU(2)$-connections $A$ over $E$, i.e. $* F_A = -F_A$, where $F_A$ is the curvature $2$-form of the connection $A$, and $*$ is the Hodge star given by the metric. The $SU(2)$-automorphisms of $E$ form a group $\G$ which we refer to as the gauge group of $E$ and whose elements are called gauge transformations. The gauge group acts on $SU(2)$-connections via pull-backs. An $SU(2)$-connection $A$ is said to be irreducible if its stabilizer is discrete under this action. By abusing the terminology, we shall refer to either an ASD connection or its orbit under the gauge group action as an instanton over $X$. The moduli space of irreducible instantons is defined as 
\[
\M^*(X):= \left\{ A \text{ an irreducible $SU(2)$-connection on $E$}: * F_A = -F_A \right\} \bign/ \G.
\]
Following the idea of Donaldson \cite{D87}, Ruberman--Saveliev \cite{RS1} proved that after adding a generic perturbation $\sigma$ to the ASD equation, the perturbed moduli space $\M^*_{\sigma}(X)$ becomes a compact orientable $0$-manifold whose orientation is specified by the orientation of $X$ and a choice of a generator $1_X \in H^1(X; \Z)$. The Furuta--Ohta invariant is then defined as a quarter of the signed count of perturbed irreducible instantons over $E$:
\[
\lambda_{FO}(X):= \frac{1}{4} \# \M^*_{\sigma}(X).
\]
This count is proved \cite{RS1} to be independent of the choice of metrics and perturbations on $X$ due to its admissibility. If one changes the sign of the generator $1_X$, the sign of $\lambda_{FO}(X)$ should be reversed. However, we have chosen to make such a choice implicit in the notation. It is worth noting that when $X=S^1 \times Y$ is given by the product of $S^1$ with an integral homology sphere $Y$, the Furuta--Ohta invariant coincides with the Casson invariant \cite{AM90} of $Y$.

The first topological operation we consider is called torus surgery. We give a brief description here, and a detailed one later in \autoref{sffo}. Let $X$ be an admissible integral homology $S^1 \times S^3$ with a fixed generator $1_X \in H^1(X; \Z)$, and $\iota: \mathcal{T} \hookrightarrow X$ an embedded $2$-torus satisfying the following homological assumption:
\[
\text{ the induced map } \iota_*: H_1(\mathcal{T}; \Z) \longrightarrow H_1(X; \Z) \text{ is surjective}. 
\]
We shall refer to such a torus as an essentially embedded torus. Let us write $\nu(\mathcal{T})$ for a tubular neighborhood of $\mathcal{T}$ in $X$. Since the intersection form of $X$ vanishes, we can fix an identification $\nu(\mathcal{T}) \cong D^2 \times T^2$ which we refer to as a framing. Let $M=\overline{X \backslash \nu(\mathcal{T})}$ be the closure of the complement of the neighborhood. It is straightforward to compute that $H_*(M; \Z) \cong H_*(D^2 \times T^2; \Z)$. The framing provides us with a basis $\{[\mu], [\lambda], [\gamma]\}$ for $H_1(\partial \nu(\mathcal{T}); \Z)$. We require $[\gamma] \in H_1(X; \Z)$ to be the dual of the generator $1_X \in H^1(X; \Z)$, $[\mu]$ to be the class represented by the meridian of $\mathcal{T}$, and $[\lambda]$ to be null-homologous in $M$. Given a relatively prime pair $(p, q)$, performing $(p,q)$-surgery along $\mathcal{T}$ results in the $4$-manifold 
\[
X_{p, q} = M \cup_{\varphi_{p,q}} D^2 \times T^2,
\]
where with respect to the basis $\{[\mu], [\lambda], [\gamma]\}$ the gluing map $\varphi_{p, q}$ is specified by 
\begin{equation}
\begin{pmatrix}
p & r & 0 \\
q & s & 0 \\
0 & 0 & 1
\end{pmatrix}
\in SL(3, \Z).
\end{equation}
Technically an element of $SL(3, \Z)$ only specifies an element in the mapping class group of $T^3$ (cf. \cite{H76}). Since the diffeomorphism type of $X_{p,q}$ is determined by the isotopy class of the gluing map, this ambiguity causes no further confusion. We will see later in \autoref{sffo} that the $(1, q)$-surgered manifold $X_{1,q}$ is still an admissible integral homology $S^1 \times S^3$ with a homology orientation induced from that of $X$. As for the $(0,1)$-surgered manifold, we have $H_*(X_{0,1}; \Z) \cong H_*(S^2 \times T^2; \Z)$. Denote by $w_{\mathcal{T}} \in H^2(X_{0,1}; \Z/2)$ the class that is dual to the mod $2$ class of the core $\{0 \} \times T^2$ in the gluing copy $D^2 \times T^2$. We write $D^0_{w_{\mathcal{T}}}(X_{0,1})$ for the signed count of gauge equivalence classes of irreducible anti-self-dual $SO(3)$-connections on the bundle $SO(3)$-bundle $E_0$ over $X_{0,1}$ characterized by 
\[
p_1(E_0)=0 \text{ and } w_2(E_0)=w_{\mathcal{T}}. 
\]
We shall explain later in \autoref{sffo} that the count is independent of the choices of generic perturbations. The torus surgery formula relates $\lambda_{FO}(X)$, $\lambda_{FO}(X_{1,q})$, and $D^0_{w_{\mathcal{T}}}(X_{0,1})$ as follows.

\begin{thm}\label{surq}
Let $X$ be an admissible homology $S^1 \times S^3$ and $\mathcal{T} \subset X$ an essentially embedded torus. Given an integer $q \in \Z$, after fixing a generator $1_X \in H^1(X; \Z)$ the Furuta--Ohta invariant of the $(1, q)$-surgered manifold $X_{1, q}$ is related to the Furuta--Ohta invariant of $X$ by 
\[
\lambda_{FO}(X_{1, q}) = \lambda_{FO}(X) + qD^0_{w_{\mathcal{T}}}(X_{0,1}). 
\]
\end{thm}

Recall that the surgery formula of the Casson invariant \cite{AM90} with respect to a knot $\mathcal{K}$ in an integral homology sphere $Y$ takes the form 
\[
\lambda\left(Y_{{1 \over q}}(\mathcal{K})\right) = \lambda(Y) + {q \over 2}\Delta''_{\mathcal{K}}(1),
\]
where $\Delta_{\mathcal{K}}(t)$ is the symmetrized Alexander polynomial of $\mathcal{K}$. Since the Furuta--Ohta invariant of the product $S^1 \times Y$ agrees with the Casson invariant $\lambda(Y)$, comparing with the surgery formula of the Furuta--Ohta invariant gives us the following result.
\begin{cor}\label{0d0s}
Let $\mathcal{K} \subset Y$ be a knot in an integral homology sphere. Then 
\[
2D^0_{w_{\mathcal{T}}}(S^1 \times Y_0(K))  = \Delta''_{\mathcal{K}}(1).
\]
\end{cor}

The surgery formula of the Furuta--Ohta invariant should be compared with that of the Casson--Seiberg--Witten invariant \cite{M1}:
\[
\lambda_{SW}(X_{1, q}) = \lambda_{SW}(X) + q \mathcal{SW}(X_{0,1}),
\] 
where $\mathcal{SW}(X_{0,1})$ is the Seiberg--Witten invariant of $X_{0,1}$ computed in the chamber specified by small perturbations (cf. \cite[(1.4)]{M1}). The Casson--Seiberg--Witten invariant is defined using Seiberg--Witten theory \cite{MRS11} as a combination of the count of irreducible monopoles and an index-theoretical correction term. If one could further verify Witten's conjecture \cite{W94} relating Donaldson invariants and Seiberg--Witten invariants in the case of non-simply connected $4$-manifolds with $b^+=1$ following the approach of Pidstrigach--Tyurin \cite{P94, PT95} and Feehan--Leness \cite{FL15, FL18},  then the surgery formulae would imply the equivalence between $\lambda_{FO}$ and $\lambda_{SW}$ for admissible homology $S^1 \times S^3$ that can be obtained via such torus surgeries on manifolds where the equivalence has been verified. 

On the other hand, we apply the surgery formula to give an independent computation of the Furuta--Ohta invariant for manifolds arising as the mapping torus of finite order diffeomorphism considered in \cite{LRS20}.

\begin{prop}\label{fod}
Let $\mathcal{K}$ be a knot in an integral homology sphere $Y$. Given an integer $n > 1$, we denote by $\Sigma_n(Y, \mathcal{K})$ the $n$-fold cyclic cover of $Y$ branched along $\mathcal{K}$ and $X_n(Y, \mathcal{K})$ the mapping torus of $\Sigma_n(Y, \mathcal{K})$ under the generating covering transformation. Assume that $\Sigma_n(Y, \mathcal{K})$ is a rational homology sphere. Then 
\[
\lambda_{FO}(X_n(Y, \mathcal{K})) = n\lambda(Y) + {1 \over 8} \sum_{m=1}^{n-1}\sign^{m/ n}(Y, \mathcal{K}),
\]
where $\sign^{m/n}(Y, \mathcal{K})$ is the Levine--Tristram signature of $\mathcal{K}$ (cf. \cite{L69, T69}). 
\end{prop}

The argument of \cite[Proposition 6.1]{LRS20} implies that the mapping torus $X_n(Y, \mathcal{K})$ of $\Sigma_n(Y, \mathcal{K})$ under the generating covering transformation is an admissible homology $S^1 \times S^3$ if and only if $\Sigma_n(Y, \mathcal{K})$ is a rational homology sphere.  So our assumption on $\Sigma_n(Y, \mathcal{K})$ is necessary. This computation was carried out in \cite[Theorem 6.4]{LRS20} using a more direct method by relating the Furuta--Ohta invariant of $X_n(Y, \mathcal{K})$ to the equivariant Casson invariant of $\Sigma_n(Y, \mathcal{K})$. In \cite{M1}, the author also computed the Casson--Seiberg--Witten invariant for $X_n(Y, \mathcal{K})$ without assuming that $\Sigma_n(Y, \mathcal{K})$ is a rational homology sphere. However, when $\Sigma_n(Y, \mathcal{K})$ fails to be a rational homology sphere, the definition of the Furuta--Ohta invariant requires a proper generalization. 

We move to the next topological operation which we call torus excision in this paper. The idea is to replace the torus neighborhood $D^2 \times T^2$ with a homology $D^2 \times T^2$ and glue it to the complement by further applying a diffeomorphism. Let $(X_1, \mathcal{T}_1)$ and $(X_2, \mathcal{T}_2)$ be two pairs of essentially embedded torus in admissible homology $S^1 \times S^3$'s as above. We fix framings for both $\nu(\mathcal{T}_1)$ and $\nu(\mathcal{T}_2)$ to get a basis $\{\mu_i, \lambda_i, \gamma_i\}$ of $H_1(\partial\nu(\mathcal{T}_i))$ as before. To emphasize the orientation, we write $X_1 = M_1 \cup \nu(\mathcal{T}_1)$ and $X_2 = \nu(\mathcal{T}_2) \cup M_2$, i.e. $\partial M_1 = -\nu(\mathcal{T}_1)$ and $\partial \nu(\mathcal{T}_2) =-M_2$ are identified with $T^3$ in an orientation-preserving manner. Suppose $\varphi: \partial M_2 \to \partial M_1$ is a diffeomorphism so that the glued manifold 
\[
X_1 \#_{\varphi} X_2:=M_1 \cup_{\varphi} M_2
\]
is still an admissible homology $S^1 \times S^3$. Moreover, we write $X_{1, \varphi} = M_1 \cup_{\varphi} D^2 \times T^2$ and $X_{2, \varphi}=D^2 \times T^2 \cup_{\varphi} M_2$. We will give more explanations on what the gluing map means concretely in \autoref{fsfo}. Roughly speaking, $D^2 \times T^2$ is glued to $M_1$ the same way as $M_2$ is and to $M_2$ the same way as $M_1$ is.  Since the admissible assumption is purely homological, we see that both $X_{1, \varphi}$ and $X_{2, \varphi}$ are both admissible. The excision formula of the Furuta--Ohta invariants states the following. 

\begin{thm}\label{exif}
Suppose $X_1$ and $X_2$ are two admissible homology $S^1 \times S^3$’s with embedded tori $\mathcal{T}_1$ and $\mathcal{T}_2$ respectively given by the above description. Then the Furuta--Ohta invariant is additive under excision along these two tori:  
\[
\lambda_{FO}(X_1 \#_{\varphi} X_2)  = \lambda_{FO} (X_{1, \varphi}) + \lambda_{FO}(X_{2, \varphi}).
\]
\end{thm}

Note that $X_{i, \varphi}$ is obtained from $X_i$ via a torus surgery. When the gluing map $\varphi$ has the form we considered in \autoref{surq}, we can further expand the formula as 
\[
\begin{split}
\lambda_{FO}(X_1 \#_{\varphi_{1, q}} X_2)  & = \lambda_{FO} (X_1) +  qD^0_{w_{\mathcal{T}}}(X_{1, \varphi_{0,1}}) \\
& + \lambda_{FO}(X_2) +  qD^0_{w_{\mathcal{T}}}(X_{2, \varphi_{0,1}}). 
\end{split}
\]
We will see later in examples there are certain interesting gluing maps that are not of the form we considered in the surgery formula. Thus we need a generalized surgery formula to compare $\lambda_{FO}(X_{i, \varphi})$ with $\lambda_{FO}(X_i)$. It turns out there is an extra term coming out in the formula as we go through the proof of \autoref{surq}, which is caused by the contribution of some `bifurcation points' (cf. \autoref{bfp}) in the moduli space of the torus complements $M_i$. 

We note that the fiber sum operation considered in \cite{M2} is a special case of \autoref{exif} The fiber sum of $(X_1, \mathcal{T}_1)$ and $(X_2, \mathcal{T}_2)$ is given by gluing the torus complements by a map interchanging the meridian $\mu$ and longitude $\lambda$:
\[
X_1 \#_{\mathcal{T}} X_2:=M_1 \cup_{\varphi_{\mathcal{T}}} M_2,
\]
where with respect to the basis $\{\mu_i, \lambda_i, \gamma_i\}$, $\varphi_{\mathcal{T}}$ is given by 
\begin{equation}\label{fsrm}
\varphi_{\mathcal{T}}=
\begin{pmatrix}
0 & 1 & 0 \\
1 & 0 & 0 \\
0 & 0 & 1
\end{pmatrix}.
\end{equation}
Note that $M_1 \cup_{\varphi_{\mathcal{T}}} D^2 \times T^2= X_1$, $D^2 \times T^2 \cup_{\varphi_{\mathcal{T}}} M_2 = X_2$. We conclude that the Furuta--Ohta invariant is additive under taking torus fiber sum. 

\begin{cor}\label{fibs}
Suppose $X_1$ and $X_2$ are two admissible homology $S^1 \times S^3$’s with embedded tori $\mathcal{T}_1$ and $\mathcal{T}_2$ respectively given by the above description. Then the Furuta--Ohta invariant is additive under fiber-sum along these two tori:  
\[
\lambda_{FO}(X_1 \#_{\mathcal{T}} X_2)= \lambda_{FO}(X_1) + \lambda_{FO}(X_2).
\]
\end{cor}

In the product case $X_i=S^1 \times Y_i$ and $\mathcal{T}_i = S^1 \times \mathcal{K}_i$, i=1, 2, the fiber sum $X_1 \#_{\mathcal{T}} X_2$ is the product of $S^1$ with the knot splicing $Y_1 \#_{\mathcal{K}} Y_2$ of the pairs $(Y_1, \mathcal{K}_1)$ and $(Y_2, \mathcal{K}_2)$. Then the fiber sum formula for the Furuta--Ohta invariant recovers the additivity of the Casson invariant under knot splicing. 

Both the proofs of \autoref{surq} and \autoref{exif} rely on understanding the moduli space of instantons on the torus complement $X \backslash \mathcal{T}$ on which we choose metrics that are cylindrical on $\nu(\mathcal{T}) \backslash \mathcal{T}$. The moduli space of instantons over end-cylindrical manifolds has been studied by various authors, see for instance \cite{T93} and \cite{MMR}. However, there are difficulties in the case we are considering that were not fully understood previously. One of the difficulties is caused by the fact that $b^+(X \backslash \mathcal{T})=0$, which prevents us from using metric perturbations to get rid of the reducible instantons. Another issue is that the asymptotic manifold of $M$ is the $3$-torus $T^3$ whose moduli space has a certain degeneracy, especially when we consider the trajectories on the torus complement flowing to the singular points in the moduli space. To resolve these issues, we combine holonomy perturbations considered by Donaldson \cite{D87} with the ‘center manifold’ technique developed in \cite{MMR} to give a complete description of the structure of such moduli spaces. 

For the rest of this section, we state the results concerning moduli spaces. Let $Z = M \cup [0, \infty) \times Y$ be an end-cylindrical manifold where $M$ is a compact $4$-manifold with boundary, $Y$ is a closed $3$-manifold. A metric $g$ is fixed on $Z$ so that $g|_{[0, \infty) \times Y} = dt + h$, where $h$ is a metric on $Y$. The charge-zero moduli space $\M_{\sigma}(Z)$ of perturbed instantons on $Z$ consists of gauge equivalence classes of $SU(2)$-connections $A$ on the trivial bundle $\C^2 \times Z$ satisfying the following:
\begin{enumerate}
\item The self-dual part of the curvature equals the perturbation, i.e. $F^+_A=\sigma(A)$.
\item The energy of $A$ is finite, i.e. $\int_Z |F_A|^2 d\vol< \infty$. 
\item The instanton charge of $A$ vanishes, i.e. $\int_Z \tr(F_A \wedge F_A) =0$. 
\end{enumerate}
The perturbation $\sigma$ is chosen to be gauge-equivariant and satisfies an exponential decay condition along the cylindrical end:
\[
\| \sigma(A) \|_{L^{\infty}(\{t \} \times Y)} \leq Ce^{-\mu t},
\]
where $C$ and $\mu$ are some positive constants independent of $A$. The space $\mathcal{P}_{\mu}$ of perturbations will be constructed more carefully as a Banach space. We write
\begin{equation}
\chi(Y):=\Hom(\pi_1(Y), SU(2)) \bign/\Ad
\end{equation}
for the $SU(2)$-character variety of the $3$-manifold $Y$. Via the holonomy map, $\chi(Y)$ is identified with the gauge equivalence classes of the flat $SU(2)$-connections on $Y$. Our next result concerns with the existence of the asymptotic map on $\M_{\sigma}(Z)$. 

\begin{thm}\label{asyom}
Given $[A] \in \M_{\sigma}(Z)$, the limit $\lim_{t \to \infty} [A|_{\{t \} \times Y}]$ exists and lies in $\chi(Y)$. The assignment of $[A]$ to its limit defines a continuous map
\[
\partial_+: \M_{\sigma}(Z)  \longrightarrow \chi(Y).
\]
\end{thm}

\begin{rem}\label{r1.8}
The convergence here is essentially stronger than the one provided by Uhlenbeck’s weak compactness (cf. \cite{U82, W04}) for connections of small energy. As one will see in the proof later, the limit of $[A|_{\{t\} \times Y}]$ is taken with respect to the $C^{\infty}$-topology, after suitable gauge transformations. 
\end{rem}

\begin{rem}
The core step in the proof of \autoref{asyom} is to estimate the length of perturbed downward gradient flowlines of the Chern--Simons functional on $Y$. Let us write $A|_{\{t\} \times Y} = B(t)$. Suppose, when restricted to the cylindrical end $[0, \infty) \times Y$, the perturbed ASD equation takes the form 
\[
\dot{B}(t) = -\grad \cs(B(t)) + \mathfrak{p}(B(t)),
\] 
where $\cs$ is the Chern--Simons functional on $Y$, and $\mathfrak{p}$ is some functional on the space of connections of $Y$ which we regard as a perturbation. The existence of limits of perturbed flowlines has been established in the following cases among existing literature. 
\begin{enumerate}[label=(\alph*)]
\item There is a perturbed functional $\cs_{\mathfrak{p}}$ satisfying 
\[
-\grad \cs_{\mathfrak{p}}(B(t)) =  -\grad \cs(B(t)) + \mathfrak{p}(B(t)).
\]
Moreover, one requires either the critical points of $\cs_{\mathfrak{p}}$ are either Morse or Morse--Bott. 
\item $\|\mathfrak{p}(B(t))\| \leq \alpha \| \grad \cs(B(t))\|$ for some $\alpha <1$ when $t \gg 0$. 

\item The perturbation $\mathfrak{p}$ is analytic. 
\end{enumerate}
Case (a) is usually encountered when defining Floer homology. One may consult \cite{D02} for $\cs_{\mathfrak{p}}$ being Morse and \cite{L18} for $\cs_{\mathfrak{p}}$ being Morse--Bott. Case (b) was considered in \cite{MMR} where $\mathfrak{p}$ is a metric perturbation. Case (c) was proved in broader circumstances by Simon \cite{S83} and Feehan-Maridakis \cite{FM20}. Sometimes people (see for instance \cite{MMS, N00, K04}) also consider perturbations with compact support, which falls into case (b). We will essentially show the `center manifold' technique in \cite{MMR} works for perturbations parametrized by a Banach space under our consideration. We also note \autoref{asyom} is proved in a slightly more general context in \autoref{t5.1} where we shall consider based moduli spaces. 
\end{rem}

We refer to $\partial_+$ as the asymptotic map. Now we focus on the type of manifolds of our primary interests. Let $Z$ be a manifold with cylindrical end satisfying the following: 
\begin{enumerate}
\item The integral homology of $Z$ is the same as that of $D^2 \times T^2$, i.e. $H_*(Z; \Z) \cong H_*(D^2 \times T^2; \Z)$. 
\item The cylindrical end of $Z$ is modeled on $[0, \infty) \times T^3$. The torus $T^3$ is endowed with a flat metric $h$. 
\end{enumerate}
We shall refer to a pillowcase as the quotient of a torus $T^2$ under the hypoelliptic involution. Thus a pillowcase is an orbifold smoothable to $S^2$ with $4$ singular points having $\Z/2$ as their isotropy groups. One will see later in \autoref{tims} that each central connection in $\chi(T^3)$ is a singular point, and there are eight of them up to gauge equivalence. We denote by $\mathfrak{C} \subset \chi(T^3)$ the set of these central classes. Recall an $SU(2)$-connection $A$ on $Z$ is said to be irreducible if its stabilizer under the action of the gauge group is discrete. We say $A$ is reducible if $A$ is not irreducible. In this way, we can decompose the moduli space into two parts consisting of irreducible instantons $\M^*_{\sigma}(Z)$ and reducible instantons $\M^{\Red}_{\sigma}(Z)$ respectively:
\[
\M_{\sigma}(Z) = \M^*_{\sigma}(Z) \cup \M^{\Red}(Z).
\]
We now give a complete description of the moduli space over $Z$. 

\begin{thm}\label{str}
Let $Z$ be a smooth Riemannian non-compact $4$-manifold whose end is modeled on $[0, \infty) \times T^3$ as a product of $[0, \infty)$ with the flat torus. Suppose the homology of $Z$ satisfies
\[
H_*(Z; \Z) \cong H_*(D^2 \times T^2; \Z).
\]
After fixing an homology orientation of $Z$ and generically small perturbation $\sigma$, the charge-zero moduli space of perturbed finite-energy anti-self-dual instantons $\M_{\sigma}(Z)$ is a compact smooth oriented stratified space with the following structures: 
\begin{enumerate}[label=(\alph*)]
\item The reducible locus $\M^{\Red}_{\sigma}(Z)$ is a pillowcase whose singular points consist of the four gauge equivalence classes of central flat $SU(2)$-connections on $Z$. 
\item The irreducible stratum $\M^*_{\sigma}(Z)$ is a smooth oriented $1$-manifold with a finite number of components, each of which is diffeomorphic to either the circle $S^1$ or the open interval $(0,1)$. 
\item The ends of the closure of the open arcs in $\M^*_{\sigma}(Z)$ lie in $\M^{\Red}_{\sigma}(Z)$ away from the singular points. Near each end $[A] \in \M^{\Red}_{\sigma}(Z)$, the moduli space $\M_{\sigma}(Z)$ is modeled on a neighborhood of $0$ in the zero set $\mathfrak{o}^{-1}(0)$, where 
\begin{equation*}
\begin{split}
\mathfrak{o}: \R^2 \oplus \R_+ &\longrightarrow \C \\
(x_1, x_2, r) & \longmapsto (x_1 + ix_2) \cdot r. 
\end{split}
\end{equation*}
\item Given a submanifold of $\chi(T^3)$, the perturbation $\sigma$ can be chosen so that the asymptotic map, when restricted to the irreducible stratum, $\partial_+: \M^*_{\sigma}(Z) \to \chi(T^3)$ is $C^2$, and transverse to that given submanifold.
\item The asymptotic values of irreducible instantons miss the central classes, i.e. 
\[
\partial_+\big (\M^*_{\sigma}(Z)\big) \cap \mathfrak{C} = \varnothing.
\]
\end{enumerate}
\end{thm}

It is natural to wonder whether the techniques in the proof of \autoref{str} can be applied to some other situations. We briefly discuss the difficulties that one might encounter. In general, one can replace the homological assumption on $Z$ by $b^+ = 0$, and the end of $Z$ by $[0, \infty) \times Y$ for an arbitrary closed $3$-manifold $Y$. Moreover, one may consider moduli spaces of perturbed finite energy instantons of higher charges, say $\kappa$. We denote the corresponding moduli space by $\M_{\kappa, \sigma}(Z)$. First, the asymptotic map $\partial_+: \M_{\kappa, \sigma}(Z) \to \chi(Y)$ still exists, since the proof of \autoref{asyom} makes no use of the charge-zero assumption. Then the asymptotic map $\partial_+$, restricted to the irreducible stratum $\M^*_{\kappa, \sigma}(Z)$, is generically transverse to the stratum of $\chi(Y)$ consisting of ‘regular’ points following the argument of \autoref{TIS}. 

The first difference shows up in the compactness of $\M_{\kappa, \sigma}(Z)$. In the presence of higher charge instantons, the phenomena of Uhlenbeck’s bubbling and energy-escape along the end will occur when one passes to subsequences. Moreover, holonomy perturbations obstruct the standard local bootstrapping argument (cf. \cite{U82, DK90}) for small energy instantons over a geodesic four ball. So even taking ‘ideal connections’ and ‘broken trajectories’ into account, the convergence of connection forms will only be in $L^p_1$-topology for $p \geq 2$, which is in contrast to our case where we can work with $L^2_k$-connections with $k$ arbitrarily large, and still get compactness. This issue was explained in \cite{K04, KM11}. 

The second major difference comes from the interaction between the irreducible stratum $\M^*_{\kappa, \sigma}(Z)$ and the reducible stratum $\M^{\Red}_{\kappa, \sigma}(Z)$ under generic perturbations. The local description in (c) of \autoref{str} relies on the assumption $H_*(Z; \Z) \cong H_*(D^2 \times T^2; \Z)$ which enables one to understand the ‘Kuranishi obstruction map’ in its proof (cf. \autoref{nbr}) explicitly under generic perturbations. When the homology of $Z$ is more complicated or the instanton charge $\kappa$ is large, the dimension of the moduli space $\M_{\kappa, \sigma}(Z)$ is large, which may result in complicated stratifications even in generic cases. The best scenario one may expect is that near the reducible stratum the moduli space $\M_{\kappa, \sigma}(Z)$ looks like a ‘cone bundle neighborhood’ as described in \cite[Chapter 14]{MMR}. But as we have seen in \autoref{str}, ‘bifurcations’ already showed up over the reducible stratum, which forms a more singular picture compared to cone bundle neighborhoods.

\vspace{1mm}

\subsection*{Outline}
Here we give an outline of this article. \autoref{pre} introduces the set-up for the moduli space. \autoref{s3} discusses holonomy perturbations with exponential decay along cylindrical ends. \autoref{s4} develops regularity estimates for instantons perturbed by holonomy perturbations. \autoref{EAM} establishes the existence of the asymptotic value for the perturbed instantons, i.e. \autoref{asyom}. \autoref{tims} deduces the transversality of the irreducible moduli space and the asymptotic map. \autoref{trl} is devoted to describing the reducible locus together with its neighborhood under small generic perturbations. \autoref{sffo} and \autoref{fsfo} prove the surgery formula \autoref{surq} and the excision formula \autoref{exif} respectively. 

\subsection*{Acknowledgment}
The author would like to express his gratitude to his advisor Daniel Ruberman for helpful discussions and encouragement. He also wants to thank Cliff Taubes for sharing his expertise and communicating a proof of \autoref{p5.7}. A further acknowledgment is to Simon Donaldson for enlightening discussions. Finally, the author wants to thank the anonymous referees for their valuable comments on improving exposition and mathematics. 

\section{Unperturbed Moduli Spaces}\label{pre}

In this section, we set up the definitions and notations for the moduli space of instantons over $4$-manifolds and flat connections over $3$-manifolds. For more details one may consult \cite{DK90, D02}. 

\subsection{Moduli Spaces of Instantons over $4$-Manifolds}

\begin{dfn}
A Riemannian $4$-manifold $(Z, g)$ with cylindrical end consists of the following data:
\begin{enumerate}
\item[(i)] a compact $4$-manifold $M \subset Z$ whose boundary consists of a closed $3$-manifold $Y$;
\item[(ii)] a cylindrical end $[0, \infty) \times Y$ attached to the boundary $\partial M$ so that 
\[
Z = M \cup [0, \infty) \times Y;
\]
\item[(iii)] a Riemannian metric $g$ whose restriction to a collar neighborhood of the end takes the form 
\[
g|_{(-1, \infty) \times Y}  = dt^2 + h, 
\]
where $h$ is a metric on $Y$, and $t$ is the coordinate function on $(-1, \infty)$.
\end{enumerate} 
\end{dfn}

Let $(Z, g)$ be a smooth $4$-manifold with cylindrical end as above. We denote by $E \to Z$ and $E’ \to Y$ the trivial $\C^2$-bundles over $Z$ and $Y$ respectively. We also identify $E|_{[0, \infty) \times Y} = \pi^* E'$, where $\pi: [0, \infty) \times Y \to Y$ is the projection map. 

Fix an integer $k \geq 3$. We write $\A_{k, loc}$ for the space of $L^2_{k, loc}$ $SU(2)$-connections on $E$. The gauge group $\G_{k+1, loc}$ consists of $L^2_{k+1, loc}$ automorphisms of the associated principle bundle $P$, which is identified with $L^2_{k+1, loc}(Z, SU(2))$. The gauge action is given by
\begin{equation}\label{e2.1}
\begin{split}
\G_{k+1, loc} \times \A_{k, loc} & \longrightarrow A_{k, loc} \\
(g, A) & \longmapsto g\cdot A := A- (d_Ag)g^{-1}
\end{split}
\end{equation}
We write $\Stab(A) \subset \G$ for the isotropy group of $A$ under gauge transformations. As pointed out in \cite[Lemma 4.2.8]{DK90}, $\Stab(A)$ is the centralizer of the holonomy group of $A$ in $SU(2)$. We say $A$ is irreducible if $\Stab(A)=\Z/2$, and reducible if $\Stab(A) \supset U(1)$. In particular when $\Stab(A)=SU(2)$, we say $A$ is central. The configuration space $\A_{k, loc}$ decomposes into the irreducible and reducible parts:
\[
\A_{k, loc}=\A^*_{k, loc} \cup \A^{\Red}_{k, loc}. 
\]
Given a connection $A$ with $F_A \in L^2(Z, \Lambda^2\otimes \SU(2))$, we denote its Chern--Weil integral by
\begin{equation}
\kappa(A):={1 \over {8\pi^2}} \int_Z \tr(F_A \wedge F_A),
\end{equation}
and its energy by 
\begin{equation}
\mathcal{E}(A):=\int_Z |F_A|^2 d\vol.
\end{equation}
Usually $\kappa(A)$ is referred to as the instanton charge when $A$ is anti-self-dual. We write $F_A^+$ for the projection of the curvature $2$-form $F_A$ onto the $+1$-eigenspace of the Hodge star. 

\begin{dfn}
The moduli space of finite-energy  charge-zero instantons is 
\[
\M(Z):= \{A \in \A_{k, loc}: F^+_A=0, \mathcal{E}(A) < \infty, \kappa(A)=0 \} / \G_{k+1, loc}.
\]
\end{dfn}

In the definition, we have omitted $k$ in the notation $\M(Z)$. This is because each class $[A] \in \M(Z)$ admits smooth representatives \cite{DK90} so that the moduli space $\M(Z)$ is actually independent of $k$. 

\begin{lem}\label{l2.3}
Each connection $[A] \in \M(Z)$ is flat.
\end{lem}

\begin{proof}
This is more or less a well-known fact that charge-zero instantons are flat. The reason lies on the observation that
\[
|F_A|^2 = |F_A^+|^2 + |F_A^-|^2 \quad \text{ and } \quad \tr(F_A \wedge F_A) = |F_A^-|^2 - |F_A^+|^2.
\]
So when $F^+_A = 0$, $\kappa(A) = 0$ implies that $F_A^- = 0$ as well. 
\end{proof}

Note that the gauge action $\G_{k+1, loc}$ on $\A$ is not free. To get rid of this issue, one can consider the based moduli space defined as follows. Fix a basepoint $z_0=(0, y_0) \in \{ 0 \} \times Y \subset Z$. Then the gauge group $\G_{k+1, loc}$ acts freely on the product $\A_{k, loc} \times E_{z_0}$, where $E_{z_0}$ is the fiber of $E$ over $z_0$. We then define the based moduli space to be 
\begin{equation}
\begin{split}
\tilde{\M}(Z):= & \bigl\{ (A, v) \in \A_{k, loc} \times E_{z_0} : \\
& F^+_A=0, \mathcal{E}(A) < \infty, \kappa(A)=0  \bigr\} \Bign/ \G_{k+1, loc}.
\end{split}
\end{equation}

\subsection{Moduli Space of Flat connections over $3$-Manifolds}

Let $l=k-1/2 \geq 5/2$. Over the trivial $\C^2$-bundle $E' \to Y$, one can consider the space of $SU(2)$-connections $\A_l(Y)$ of class $L^2_l$, and the gauge group $\G_{l+1}(Y)$ consisting of $L^2_{l+1}$ maps $u: Y \to SU(2)$. The moduli space over $Y$ consists of equivalence classes of flat connections on $E'$: 
\begin{equation}
\M(Y):= \{ B \in \A_{l}(Y): F_B=0 \}\Bign/ \G_{l+1}(Y).
\end{equation}
Usually we write $B$ for a general connection on $E'$, and $\Gamma$ for a general flat connection. Via the holonomy morphism the moduli space $\M(Y)$ is identified with the character variety 
\begin{equation}
\chi(Y):=\Hom(\pi_1(Y), SU(2)) \bign/\Ad.
\end{equation}
Using a basepoint $y_0 \in Y$, one can also consider the based moduli space
\[
\tilde{\M}(Y):=\{ (B, v) \in \A_l(Y) \times E'_{y_0} : F_B=0\} / \mathcal{G}_{l+1}(Y)
\]
Via the holonomy morphism the based moduli space is identified with the representation variety 
\begin{equation}
\mathcal{R}(Y):=\Hom(\pi_1(Y), SU(2)). 
\end{equation}

Let us write $\Gamma_0$ for the product connection on $E’$. Over $\A_l(Y)$ the Chern--Simons functional \cite{CS74} is an analytic function defined by 
\begin{equation}
\cs(B)=- \int_Y \tr\left(\frac{1}{2} b \wedge db + { 1 \over 3} b \wedge b \wedge b\right),
\end{equation}
where $b=B - \Gamma_0 \in L^2_l(Y, T^*Y \otimes \SU(2))$. With respect to the $L^2$-inner product on $1$-forms, the gradient and Hessian of the Chern--Simons functional are computed as (cf. \cite{D02})
\begin{equation}\label{e2.9}
\grad \cs(B)  = *F_B \text{ and } \Hess \cs|_B (b) =* d_B b. 
\end{equation}
Thus the critical points of the Chern--Simons functional consist of flat connections on $Y$. Moreover the restricted Hessian $*d_B|_{\ker d_B^*}$ has real, discrete, and unbounded spectrum (cf. \cite[Lemma 2.1.1]{MMR}). Given a gauge transformation $u \in \G_{l+1}(Y)$ and $B \in \A_l(Y)$, we have (cf. \cite{D02})
\begin{equation}
\cs(u\cdot B) - \cs (B)  = -4\pi^2 \deg u.
\end{equation}
Thus the Chern--Simons functional descends to an $S^1$-valued function on the quotient space $\B(Y)$. 

Over a cylinder $[t_0, t_1] \times Y = I \times Y$, the ASD equation can be written as the equation of downward gradient flowlines in the following sense. Let $A$ be an $SU(2)$-connection over $I \times Y$ in temporal gauge, i.e. $A = B(t) + dt$ with $B(t)$ a path of connections on $Y$. Then the equation $F^+_A = 0$ is equivalent to 
\begin{equation}\label{e2.11}
\dot{B}(t) = - *F_{B(t)} = - \grad \cs(B(t)). 
\end{equation}
More generally, a connection $A$ over $I \times Y$ takes the form $A = B(t) + \beta(t)dt$ where $\beta(t)$ is a path of $\SU(2)$-valued $0$-forms on $Y$. Then the ASD equation takes the form (cf. \cite{MMR}):
\[
\dot{B}(t) = -*F_{B(t)} + d_{B(t)} \beta(t). 
\]
In either case, the energy of $A$ is related to the drop of the Chern-Simons functional (cf. \cite{D02})
\begin{equation}\label{e2.12}
\cs(B(t_0)) - \cs(B(t_1)) = -\frac{1}{2}\int_{I \times Y} \tr(F_A \wedge F_A) = \frac{1}{2}\|F_A\|^2_{L^2(I \times Y)}.
\end{equation}

\section{Holonomy Perturbations}\label{s3}
To get regularity for moduli spaces and transversality for asymptotic maps, we need to introduce perturbations. Since the manifolds we are interested in have $b_+=0$, we cannot use merely metric perturbations as in \cite{FU, MMR}. Following the lines in \cite{D87, K04}, we adopt holonomy perturbations instead. Since we are dealing with manifolds with cylindrical end, the holonomy perturbation will be a mixture of two kinds. The first arises from perturbations on a compact region, while over the end $[0, \infty) \times Y$ it comes from the gradient of some ‘cylinder functions’ over $Y$ with an exponentially decay assumption imposed. 

\subsection{Perturbations over a Compact Region}

Let $Z = M \cup [0, \infty) \times Y$ be a $4$-manifold with cylindrical end as before. Over $Z$, we have the trivialized $\C^2$-bundle $E$. We start with an embedded open four ball $N \subset Z$ and a smooth map $q: S^1 \times N \to Z$ satisfying
\begin{enumerate}
\item[\upshape (i)] $q$ is a submersion; 
\item[\upshape (ii)] $q(1, x)=x$ for any $x \in N$. 
\end{enumerate}
To each $q$, we assign $\Omega_q \subset Z$ a compact domain that contains the image of $q$. For each $x \in N$, the assignment $q(-, x)$ gives us a loop
\[
q_x: S^1 \longrightarrow Z
\]
based at $x$. We denote by $\hol_{q_x} (A) \in \SU(2)$ the traceless part of the holonomy $\Hol_{q_x}(A)$ of $A$ around $q_x$.  Let $\omega \in \Omega^+(N)$ be a self-dual $2$-form supported on $N$. We form an $\SU(2)$-valued $2$-form 
\begin{equation}
V_{q, \omega}(A):=\omega \otimes \hol_q(A) \in \Omega^+(Z; \SU(2))
\end{equation}
supported on $N \subset Z$. Let us write $A_0$ for the product connection over $E$. We have the following estimate for $V_{q, \omega}$. 

\begin{lem}\label{l3.1}
Let $q: S^1 \times N \to Z$ be a submersion and $\omega \in \Omega^+(N)$ a self-dual $2$-form as above. Then for any $L^2_{k, loc}$-connection $A$ over $Z$, we have 
\[
\|V_{q, \omega}(A)\|_{C^0} \leq K_0 \|\omega\|_{C^k},
\]
where $K_0$ is a constant depending only on $q$. Moreover, 
\[
\|V_{q, \omega}(A)\|_{L^2_j(Z)} \leq K_j \|\omega\|_{C^k} \left( 1 + \|a\|^j_{L^2_j(\Omega_q)}\right), \quad 1 \leq j \leq k
\]
where $a = A - A_0 \in L^2_{k, loc}(Z, T^*Z \otimes \mathfrak{su}(2))$, and $K_j$ is a constant depending only on $j$ and $q$. 
\end{lem}

\begin{proof}
Since $q$ is a submersion, we know that 
\[
\|q^*(A) - q^*(A_0)\|_{L^2_j(S^1 \times N)} \leq c. \|A - A_0\|_{L^2_j(\Omega_q)}, \quad 0 \leq j \leq k.
\]
Let’s write $q^*(A) - q^*(A_0) = B(s) + \beta(s)ds$, where $s$ is the coordinate on $S^1$, $B(s)$ a family of $\SU(2)$-valued $1$-forms on $N$, and $\beta(s)$ a family of $\SU(2)$-valued $0$-forms on $N$. Then the holonomy around the loop $q_x$ is given by 
\[
\Hol_{q_x}(A) = \exp(-\int_0^1 \beta(s) ds) \in SU(2).
\]
Taking derivatives of $\Hol_{q_x}(A)$ along the $x$-variable results in 
\[
\nabla^j \Hol_{q_x}(A) = \Hol_{q_x}(A) \cdot \sum_{r_1a_1+...+r_ka_k = j} c_{\pmb{r}, \pmb{a}} \prod_{l=1}^k \left( \int_0^1 \nabla^{r_l} \beta(s) ds\right)^{a_l},
\]
where $c_{\pmb{r}, \pmb{a}}$ is some constant depending only on the tuples $\pmb{r}=(r_1, ..., r_k)$ and $\pmb{a} = (a_1, ..., a_k)$. Thus we conclude
\[
\| \nabla^j \Hol_{q_x}(A)\|_{L^2(N)} \leq c. \|\Hol_{q_x}(A)\| \left(1+ \int_0^1\|\beta(s)\|_{L^2_j(N)}ds \right)^j,
\]
where the constant only depends on $j$. Note that 
\[
 \int_0^1\|\beta(s)\|_{L^2_j(N)}ds  \leq c. \|q^*(A) - q^*(A_0)\|_{L^2_j(S^1 \times N)}  \leq c. \|a\|_{L^2_j(\Omega_q)}. 
\]
The statement now follows immediately from applying Leibniz’s rule to $\omega  \otimes \Hol_q(A)$ and the fact that the $L^2$-norm of $\Hol_q(A)$ is uniformly bounded. 
\end{proof}

Therefore, the holonomy perturbation defines a continuous $\G_{k+1, loc}$-equivariant map
\[
V_{q, \omega}: \A_{k, loc} \longrightarrow L^2_{k, loc}(Z, \Lambda^+ \otimes \SU(2)). 
\]
It turns out that this map is also smooth, which follows from the following estimates derived in \cite{K04}.

\begin{prop}[{\cite[Proposition 3.1]{K04}}]\label{p3.2} Let $q: S^1 \times N \to Z$ be a submersion and $\omega \in \Omega^+(N)$ a self-dual $2$-form as above. Then for any $L^2_{k, loc}$-connection $A=A_0+a$ over $Z$ and $0 \leq j \leq k$, we have
\[
\|D^n V_{q, \omega}|_A (a_1, ..., a_n)\|_{L^2_j(Z)} \leq K_{n, j} \|\omega\|_{C^j} \left( 1 + \|a\|^j_{L^2_j(\Omega_q)}\right) \prod_{i=1}^n \|a_i\|_{L^2_j(\Omega_q)},
\]
where $D^n V_{q, \omega}$ is the $n$-th differential of $V_{q, \omega}$, and $K_{n, j}$ is a constant depending only on $q$. 
\end{prop}

\begin{proof}
In \cite{K04}, the original statement didn’t include the $L^2_j$-norm of $a$. However, the dependence was pointed out in the paragraph below \cite[Proposition 3.5]{KM11}. Following the argument of Kronheimer \cite{K04}, we explain why there should be dependence on the connection form $a = A - A_0$ when higher derivatives are involved. 

Let’s consider a Hilbert bundle $\mathcal{H} \to N$ whose fiber consists of square-integrable $\SU(2)$-valued $1$-forms over $S^1 \times \{x\}$, i.e.
\[
\mathcal{H}_x := L^2(S^1 \times \{x\}, T^*S^1 \otimes \SU(2)).
\]
Then taking traceless holonomy gives rise to a fiber-preserving map
\[
\begin{split}
\hol: \mathcal{H} &\longrightarrow \underline{\SU}(2)\\
\alpha_x & \longmapsto \hol_{S^1 \times \{x\}}(d + \alpha_x),
\end{split}
\]
where $\underline{\SU}(2)$ is the trivial bundle over $N$ with fiber $\SU(2)$, and $\hol$ means the holonomy of $d+\alpha_x$ around $S^1 \times \{x\}$ followed by a projection onto its traceless part. Due to \cite[Lemma 3.2]{K04}, this map extends to a smooth map 
\[
\hol: L^2_k(N, \mathcal{H}) \longrightarrow L^2_k(N, \underline{\SU}(2)). 
\]
On the other hand, we have the restriction map:
\[
\begin{split}
r: \A_k(S^1 \times N) & \longrightarrow L^2_k(N, \mathcal{H}) \\
A & \longmapsto (\alpha:  x  \mapsto (A-A_0)|_{S^1 \times \{x\}}).
\end{split}
\]
Such a restriction makes sense, because $k \geq 2$. From its construction, we also know
\[
\|r(A)\|_{L^2_j}  \leq \|A-A_0\|_{L^2_j},
\]
since all higher derivatives along the $S^1$-direction are not counted in the right-hand side. Now the holonomy perturbation map can be written as $V_{q, \omega}(A) = \omega \otimes \hol \circ r(A)$, whose differential takes the form 
\[
DV_{q, \omega}|_A (a_1) = \omega \otimes D\hol|_{r(A)} \circ Dr|_A(a_1). 
\]
\cite[Lemma 3.2]{K04} tells us that $|D\hol|_{r(A)}(\alpha)| \leq c_1 \|\alpha\|_{L^2}$ fiberwise, where $c_1$ is independent of $A$. Thus we conclude that 
\[
\|DV_{q, \omega}|_A(a_1)\|_{L^2} \leq c. \|\omega\|_{C^0} \|a_1\|_{L^2},
\]
where the constant is independent of $A$. Similar $L^2$-estimate holds for higher differentials of $V_{q, \omega}$, which proves the claim when $j=0$. 

Now let’s first consider the case when $j=1$. We compute that
\[
\nabla DV_{q, \omega}|_A(a_1)  = \nabla \omega \otimes D\hol|_{r(A)} \circ Dr|_A(a_1) 
+  \omega \otimes \nabla \left(D\hol|_{r(A)} \circ Dr|_A(a_1) \right). 
\]
The first term on the right-hand side still gets bounded by $\|\omega\|_{C^1} \|a_1\|_{L^2}$ due to \cite[Lemma 3.2]{K04}. However, the second term will involve $\|A-A_0\|_{L^2_1(N)}$ just as the proof of \autoref{l3.1}. More concretely, let’s write $A-A_0 = B(s) + \beta(s)ds$ and $a_1 = B_1(s) + \beta_1(s)ds$ as before. Then 
\[
\begin{split}
D\hol|_{r(A)} \circ Dr|_A(a_1) & = D\hol|_{r(A)}(\beta_1(s)ds) \\
& = -\left(\int_0^1 \beta_1(s) ds\right) \cdot \exp^0 \left(- \int_0^1 \beta(s) ds\right),
\end{split}
\]
where $\exp^0$ means the composition of the exponential map with the projection onto its traceless part. Thus when we take the covariant derivative $\nabla$ along the base direction, the $L^2_1$-norm of $A-A_0$ will appear due to the Leibniz’s rule, which then contributes to the $L^2_1$-norm of $a$ in our formula. Such a computation can be made for $\nabla^j$ for $j > 1$ in the same manner. This proves our statement. 
\end{proof}

Now we take a countable family of embedded open balls $\{N_{\alpha}\}_{\alpha \in \N}$ in a way that 
\[
M \subset \bigcup_{\alpha} N_{\alpha} \subset M \cup [0, 1] \times Y,
\]
together with submersions $\{q_{\alpha}\}_{\alpha \in \N}$ satisfying the condition that for each $x \in M$, the countable family of maps 
\[
\{q_{\alpha}(-, x) : \alpha \in \N, x \in N’_{\alpha} \}
\]
is $C^1$-dense in the space of smooth loops in $Z$ based at $x$, where $N’_{\alpha}$ is the ball obtained by shrinking $N_{\alpha}$ to half of its radius. We fix 
\[
C_{\alpha}:=\sup \{K_{n, \alpha} : 0 \leq n \leq \alpha \},
\]
where $K_{n, \alpha}$ is the constant for $q_{\alpha}$ in \autoref{p3.2} with $j=k$, and a family of cut-off functions $\beta_{\alpha}: Z \to \R$ with $\supp \beta_{\alpha} \subset N_{\alpha}$ and $\beta_{\alpha}|_{N’_{\alpha}} \equiv 1$. 

\begin{dfn}\label{d3.3}
We denote by $\mathcal{P}_1$ the space of all families $\pmb{\omega}$ so that $\sum_{\alpha} C_{\alpha} \|\omega_{\alpha}\|_{C^k}$ is convergent. For each $\pmb{\omega} \in \mathcal{P}_1$, we define the holonomy perturbation 
\[
\sigma_{\pmb{\omega}}= \sum_{\alpha} \beta_{\alpha} V_{q_{\alpha}, \omega_{\alpha}}. 
\]
\end{dfn}

It’s clear that $\mathcal{P}_1$ is a Banach space with respect to the norm $\|\pmb{\omega}\| := \sum_{\alpha} C_{\alpha} \|\omega\|_{C^k}$. Our convergence requirement implies that $\sum_{\alpha} \beta_{\alpha} V_{q_{\alpha}, \omega_{\alpha}}$ converges uniformly in the $C^n$-topology of any compact domain of $M \cup [0, 1] \times Y \subset Z$ for each fixed $n$. Then it follows from \autoref{p3.2} that $\sigma_{\pmb{\omega}}$ defines a smooth $\G_{k+1, loc}$-equivariant map
\[
\sigma_{\pmb{\omega}}: \A_{k, loc} \longrightarrow L^2_{k, loc}(Z, \Lambda^+ \otimes \SU(2)). 
\] 
Due to our choices of submersions $q_{\alpha}$, we know $\supp \sigma_{\pmb{\omega}}(A) \subset M \cup [0, 1] \times Y$. 

\subsection{Perturbations over Cylindrical Ends}

To perturb the ASD equation over the end $[0, \infty) \times Y$, we need to recall the notion of ‘cylinder function’ on $\A_l(Y)$. Such construction appeared originally in the work of Floer \cite{F88}. We mainly follow the exposition in \cite[Section 3.2]{KM11}. 

Suppose we have an $r$-tuple $\pmb{\rho} = (\rho_1, ..., \rho_r)$ of smooth immersions $\rho_i: S^1 \times D^2 \to Y$ which coincide over $[-\epsilon, \epsilon] \times D^2$ for some small constant $\epsilon$ independent of $i$, i.e. 
\begin{equation}
\rho_i(s, z) = \rho_j(s, z), \quad \text{ for all } s \in [-\epsilon, \epsilon], \; z \in D^2,
\end{equation}
where $s$ is the coordinate of $S^1$ parametrized by $e^{i2\pi s}$. Similarly we write $\rho_{i, z}: S^1 \to Y$ for the loop given by $\rho_i(-, z)$. Let $\eta: SU(2)^r \to \R$ be a smooth function which is invariant under the diagonally $SU(2)$-adjoint action on $SU(2)^r$. We also fix a non-negative $2$-form $\nu \in \Omega^2(D^2)$ supported near the center of $D^2$ so that $\int_{D^2} \nu = 1$. 

\begin{dfn}
Given $(\pmb{\rho}, \eta, \nu)$ as above, we define
\[
\begin{split}
\tau_{\pmb{\rho}, \eta}: \A_l(Y) & \longrightarrow \R \\
B & \longmapsto \int_{D^2} \eta\left(\Hol_{\rho_{1, z}}(B), ..., \Hol_{\rho_{r, z}}(B)\right) \nu
\end{split}
\]
as the associated cylinder function over $\A_l(Y)$. 
\end{dfn}

To relate the cylinder functions to the ASD equation, we consider its gradient with respect to the $L^2$-inner product on $L^2_l(Y, T^*Y \otimes \SU(2))$. The computation was carefully explained in \cite{D02, KM11}. We denote by $\partial_i \eta$ the $i$-th partial derivative of $\eta$. After trivializing the tangent bundle of $SU(2)$ via left translation and identifying the dual $\SU(2)^*$ with $\SU(2)$ via the Killing form, one can regard this map as
\[
\partial_i \eta: SU(2)^r \longrightarrow \SU(2).
\]
We now extend the $\SU(2)$-valued functions $\partial_i \eta(\Hol_{\rho_z}B)$ over $D^2$ to a function $H_i(B): S^1 \times D^2 \to \SU(2)$ via parallel transport under $B$. As mentioned in \cite{KM11}, the functions $H_i(B)$’s do not have to be continuous. We also extend the $2$-form $\nu$ to an $S^1$-invariant $2$-form on $S^1 \times D^2$, still denoted by $\nu$. The the formal gradient of the cylinder function $\tau$ can be computed as (cf. \cite[Page 876]{KM11})
\begin{equation}\label{e3.3}
\grad \tau(B) = * \left(\sum_{i=1}^r \rho_{i,*} (H_i(B) \otimes \nu) \right) \in \Omega^1(Y, \SU(2)). 
\end{equation}
Due to the $SU(2)$-invariance of $\eta$, the sum in (\ref{e3.3}) defines at least a continuous $\SU(2)$-valued $1$-form. It turns out that $\grad \tau(B) \in L^2_l(Y, T^*Y \otimes \SU(2))$. We record the key estimates below, which can be derived in a similar fashion as \autoref{l3.1} and \autoref{p3.2}. 

\begin{lem}[{\cite[Proposition 3.5]{KM11}}]\label{l3.5}
Let $\tau$ be a cylinder function as above, and $b = B-\Gamma_0 \in L^2_l(Y, T^*Y\otimes \SU(2))$ the connection form of a general connection $B$. Then we have the following estimates. 
\begin{enumerate}
\item There exists a constant $K’_0$ so that 
\[
\|\grad \tau(B)\|_{C^0} \leq K’_0. 
\]
\item For each $1 \leq j \leq l$, there exists a constant $K’_j$ so that 
\[
\|\grad \tau(B)\|_{L^2_j} \leq K’_j\left(1+ \|b\|^j_{L^2_j} \right). 
\]
\item For each positive integer $n \geq 1$ and $0 \leq j \leq l$, there exists a constant $K’_{n,j}$ so that 
\[
\|D^n \grad\tau|_B(b_1, ..., b_n)\|_{L^2_j} \leq K’_{n, j}\left(1+\|b\|^j_{L^2_j} \right)\prod_{i=1}^n \|b_i\|_{L^2_j}.
\]
\end{enumerate}
All constants appeared above are independent of the connection $B$. 
\end{lem}

Now we fix a countable family of cylinder functions $\pmb{\tau}=\{\tau_{\alpha}\}_{\alpha \in \N}$ satisfying the following.
\begin{enumerate}
\item[\upshape (i)] For each integer $r>0$, there is a countable sub-family of $r$-tuples of smooth immersions $\pmb{\rho}_{\alpha}$ corresponding to $\tau_{\alpha}$ that are dense in the space of all $r$-tuples of smooth immersions satisfying (\ref{e3.3}) with respect to $C^1$-topology. 
\item[\upshape (ii)] For each integer $r>0$, there is a countable sub-family of smooth $SU(2)$-invariant functions $\eta_{\alpha}: SU(2)^r \to \R$ corresponding to $\tau_{\alpha}$ so that this family is dense in the space of all smooth $SU(2)$-invariant functions on $SU(2)^r$ with respect to $C^{\infty}$-topology. 
\end{enumerate}
Similarly as before, we fix $C’_{\alpha}:=\sup\{K’_{n, \alpha}: 0\leq n \leq \alpha\}$, where $K’_{n, \alpha}$ is a constant given by \autoref{l3.5} corresponding to the cylinder function $\tau_{\alpha}$. We denote by $\mathcal{P}_2$ the space of all real sequences $\pmb{\pi}=\{\pi_{\alpha}\}_{\alpha \in \N}$ so that $\sum_{\alpha} C’_{\alpha} |\pi_{\alpha}|$ is convergent, which form a Banach space with respect to the norm $\|\pmb{\pi}\| = \sum_{\alpha} C’_{\alpha} |\pi_{\alpha}|$. Finally, we fix a smooth function on $\delta: [-1, \infty) \to \R$ so that $\delta|_{[-1, -1+\epsilon)} \equiv 0$ for some small $\epsilon > 0$, and $\delta(t) = e^{-\mu t}$, $t \in [0, \infty)$, for some fixed positive number $\mu > 0$.

\begin{dfn}\label{d3.6}
Given such choices of $\pmb{\tau}$, $\pmb{\pi}$, and $\delta$ as above, we associate an $\SU(2)$-valued self-dual $2$-form to each connection $A \in \A_k(Z)$ by 
\[
\sigma_{\pmb{\pi}}(A):= \delta(t) \cdot \left(dt \wedge \sum_{\alpha} \pi_{\alpha} \grad \tau_{\alpha}(A|_{\{t\} \times Y}) \right)^+ 
\]
\end{dfn}

Due to \autoref{l3.5}, the sum $\tau_{\pmb{\pi}} := \sum_{\alpha} \pi_{\alpha} \tau_{\alpha}$ defines a smooth $SU(2)$-invariant function on $\A_l(Y)$. We can write the holonomy perturbation alternatively as $\sigma_{\pmb{\pi}} = \delta \cdot (dt \wedge \grad \tau_{\pmb{\pi}})^+$. Analogous to \autoref{l3.5}, we have the following estimates for $\sigma_{\pmb{\pi}}$. 

\begin{lem}\label{l3.7}
Let $\pmb{\pi} \in \mathcal{P}_2$, $a = A-A_0 \in L^2_{k, loc}(Z, T^*Z \otimes \SU(2))$ the connection $1$-form for a general connection $A \in \A_k(Z)$, and $I = (t_1, t_2) \subset [0, \infty)$ a sub-interval of finite length. Then we have the following estimates.
\begin{enumerate}
\item There exists a constant $K''_0$ so that 
\[
\| \sigma_{\pmb{\pi}}(A) \|_{C^0(I\times Y)} \leq e^{-\mu t_1}  K''_0 \|\pmb{\pi}\|.  
\]
\item For each $1 \leq j \leq k$, there exists a constant $K''_j$ so that 
\[
\|\sigma_{\pmb{\pi}}(A)\|_{L^2_j(I\times Y)} \leq e^{-\mu t_1} K''_j \|\pmb{\pi}\| \left(1+ \|a\|^j_{L^2_j(I \times Y)} \right). 
\]
\item For each positive integer $n \geq 1$ and $0 \leq j \leq k$, there exists a constant $K''_{n,j}$ so that 
\[
\|D^n \sigma_{\pmb{\pi}}|_A(a_1, ..., a_n)\|_{L^2_j(I\times Y)}\leq e^{-\mu t_1}K''_{n, j} \|\pmb{\pi}\| \cdot  P_j(a),
\]
where $P_j(a) = \left(1+\|a\|^j_{L^2_j(I\times Y)} \right)\prod_{i=1}^n \|a_i\|_{L^2_j(I\times Y)}$.
\end{enumerate}
All constants appeared above are independent of the connection $A$ and the length of $I$. 
\end{lem}

\begin{proof}
All the estimates follow from \cite[Proposition 3.15]{KM11} and the Leibniz rule applied to the product of $\delta(t) = e^{-\mu t}$ with $(dt \wedge \grad \tau_{\pmb{\pi}})^+$. 
\end{proof}

As a consequence of \autoref{l3.7}, the holonomy perturbations in $\mathcal{P}_2$ define $\G_{k+1, loc}$-equivariant maps 
\[
\sigma_{\pmb{\pi}} : \A_{k, loc} \longrightarrow L^2_{k, loc}(Z, \Lambda^+ \otimes \SU(2)),
\]
with $\supp \sigma_{\pmb{\pi}} \subset [0, \infty) \times Y$. 

Finally, we write $\mathcal{P}_{\mu} := \mathcal{P}_1 \oplus \mathcal{P}_2$ for the Banach space of perturbations arising as the direct sum of both types (the subscript $\mu$ here is to emphasize that the decay rate along the end is $e^{-\mu t}$). We shall write $\sigma$ for a general element in $\mathcal{P}_{\mu}$, which by definition takes the form $\sigma_{\pmb{\omega}} + \sigma_{\pmb{\pi}}$. For the ease of notations, we simply write $\|\sigma\|_{\mathcal{P}}$ for the norm $\|\pmb{\omega}\| + \|\pmb{\pi}\|$. Then each $\sigma \in \mathcal{P}_{\mu}$ defines a smooth, gauge-equivariant map from the space of connections to the space of $\SU(2)$-valued self-dual $2$-forms. 

\section{Gauge Fixing and Regularity}\label{s4}

In this section, we discuss the choice of gauge fixing and regularity of perturbed ASD connections under this gauge. With the current definition of the holonomy perturbation, the regularity involves a major difference from the metric perturbed case. Following the idea of Uhlenbeck \cite{U82}, regularity of ASD connections is improved via the bootstrapping argument on small balls where the energy is sufficiently small. However, holonomy perturbations only admit a uniform $C^0$-bound that is independent of the locally chosen gauge. So one cannot expect to gain regularity via a local argument, while gauge-fixing for $SU(2)$-connections is always local due to its non-abelian nature. Due to this reason, we have required an exponential decay on the perturbations. Nevertheless, it stills weakens some statements in \cite{MMR}. So the main purpose of this section is to justify their statements in our setting under certain modifications. For the sake of clarity and completeness in exposition, we also rewrite some parts of the proofs. We shall first introduce the Morgan--Mrowka--Ruberman gauge fixing, then derive regularity estimates for perturbed ASD connections. 

\subsection{The Morgan--Mrowka--Ruberman Gauge Fixing}
We recall that each connection $A$ on the cylinder $[0, \infty) \times Y$ takes the form 
\[
A= B(t) + \beta(t)dt, 
\]
where $B(t) = A|_{\{t\} \times Y}$ is the restriction of the connection $A$ to a cross-section $\{t\} \times Y$, and $\beta(t) \in L^2_l(Y, \SU(2))$ is a time-dependent $\SU(2)$-valued function on $Y$. Then
\begin{equation}
F^+_A=\frac{1}{2}\left(*(*F_{B}+\dot{B}-d_B\beta) + dt \wedge (*F_B+\dot{B} - d_B \beta)\right),
\end{equation}
where $*$ means the Hodge star on the $3$-manifold $Y$. Thus the ASD equation on the cylindrical end is equivalent to the equation
\begin{equation}\label{e4.2}
\dot{B}=-*F_B+d_B \beta. 
\end{equation}

If one chooses the temporal gauge, i.e. trivialize the bundle $E|_{[0, \infty) \times Y}$ along the $t$-direction via parallel transport of $A$, then we can write $A = B(t) + dt$. In this way, the ASD equation (\ref{e4.2}) agrees with the downward gradient flow equation of the Chern--Simons functional as in (\ref{e2.11}). However, temporal gauge usually does not interact well with elliptic estimates since we don’t have any control of $B(t)$. To this end, another gauge-fixing choice was introduced in \cite{MMR}.

Let $\Gamma \in \A(Y)$ be a flat connection. We write 
\[
\mathcal{K}_{\Gamma}:= \ker d_{\Gamma}^*\subset L^2_l(Y, T^*Y \otimes \SU(2)) \text{ and } \mathcal{S}_{\Gamma}:= \Gamma+ \mathcal{K}_{\Gamma} \subset \A(Y).
\]
We refer to $\mathcal{S}_{\Gamma}$ as a slice at $\Gamma$. An $\Stab(\Gamma)$-invariant open $L^2_j$-neighborhood $U_{\Gamma}$ of $\Gamma$ in $\mathcal{S}_{\Gamma}$ is called a local slice $L^2_j$-neighborhood. In most cases, we shall work with $j=l$, and will simply call $U_{\Gamma}$ a local slice neighborhood unless otherwise noted. 

We remark that the isotropy group $\Stab(\Gamma) \subset SU(2)$ acts on $\mathcal{S}_{\Gamma}$ in the following way. We identify $L^2_l(Y, T^*Y \otimes \SU(2))$ with the tangent space $T_{\Gamma} \A(Y)$ which is preserved under the action of $\Stab(\Gamma)$ given by differentiating gauge transformations. Explicitly, an element $u \in \Stab(\Gamma)$ acts on $L^2_l(Y, T^*Y \otimes \SU(2))$ by $u \cdot b = ubu^{-1}$. Since $d_{\Gamma} u = 0$, we conclude that $d^*_{\Gamma} (u\cdot b) = 0$ if $d^*_{\Gamma} b = 0$. 

\begin{dfn}[{\cite[Definition 2.4.2]{MMR}}]\label{d4.1}
Let $I \subset [0, \infty)$ be a sub-interval. We say a connection $A=B(t)+ \beta(t)dt$ on $I \times Y$ is in standard form with respect to $\Gamma$ if for all $t \in I$ one has 
\[
B(t) \in \mathcal{S}_{\Gamma} \text{ and } \beta(t) \in (\ker \Delta_{\Gamma})^{\perp},
\]
where $\Delta_{\Gamma}:=d^*_{\Gamma}d_{\Gamma}: L^2_l(Y, \SU(2)) \to L^2_{l-2}(Y, \SU(2))$ is the Laplacian twisted by $\Gamma$. 
\end{dfn}

Let us write $\cs_{\Gamma}:=\cs|_{\mathcal{S}_{\Gamma}}$ for the restriction on the Chern--Simons functional on the slice $\mathcal{S}_{\Gamma}$. If one uses the metric on $\mathcal{S}_{\Gamma}$ given by the $L^2$-inner product on $\mathfrak{su}(2)$-valued $1$-forms, the gradient of $\cs_{\Gamma}$ at $B$ is given by 
\[
\grad_{\Gamma} cs_{\Gamma}(B) = \Pi_{\Gamma} (*F_B),
\]
where $\Pi_{\Gamma}: L^2_l(Y, T^*Y \otimes \SU(2)) \to \ker d^*_{\Gamma}$ is the $L^2$-orthogonal projection. So the downward gradient flow equation of $\cs_{\Gamma}$ does not agree with the ASD equation (\ref{e4.2}) in general. To compensate for this issue, Morgan--Mrowka--Ruberman \cite[Section 2.3]{MMR} considered a different metric on $\mathcal{S}_{\Gamma}$ locally near $\Gamma$. Let $U_{\Gamma}$ be a local slice neighborhood of $\Gamma$. By choosing $U_{\Gamma}$ sufficiently small, we can guarantee that the operator $d_B: (\ker \Delta_{\Gamma})^{\perp} \to L^2_{l-1}(Y, T^*Y\otimes \SU(2))$ is injective for all $B \in U_{\Gamma}$. We denote by 
\[
\mathcal{K}’_B := \left( d_B\left((\ker \Delta_{\Gamma})^{\perp} \right) \right)^{\perp}.
\]
the complement of the image of $(\ker \Delta_{\Gamma})^{\perp}$ under $d_B$. Since $\mathcal{K}_B = \ker d^*_B= (\im d_B)^{\perp}$, we see that $\mathcal{K}_B \subset \mathcal{K}’_B$. When $\Gamma$ is irreducible, it follows from the fact $\ker \Delta_{\Gamma} = 0$ that $\mathcal{K}_B = \mathcal{K}’_B$. Denote by $\Pi^{\Gamma}_B: \mathcal{K}_{\Gamma} \to \mathcal{K}’_B$ the projection. The local metric on $U_{\Gamma}$ is defined to be 
\begin{equation}
\langle b_1, b_2 \rangle_{\Gamma} = \langle \Pi^{\Gamma}_B(b_1), \Pi^{\Gamma}_B(b_2) \rangle_{L^2},
\end{equation} 
where $b_1, b_2 \in T_B U_{\Gamma} \simeq \mathcal{K}_{\Gamma}$. We call this metric $\langle - , - \rangle_{\Gamma}$ the (local) slice metric at $\Gamma$. This local slice metric is commensurate with the usual metric on $U_{\Gamma}$ given by the $L^2$-inner product. 

\begin{lem}\label{l4.2}
One can choose $U_{\Gamma}$ sufficiently small so that 
\[
\frac{1}{2} \|b’\|^2_{L^2} \leq \|b’\|^2_{\Gamma} \leq 2 \|b’\|^2_{L^2}
\]
for all $b’ \in T_B U_{\Gamma}$ and $B \in U_{\Gamma}$. 
\end{lem}

\begin{proof}
At $\Gamma$, the slice metric agrees with the $L^2$-inner product. The result then follows immediately from the smoothness of the slice metric. Alternatively, one can think of $d_B\left( (\ker \Delta_{\Gamma})^{\perp}\right)$ as the graph of a bounded linear map $m_B: d_{\Gamma} \left( (\ker \Delta_{\Gamma})^{\perp}\right) \to \ker d^*_{\Gamma}$. The operator norm $\|m_B\|$ depends continuously on $B$ and vanishes when $B = \Gamma$. Choosing $U_{\Gamma}$ so that $\|m_B\| \leq 1$ for all $B \in U_{\Gamma}$ gives us the desired bound on the metrics. 
\end{proof}

Since the operator $d_B$ is $\Stab(\Gamma)$-invariant in the sense that 
\[
d_{u \cdot B} (u \cdot b’) = u (d_B b’ ) u^{-1}, \quad u \in \Stab(\Gamma) \text{ and } b’ \in \mathcal{K}_{\Gamma}.  
\]
We conclude that the slice metric is also $\Stab(\Gamma)$-invariant. We shall write $\grad_{\Gamma} \cs_{\Gamma}$ for the gradient of $\cs_{\Gamma}$ computed with respect to the slice metric. The strength of this metric is realized in the following lemma. 

\begin{lem}[{\cite[Lemma 2.5.1]{MMR}}]\label{l4.3}
Let $\Gamma$ be a flat connection on $E'$. Then there exists a local slice neighborhood $U_{\Gamma}$ and a unique smooth $\Stab(\Gamma)$-equivariant map $\Theta: U_{\Gamma} \to L^2_l(Y, \SU(2))$ such that for all $B \in U_{\Gamma}$ one has 
\[
*F_B - d_B (\Theta(B))  \in \mathcal{K}_{\Gamma} \text{ and } \Theta(B) \in (\ker \Delta_{\Gamma})^{\perp}.
\]
Furthermore the map $\Theta$ has the following properties. Let $B=\Gamma+b \in U_{\Gamma}$ be a connection in the slice neighborhood of $\Gamma$.
\begin{enumerate}
\item[\upshape (i)] The gradient of $\cs_{\Gamma}$ at $B$ with respect to $\langle - , - \rangle_{\Gamma}$ is 
\[
\grad_{\Gamma} \cs_{\Gamma} (B) = *F_B - d_B (\Theta(B)).
\]
\item[\upshape (ii)] One has the estimate:
\[
\|\Theta(B)\|_{L^2} \leq \frac{1}{2} \|\grad_{\Gamma} \cs_{\Gamma} (B)\|_{L^2}.
\]
\end{enumerate}
\end{lem}

\begin{proof}
The proof is given originally in \cite{MMR}. To emphasize the role of $\Stab(\Gamma)$-equivariance, we rewrite part of the proof. 

Let’s consider the smooth map 
\begin{equation}
\begin{split}
\eta: \mathcal{K}_{\Gamma} \times (\ker \Delta_{\Gamma})^{\perp} &\longrightarrow \im d^*_{\Gamma}\\
(b, \xi) & \longmapsto d^*_{\Gamma}(*F_B - d_B \xi)
\end{split}
\end{equation}
where $B=\Gamma+b$, and $\im d^*_{\Gamma} \subset L^2_{l-2}(Y, \SU(2))$. The partial derivative of $\eta$ with respect to its second factor is 
\[
D_2\eta|_{(0,0)} = -\Delta_{\Gamma}: L^2_l \cap (\ker \Delta_{\Gamma})^{\perp} \to L^2_{l-2} \cap \im d^*_{\Gamma},
\]
which is a $\Stab(\Gamma)$-equivariant isomorphism. Since $\eta|_{\mathcal{K}_{\Gamma} \times\{0\}}$ is $\Stab(\Gamma)$-equivariant, the implicit function theorem then implies that one can find a $\Stab(\Gamma)$-invariant neighborhood $U_{\Gamma}$ of $\mathcal{K}_{\Gamma}$ around the origin and a unique $\Stab(\Gamma)$-equivariant map $\Theta: U_{\Gamma} \to (\ker \Delta_{\Gamma})^{\perp}$ so that 
\[
\eta(b, \Theta(b)) = 0,
\]
which is equivalent to $*F_B - d_B(\Theta(B)) \in \mathcal{K}_{\Gamma}$ with $B = \Gamma+b$. 

The computation of the gradient $\grad_{\Gamma} \cs_{\Gamma}$ was carried out in \cite[Lemma 2.5.1]{MMR}. The point is to establish 
\[
\langle *F_B, b’ \rangle = \langle *F_B - d_B(\Theta(B)), b’ \rangle_{\Gamma}, \quad b’ \in \mathcal{K}_{\Gamma}. 
\]

Finally, we explain why the estimate in (ii) can be achieved. It was proved in \cite[Lemma 2.5.1]{MMR} that 
\begin{equation}\label{e4.5}
\|\Theta(B)\|_{L^2_{7/4}} \leq K_0 \|b\|_{L^2_{3/2}} \|\grad_{\Gamma} \cs_{\Gamma}(B)\|_{L^2},
\end{equation}
where the constant $K_0$ only depends on $U_{\Gamma}$. Actually, the dependence comes from the operator norm of the inverse of the operator
\[
d^*_{\Gamma}d_B: (\ker \Delta_{\Gamma})^{\perp} \cap L^{7/4}_2 \longrightarrow \im d^*_{\Gamma} \cap L^{7/4}, \quad B \in U_{\Gamma}. 
\]
In particular, we can first choose $U’_{\Gamma}$ sufficiently small so that $\|(d^*_{\Gamma}d_B)^{-1}\|$ on $ (\ker \Delta_{\Gamma})^{\perp} \cap L^{7/4}_2$ is uniformly bounded for all $B \in U’_{\Gamma}$. The choice of $U’_{\Gamma}$ provides us an constant $K’_0$. Then we can choose $U_{\Gamma} \subset U’_{\Gamma}$ so that $\|b\|_{L^2_{3/2}} \leq 1/2K’_0$ for all $\Gamma+b \in U_{\Gamma}$, since $l > 3/2$ from our assumption. With this choice of $U_{\Gamma}$, (\ref{e4.5}) implies the estimate in (ii). 
\end{proof}

\subsection{Regularity Estimates}

Now we add the holonomy perturbations to this gauge-fixing picture. Combined with the anti-self-duality condition, we can get various regularity estimates. 

Concerned with the cylindrical end, we may choose perturbations $\sigma = \sigma_{\pmb{\pi}} \in \mathcal{P}_2$. Given $A \in \A_{k, loc}(Z)$, we write
\begin{equation}\label{e4.6}
\sigma(A)|_{\{t \} \times Y} = *\rho_A(t) +dt \wedge \rho_A(t),
\end{equation}
where $2\rho_A(t) = e^{-\mu t} \grad \tau_{\pmb{\pi}} (A|_{\{t\} \times Y})$ is a smooth path of $\SU(2)$-valued $1$-form on $Y$. Suppose $A|_{[0, \infty) \times Y} = B(t) +\beta(t) dt $ is in standard form with respect to $\Gamma$. The perturbed ASD equation restricted to the end $[0, \infty) \times Y$ reads as 
\begin{equation}\label{e4.7}
\dot{B}(t)=-*F_{B(t)} +d_B \beta(t) + 2\rho_A(t). 
\end{equation}
The estimates in \autoref{l3.7} for $\sigma \in \mathcal{P}_2$ implies that 
\begin{equation}
\|\rho_A(t)\|_{C^0} \leq c. \|\sigma\|_{\mathcal{P}} e^{-\mu t},
\end{equation}
where the constant depends on neither $A$ nor $\sigma$. 

\begin{lem}\label{l4.4}
Let $\Gamma$ be a flat connection on $E’$, $U_{\Gamma} \subset \mathcal{S}_{\Gamma}$ a neighborhood of $\Gamma$ provided by \autoref{l4.3}. Suppose $A$ is a perturbed ASD connection on $E$ such that $A|_{I \times Y}$ is in standard form with respect to $\Gamma$ and $B(t) \in U_{\Gamma}$ for all $t \in I$. Then 
\[
\|\beta(t) - \Theta(B(t))\|_{C^1} \leq c. \|\sigma\|_{\mathcal{P}} e^{-\mu t},
\]
where the constant depends only on $\Gamma$ and $U_{\Gamma}$. 
\end{lem}

\begin{proof}
Using (i) of \autoref{l4.3} and (\ref{e4.7}), the perturbed ASD equation on $I \times Y$ reads as  
\[
\dot{B} = -\grad_{\Gamma} \cs_{\Gamma}(B) + d_B\big(\beta(t) - \Theta(B(t)) \big) + 2\rho_A(t).
\]
By construction $\beta(t) - \Theta(B(t)) \in (\ker \Delta_{\Gamma})^{\perp}$ and $\dot{B} + \grad_{\Gamma} \cs_{\Gamma}(B) \in \ker d^*_{\Gamma}$. Note that over $Y$ one has $L^4_2 \hookrightarrow C^1$ on $Y$. Since $d^*_{\Gamma}d_B: L^4_{2} \cap (\ker \Delta_{\Gamma})^{\perp} \to L^4 \cap \im d^*_{\Gamma}$ is uniformly invertible for $B$ near $\Gamma$, we have 
\[
\begin{split}
\|\beta(t) - \Theta(B(t))\|_{C^1} &\leq c. \|\beta(t) - \Theta(B(t))\|_{L^4_2} \\
& \leq c. \|d^*_{\Gamma} d_B\left( \beta(t) - \Theta(B(t))\right)\|_{L^4} \\
& \leq c. \|d^*_{\Gamma}\rho_A(t)\|_{L^4}.
\end{split}
\]
Recall that $2\rho_A(t) = e^{-\mu t} \grad \tau_{\pmb{\pi}} (B(t))$. Using the embedding $L^2_l \hookrightarrow L^4_1$ for $l \geq 5/2$,  (2) of \autoref{l3.5} implies that 
\[
\|d^*_{\Gamma} \rho_A(t)\|_{L^4}  \leq c. \|\rho_A(t)\|_{L^4_1}
 \leq c. \|\rho_A(t)\|_{L^2_l} 
 \leq c. \|\sigma\|_{\mathcal{P}} \cdot e^{-\mu t},
\]
where the constant depends only on $\Gamma$ and $U_{\Gamma}$. 
\end{proof}

To simplify notations, we write
\begin{equation}
p_A(t) := d_B \big( \beta(t) - \Theta(B(t)) \big) + 2\rho_A(t). 
\end{equation}
Then an ASD connection $A=B(t) + \beta(t) dt$ in standard form satisfies
\begin{equation}\label{pas}
\dot{B}(t)= - \grad_{\Gamma} \cs_{\Gamma}(B(t)) + p_A(t),
\end{equation}
where the extra perturbation term admits the bound 
\begin{equation}\label{e4.11}
\|p_A(t)\|_{C^0} \leq \mathfrak{c}_0 \|\sigma\|_{\mathcal{P}} e^{-\mu t}
\end{equation}
for some constant $\mathfrak{c}_0> 0$ independent of $A$ and $\sigma$. Since the $C^0$-norm of the perturbation $p_A(t)$ is gauge-invariant, $\mathfrak{c}_0$ only depends on the gauge equivalence class of the flat connection $\Gamma$ and its neighborhood $U_{\Gamma}$. Since the $SU(2)$-representation variety of $Y$ is compact, the constant $\mathfrak{c}_0$ can be chosen to be independent of $\Gamma$. Moreover, we see that the unperturbed ASD equation restricts on a cylinder as the downward gradient flow equation of the Chern--Simons functional with respect to the Morgan--Mrowka--Ruberman gauge fixing condition and the local slice metric. 

Analogous to the unperturbed case, we can use the curvature on a tube $I \times Y$ to control the curvature on a cross-section $\{t \} \times Y$. The following lemma is crucial in the proof of the existence of asymptotic value of a perturbed gradient flowline. 

\begin{lem}\label{l4.5}
For any $\sigma$-perturbed ASD connection $A$ of finite energy on $E$, one can find $T_0 > 0$ so that for all $t > T_0$ one has 
\[
\| F_{B(t)} \|_{L^2} \leq \mathfrak{c}_1 \left( \|F_A\|_{L^2(I \times Y)} + \|\sigma\|_{\mathcal{P}} e^{-\mu t}\right) 
\]
where $B(t) = A|_{\{ t\} \times Y}$, $I = (t-1, t+1)$, and $\mathfrak{c}_1$ is a positive constant independent of $A$. 
\end{lem}

\begin{proof}
We first work locally over a geodesic four ball $D$ of $I \times Y$. Uhlenbeck's gauge fixing \cite{U82} tells us that whenever $\|F_A\|_{L^2(I \times Y)}$ is sufficiently small one can find a gauge transformation $g_D$ over $D$ of class $L^2_{k+1}$ so that 
\[
d^*a_{g_D} = 0, \quad  *a_{g_D}|_{\partial D} = 0, \quad \|a_{g_D}\|_{L^2_1} \leq c. \|F_A\|_{L^2(D)},
\]
where $a_{g_D} = g_D \cdot A - A_0$ is the connection $1$-form after applying the gauge transformation to $A$. Let $I'=[t-1/2, t+1/2]$. We may use finitely many such geodesic balls to cover $I' \times Y$. Then the gauge patching argument in \cite{W04} implies that one can find an $L^2_2$ gauge transformation, still denoted by $g_D$, on $I' \times Y$ together with a flat connection $A_D$ so that the connection form $g_D \cdot A - A_D$ admits a uniform $L^2_1$ bound over $I' \times Y$:
\begin{equation}
\|g_D \cdot A - A_D \|_{L^2_1(I' \times Y)} \leq c. \|F_A\|_{L^2(I \times Y)},
\end{equation}
where the constant is independent of $A$. Since the character variety of $I' \times Y$ is compact, one can find a constant $C_D > 0$ so that $\|A' - A_0\|_{L^2_1(I' \times Y)} \leq C_D$ for all flat connections $A'$, up to gauge transformations. So composing $g_D$ possibly with a further gauge transformation corresponding to $A_D$, we have 
\begin{equation}
\| g_D \cdot A - A_0\|_{L^2_1(I' \times Y)} \leq c. \left(\|F_A\|_{L^2(I \times Y)} + 1 \right),
\end{equation}
where the constant is still independent of $A$. Moreover, one can arrange that $d^*(g_D \cdot A - A_0)|_{D'} = 0$, where $D'$ stands for the concentric ball with half the radius of $D$. We may choose $D$ in the beginning so that $D' \subset I' \times Y$. Let us write $a_D = g_D \cdot A - A_0$. Over $D'$, the perturbed ASD equation takes the form 
\begin{equation}\label{e4.14}
d^+a_D + a_D \wedge a_D = \sigma(g_D \cdot A)|_{D'}.
\end{equation}
When $\|F_A\|_{L^2(I \times Y)}$ is sufficiently small, the standard bootstrapping argument (cf. \cite[Theorem 2.3.8]{DK90}) implies that over a smaller ball $D''\Subset D'$ one has 
\[
\begin{split}
\|a_D\|_{L^2_2(D'')} & \leq c. \left(\|F_A\|^2_{L^2(D')} + \|\sigma(g_D   \cdot A)\|_{L^2_1(I' \times Y)} \right)  \\
& \leq c. \left( \|F_A\|^2_{L^2(D')} + \|\sigma\|_{\mathcal{P}}  e^{-\mu t} \|a_D\|_{L^2_1(I' \times Y)} \right) \\
& \leq c. \left(\|F_A\|_{L^2(I \times Y)} + \|\sigma\|_{\mathcal{P}} e^{-\mu t} \right). 
\end{split}
\]
where the second inequality follows from estimates in \autoref{l3.7}, and we have absorbed the higher order terms of $\|F_A\|_{L^2}$ to a linear term since It is small.  

Now let $S''_t =  D'' \cap \{t\} \times Y$. Then it follows from the equation $F_{B(t)} = db(t) + b(t) \wedge b(t)$ that 
\[
\|F_{B(t)}\|_{L^2(S''_t)} \leq c. \left( \|b(t)\|_{L^2_1(S''_t)} + \|b(t)\|^2_{L^2_1(S''_t)} \right),
\]
where we have used the Sobolev multiplication $L^2_1 \times L^2_1 \hookrightarrow L^2$. Note that the Sobolev restriction gives us 
\[
\|b(t)\|_{L^2_1(S''_t)} \leq c. \|a_D\|_{L^2_2(D'')}. 
\]
Absorbing higher order terms to linear terms, we conclude that 
\[
\|F_{B(t)}\|_{L^2(S''_t)} \leq c. \left(\|F_A\|_{L^2(I \times Y)} + \|\sigma\|_{\mathcal{P}} e^{-\mu t}\right),
\]
where the constant is independent of $A$. The statement now follows once we cover $\{t\} \times Y$ with finitely many such geodesic balls $D''$. 
\end{proof}

\begin{rem}\label{r4.6}
In \cite[Lemma 3.5.1]{MMR}, the corresponding result was stated for $\|F_{B(t)}\|_{L^2_1}$. The author thinks such an estimate cannot be achieved even in the unperturbed case. The reason lies on the fact that the $L^2_1$-norm is not gauge-invariant, while the right-hand side of the inequality in the statement is gauge-invariant. So if one applies gauge transformations with large derivatives, the $L^2_1$-norm of $F_{B(t)}$ can blow up unless $B(t)$ is flat. As a consequence, the corresponding argument in \cite[Theorem 4.3.1]{MMR}, which establishes the existence of the asymptotic map, requires some slight improvement which will be covered later in the proof of our case. 
\end{rem}

We postpone the argument for achieving Morgan--Mrowka--Ruberman gauge fixing to the next section. Here we emphasize an intermediate step. Suppose for each $B(t) = A|_{\{ t\} \times Y}$ with $t$ fixed, we can find a gauge transformation $u(t)$ over $Y$ so that $u(t) \cdot B(t) - \Gamma \in \ker d^*_{\Gamma}$ for some fixed flat connection $\Gamma$. We wish to patch these gauge transformations together to get an $L^2_{k+1}$ gauge transformation $g$ on $I \times Y$ so that $g \cdot A$ is in standard form with respect to $\Gamma$. Besides the standard patching argument (cf. \cite{DK90, W04}), we still need to deal with the regularity of the path $u(t)$ and $g \cdot A$ as discussed in \cite[Chapter 3]{MMR}. The idea is to first solve the problem with a slightly weaker regularity, then to bootstrap the perturbed ASD connection to get back the lost regularity. The rest of this section is devoted to proving the following result, which corresponds to \cite[Theorem 2.6.3]{MMR}. 

\begin{prop}\label{p4.7}
Let $\Gamma$ be a smooth flat connection on $E’$. Then one can find a sufficiently small local slice neighborhood $U_{\Gamma}$ of $\Gamma$ with the following satisfied. Suppose $A$ is a perturbed ASD $L^2_{k, loc}$-connection on $E$ whose restriction $B(t) = A|_{\{t\} \times Y}$ is $L^2_{l+1}$ gauge equivalent to a connection in $U_{\Gamma}$ for all $t \in I$. Then one can find a sufficiently large constant $T > 0$ and $L^2_{k+1, loc}$ gauge transformations $g$ on $I’ \times Y$ for any $I’ \Subset I \cap [T, \infty)$ so that $g \cdot A|_{I’ \times Y}$  is in standard form with respect to $\Gamma$. Moreover, such a gauge transformation is unique up to composing with elements in $\Stab(\Gamma)$. 
\end{prop}

To prove \autoref{p4.7}, we briefly recall the anisotropic Sobolev spaces $L^2_{m, s}(I \times Y)$ discussed in \cite{H85}. Let’s first assume $m$ and $s$ are both non-negative integers. Then $L^2_{m, s}(I \times Y)$ is defined to be the completion of $C^{\infty}_0(I \times Y)$ with respect to the norm 
\begin{equation}
\|f \|_{L^2_{m, s}}:= \sum_{i+j \leq m+s, \; i \leq m} \| \partial^{(i)}_t \nabla^{(j)} f \|_{L^2},
\end{equation}
where $\partial^{(i)}_t$ means to take the $i$-th derivative along the $t$-direction, and $\nabla^{(j)}$ means to take the $j$-th derivative on $Y$. When $n$ or $m$ fails to be an integer, one can still define $L^2_{m, s}$ using Fourier transforms as Bessel potential spaces, see \cite[Appendix B]{H85} for more details. Besides the theorems of Sobolev embedding and multiplication, which can be found in \cite[Theorem 3.1.2]{MMR}, we record the restriction and extension theorem for later use.

\begin{thm}[{\cite[Theorem B1.11]{H85}}]\label{t4.8}
Suppose $m  > 1/2$. Then the restriction gives rise to a continuous map 
\[
r: L^2_{m, s}(I \times Y) \longrightarrow L^2_{m+s-1/2}( \{t\} \times Y),
\]
and there is a continuous extension map
\[
e: L^2_s(\{t\} \times Y) \longrightarrow L^2_{m, s-m+1/2}(I \times Y)
\]
so that $r \circ e = \id$. 
\end{thm} 

The following lemma is useful in patching gauge transformations together. Essentially, it’s part of the slice theorem stating that a neighborhood in $\ker d^*_{\Gamma} / \Stab(\Gamma)$ provides a local chart of $[\Gamma]$ in the quotient $\A/\G$. 

\begin{lem}\label{l4.9}
Let $\Gamma$ be a flat connection on $E’$. Then one can find a sufficiently small local slice neighborhood $U_{\Gamma}$ of $\Gamma$ so that for any connection $B \in U_{\Gamma}$ and $u \in \G_{l+1}$
\[
u \cdot B \in U_{\Gamma} \Longleftrightarrow u \in \Stab(\Gamma).
\]
\end{lem}

\begin{proof}
Let’s write $B = \Gamma + b$ and $u\cdot B = \Gamma + b_u$ with $b, b_u \in L^2_l(Y, T^*T \otimes \SU(2))$. Then $b_u = b - (d_B u)u^{-1}$ gives us 
\begin{equation}\label{e2.16}
d_{\Gamma} u = ub - b_u u
\end{equation}
Note that $u \in \Stab(\Gamma)$ is equivalent to $d_{\Gamma} u = 0$. So it’s clear that $b_u \in \ker d^*_{\Gamma}$ provided $u \in \Stab(\Gamma)$. 

Now we assume $b_u \in \ker d^*_{\Gamma}$. Then (\ref{e2.16}) tells us that 
\[
d^*_{\Gamma}d_{\Gamma} u  = -*(d_{\Gamma} u\wedge *b) + b_u \wedge *d_{\Gamma} u. 
\]
Since $l \geq 5/2$, the Sobolev multiplication tells us 
\[
\| \Delta_{\Gamma} u\|_{L^2_{l-1}} \leq c. \left(\|b\|_{L^2_{l-1/2}} + \|b_u\|_{L^2_{l-1/2}} \right) \|d_{\Gamma} u\|_{L^2_l}. 
\]
On the other hand, the elliptic estimate of $\Delta_{\Gamma}$ implies that 
\[
\|d_{\Gamma} u \|_{L^2_l} \leq c. \|\Delta_{\Gamma} u \|_{L^2_{l-1}}. 
\]
So we can choose $U_{\Gamma}$ sufficiently small, which gives small bound on $\|b\|_{L^2_l}$ and $\|b_u\|_{L^2_l}$ so that $\|d_{\Gamma} u\|_{L^2_{l-1}} = 0$. Thus $d_{\Gamma} u = 0$. 
\end{proof}

We first prove a weaker version of \autoref{p4.7}. 

\begin{lem}[{\cite[Lemma 3.1.3]{MMR}}]\label{l4.10}
Let $\Gamma$ be a smooth flat connection on $E’$ and $U_{\Gamma}$ a local slice neighborhood provided by \autoref{l3.7}. Suppose $A$ is a $L^2_{1, k-1, loc}$-connection over $I \times Y$ such that $A|_{\{t\} \times Y}$ is $L^2_{l+1}$ gauge equivalent to a connection in $U_{\Gamma}$ for all $t \in I$. Then there exists a $L^2_{1, k, loc}$ gauge transformation $g$ over $I \times Y$ so that $g \cdot A$ is in standard form with respect to $\Gamma$. Moreover, such a gauge transformation is unique up to composing with elements in $\Stab(\Gamma)$. 
\end{lem}

\begin{proof}
As explained in \cite{MMR}, the proof for the existence of $g$ consists of two parts. The first part is to find a $L^2_{1,k, loc}$ gauge transformation $g_1$ over $I \times Y$ so that $g_1 \cdot A|_{\{t\} \times Y} \in U_{\Gamma}$. The second part is to find $g_2$ so that $(g_2g_1) \cdot A$ solves the equation imposed by the assumption in \autoref{d4.1}. For the sake of completeness in exposition, we illustrate more details here.  

Let’s write $A=B(t) + \beta(t)dt$ and fix $t_o \in I$ so that $B_o:= u_o \cdot B(t_o) \in U_{\Gamma}$ for some $u_o \in \G_{l+1}(Y)$. \autoref{t4.8} provides us with a $g_o \in \G_{k+1}(I_o \times Y)$ so that $g_o|_{\{t_o\} \times Y} = u_o$ for some neighborhood $I_o$ of $t_o$. Denote by $u_o(t) = g_o|_{\{t\}\times Y}$, $B_o(t) := u_o(t) \cdot B(t)$, and $b_o(t)=B_o(t) - \Gamma$. Then we get a map 
\[
\begin{split}
\varphi: I_o \times (\ker \Delta_{\Gamma})^{\perp} \cap L^2_{l} &\longrightarrow \im d^*_{\Gamma} \cap L^2_{l-2} \\
(t, \xi) & \longmapsto d^*_{\Gamma}(b_o(t) - d_{B_o(t)} \xi)
\end{split}
\]
whose partial derivative along the second factor at $(t_o, 0)$ is an isomorphism $D_2\varphi|_{(t_o, 0)} = d^*_{\Gamma}d_{B_o}$. Since $A$ is a $L^2_{1, k-1}$-connection over $I_o \times Y$, it follows from \autoref{t4.8} that its restriction $A|_{\{t\} \times Y} = B(t)$ is an $L^2_1$-path of $L^2_{l-1}$-connections. Thus $\varphi$ is an $L^2_1$-map in its first factor, and a smooth map in its second factor. The implicit function theorem then gives us a neighborhood $I’_o \subset I_o$ of $t_o$ and a $L^2_1$-path $\xi(t)$ of $L^2_{l, loc}$ $0$-forms satisfying $e^{\xi(t)} \cdot B_o(t) \in U_{\Gamma}$ for $t \in I_o$. This $L^2_1$-path $\xi(t)$ combines as a $L^2_{1, k, loc}$ $0$-form $\zeta$ over $I’_o \times Y$. Setting $g = e^{\zeta}u_o$, then $(g \cdot A)|_{\{t\} \times Y} = e^{\xi(t)} \cdot B_o(t) \in U_{\Gamma}$ as desired. 

Now we need to patch together these local gauge transformations. Suppose we have two open intervals $I_1$ and $I_2$ over which there are $L^2_1$-paths of $L^2_{l, loc}$ gauge transformations $u_1(t)$ and $u_2(t)$ respectively so that $u_i(t) \cdot B(t) \in U_{\Gamma}$ for $i=1, 2$. Fix $t_* \in I_1 \cap I_2$. \autoref{l3.7} implies that $g_*:= u_1(t_*) u^{-1}_2(t_*) \in \Stab(\Gamma)$. We can find a small neighborhood $I_* \subset I_1 \cap I_2$ of $t_*$ so that $u_1(t)u_2^{-1}(t) = e^{\xi(t)} g_*$ for some $L^2_1$-path $\xi(t)$ of $L^2_{l, loc}$ $0$-forms. We also choose a cut-off function $\psi: I_1 \cup I_2 \to \R$ so that $\psi|_{I_2 \backslash {I_1 \cap I_2}} \equiv 0$. Then we define an $L^2_1$-path 
\[ 
u(t) :=
\begin{cases}
u_1(t) & \text{ if $ t \in I_1 \backslash I_1 \cap I_2$} \\
e^{\psi(t)\xi(t)} g_*u_2(t) & \text{ if $ t \in I_1 \cap I_2$} \\
g_*u_2(t) & \text{ if $ t \in I_2 \backslash I_1\cap I_2$} \\
\end{cases}
\]
Since $g_* \in \Stab(\Gamma)$, we know that $g_*u_2(t) \cdot B(t) \in U_{\Gamma}$. Thus $u(t)$ gives rise to an $L^2_{1, k, loc}$ gauge transformation over $I_1 \cup I_2$. 

Now we can assume $A = B(t) + \beta(t)dt$ satisfies $B(t) \in U_{\Gamma}$. We look for an $L^2_1$-path of $L^2_{l+1}$-maps $u(t): Y \to \Stab(\Gamma)$ so that 
\[
\beta(t) - \left( \dot{u}(t) + [\beta(t), u(t)] \right) u^{-1}(t) \in \left( \ker \Delta_{\Gamma} \right)^{\perp},
\]
which can be achieved by setting $u(t)$ to be a solution of the ODE $\dot{u}(t) = - u(t)\beta^h(t)$, where $\beta^h(t) \in \ker \Delta_{\Gamma}$ is the orthogonal projection of $\beta(t)$ onto $\ker \Delta_{\Gamma}$. Since $\Stab(\Gamma)$ preserves $U_{\Gamma}$, such a path combines to a $L^2_{1,k,loc}$ gauge transformation $v$ so that $v \cdot A$ is in standard form with respect to $\Gamma$. 

In the end, we prove the uniqueness up to $\Stab(\Gamma)$. Suppose $g_1$ and $g_2$ are two gauge transformations satisfying all requirements. Then the transition transformation $g = g_1  g_2^{-1}$ has to take value in $\Stab(\Gamma)$ since both $(g_1 \cdot A)|_{\{t\} \times Y}$ and $(g_2 \cdot A)|_{\{t\} \times Y}$ stay in $U_{\Gamma}$. Comparing the $dt$-parts of $g_1 \cdot A$ and $g_2 \cdot A$ via $g$, one sees that $\dot{u}(t)u^{-1}(t) \perp \ker \Delta_{\Gamma}$ for $u(t) = g|_{\{t\} \times Y}$. However, $u(t) \in \Stab(\Gamma)$ implies that $\dot{u}(t)u^{-1}(t)$ lies in the Lie algebra of $\Stab(\Gamma)$ which is $\ker \Delta_{\Gamma}$. Thus $u(t)$ has to be a constant path. 
\end{proof}

Given a flat connection $\Gamma$ over $Y$, let’s write $A_{\Gamma}$ for the corresponding translation-invariant connection over $I \times Y$. For perturbed ASD connections, we have the following standard regularity estimate. 

\begin{lem}\label{l4.11}
Let $\Gamma$ be a smooth flat connection on $E’$. One can find a sufficiently small local slice neighborhood $U_{\Gamma}$ of $\Gamma$ with the following satisfied. Given a finite-energy $\sigma$-perturbed ASD connection $A$ on $E$ of class $L^2_{k, loc}$, one can find a constant $T'$ sufficiently large so that whenever the restriction $A|_{I \times Y}=B(t) + \beta(t)dt$ is in standard form with respect to $\Gamma$ and $B(t) \in U_{\Gamma}$ for $I = (T-1, T+1)$ and $T > T'$, we have 
\[
\| A - A_{\Gamma} \|_{L^2_k(I’ \times Y)} \leq \mathfrak{c}_2 \left(\|A - A_{\Gamma}\|_{L^2(I \times Y)} + \|\sigma\|_{\mathcal{P}} e^{-\mu T}  \right),
\]
where $I’ =[T-1/2, T+1/2]$, and $\mathfrak{c}_2$ is a constant depending only on $U_{\Gamma}$.
\end{lem}

\begin{proof}
The proof goes through the standard bootstrapping argument. The only subtle point is that the Morgan--Mrowka--Ruberman gauge fixing does not put $A$ into the Coulomb slice of $A_{\Gamma}$. We need to make use of the gradient flow equation and ellipticity of $d^*_{\Gamma}d_{\Gamma}$ on $(\ker \Delta_{\Gamma})^{\perp}$ to bound the error. 

Let $\eta$ be a cut-off function on $I \times Y$ so that $\eta|_{I' \times Y} \equiv 1$. We write $A - A_{\Gamma} = b(t) + \beta(t)dt = a$ over $I \times Y$. Then we have
\[
d^+_{A_{\Gamma}}  a = \sigma(A)  -  (a\wedge a)^+ \quad \text{ and } \quad d^*_{A_{\Gamma}} a = -\dot{\beta}(t). 
\]
Let us write $a_{\eta}:= \eta a$. The elliptic estimate for $d^*_{A_{\Gamma}} \oplus d^*_{A_{\Gamma}}$ implies that 
\[
\begin{split}
\|a_{\eta}\|_{L^2_k(I \times Y)} & \leq c. \left( \|d^+_{A_{\Gamma}} a_{\eta} \|_{L^2_{k-1}(I \times Y)} + \|\eta\dot{\beta}(t)\|_{L^2_{k-1}(I \times Y)} \right) \\
&  \leq c. \Big(  \| a \|_{L^2_{k-1}} + \|a_{\eta} \wedge a_{\eta}\|_{L^2_{k-1}} + \|\eta \sigma(A)\|_{L^2_{k-1}} + \|\eta\dot{\beta}(t)\|_{L^2_{k-1}} \Big).
\end{split}
\]
Note that 
\begin{equation}\label{e4.17}
\|\dot{\beta}(t)\|_{L^2_{k-1}(I \times Y)} \leq c. \int_{T-1}^{T+1} \sum_{i+j \leq k-1} \|\partial^{i+1}_t \beta(t)\|_{L^2_k(Y)} dt. 
\end{equation}
Choosing $U_{\Gamma}$ sufficiently small, one can guarantee that 
\[
d^*_{\Gamma}d_B: (\ker \Delta_{\Gamma})^{\perp} \cap L^2_j \to  (\ker \Delta_{\Gamma})^{\perp} \cap L^2_{j-2}
\]
is uniformly invertible. Since $\partial^{i+1}_t \beta(t) \perp \ker \Delta_{\Gamma}$, we conclude that 
\[
\begin{split}
\|\partial^{i+1}_t \beta(t)\|_{L^2_j} & \leq c. \|d^*_{\Gamma} d_{B(t)} \partial^{i+1}_t \beta(t)\|_{L^2_{j-2}} \\ 
& \leq c. \left( \|d^*_{\Gamma} * \partial^{i+1}_t F_{B(t)} \|_{L^2_{j-2}} + \|d^*_{\Gamma}  \partial^{i+1}_t \rho_A(t)\|_{L^2_{j-2}} \right),
\end{split}
\]
where in the second inequality we have used the gradient equation (\ref{e4.6}). Note that $F_{B(t)} = d_{\Gamma} b(t)  + b(t) \wedge b(t)$, and $\rho_A(t) = 1/2 e^{-\mu t} \grad \tau_{\pmb{\pi}}(B(t))$. \autoref{l3.5} then implies that 
\[
\|\partial^{i+1}_t \beta(t)\|_{L^2_j}  \leq c. \Big( \|\partial^{i+1}_t (b(t) \wedge b(t))\|_{L^2_{j-1}} + e^{-\mu t}  \|\sigma\|_{\mathcal{P}} \big(1 + \|b(t)\|^{j-1}_{L^2_{j-1}} \big)\sum_{l=1}^j \|\partial^l_t b(t)\|_{L^2_{j-1}}\Big). 
\] 
Replace $\dot{\beta}(t)$ by $\eta \dot{\beta}(t)$ and substituting it back to (\ref{e4.17}), we conclude that 
\[
\begin{split}
\|\eta \dot{\beta}(t)\|_{L^2_{k-1}(I \times Y)} & \leq c. \bigg(\|a_{\eta} \wedge a_{\eta}\|_{L^2_{k-1}}  + e^{-\mu T}\|\sigma\|_{\mathcal{P}} \Big( 1 + \sum_{l=1}^k \|a_{\eta}\|^l_{L^2_{k-1}} \Big) \bigg). 
\end{split}
\]
In this way, we have arrived at a recursive inequality
\begin{equation}\label{e4.18}
\|a_{\eta}\|_{L^2_k(I \times Y)} \leq c. \bigg(\|a\|_{L^2_{k-1}} +  \|a_{\eta} \wedge a_{\eta}\|_{L^2_{k-1}}  + e^{-\mu T} \|\sigma\|_{\mathcal{P}} \Big( 1 + \sum_{l=1}^k \|a_{\eta}\|^l_{L^2_{k-1}} \Big) \bigg). 
\end{equation}
Due to the Sobolev multiplication $L^4 \times L^4 \hookrightarrow L^2$, we need the initial input that $\|a\|_{L^4(I \times Y)}$ is sufficiently small to do the rearrangement on the term $a_{\eta} \wedge a_{\eta}$. Plainly we have 
\[
\|a\|_{L^4(I \times Y)} \leq c. \int_{T-1}^{T+1} \|b(t)\|_{L^4} + \|\beta(t)\|_{L^4} dt. 
\]
The above argument that uses the ellipticity of $d^*_{\Gamma}d_{B(t)}$ implies that 
\[
\|\beta(t)\|_{L^2_1} \leq c. \left(\|b(t) \wedge b(t)\|^2_{L^4} + e^{-\mu t} \|\sigma\|_{\mathcal{P}} \right).  
\]
Since $L^2_1 \hookrightarrow L^4$ over $Y$, we conclude that 
\[
\|a\|_{L^4(I \times Y)} \leq c. \int_{T-1}^{T+1} \|b(t)\|_{L^4} + \|b(t) \wedge b(t)\|^2_{L^4} + e^{-\mu t} \|\sigma\|_{\mathcal{P}} dt,
\]
which can be chosen to be as small as we desire by shrinking $U_{\Gamma}$ and enlarging $T_1$. Now It is routine to run the bootstrapping argument over (\ref{e4.18}), which gives us the claimed estimate. 
\end{proof}

\begin{proof}[Proof of \autoref{p4.7}]
Let $A$ be perturbed ASD $L^2_{k, loc}$-connection satisfying the given assumptions. Since $L^2_{k,loc} \hookrightarrow L^2_{1,k-1, loc}$, \autoref{l4.10} implies that one can find a $L^2_{1, k, loc}$ gauge transformation $v$ on $I \times Y$ so that $A_g:=g \cdot A$ is in standard form with respect to $\Gamma$. Now \autoref{l4.11} implies that $A_g \in L^2_{k, loc}$ once $T$ is chosen to be sufficiently large and $U_{\Gamma}$ sufficiently small. If we write $A = d+a$ and $A_g = d+a_g$, then the gauge transformation satisfies 
\[
dg = ga - a_g g. 
\]
Since $a, a_g \in L^2_{k, loc}$ and $g \in L^2_{1, k, loc}$, we conclude that $g \in L^2_{k+1, loc}$ as claimed. 
\end{proof}

\section{The Asymptotic Behaviors}\label{EAM}
In this section we take on the task of deducing the existence of the asymptotic map described in \autoref{asyom}. Following the idea of Morgan--Mrowka--Ruberman \cite{MMR}, we shall give a refined description of the asymptotic map using ‘center manifolds’. In the end, we will also address the compactness of the moduli space $\M_{\sigma}(Z)$. 

\subsection{The Asymptotic Map} For later use, we state the existence of the asymptotic maps for based instantons. 

\begin{thm}\label{t5.1}
Let $A$ be a finite energy perturbed ASD connection on the bundle $\pi: E \to Z$, and $z_0=(0, y_0) \in Z$ a fixed basepoint. Given $v\in E_{z_0}:=\pi^{-1}(z_0)$ we get a path $v(t) \in E'_{y_0}$ by parallel transporting $v$ via $A$ along the path $[0, \infty) \times \{y_0\}$. Let $B(t)=A(t)|_{\{t\} \times Y}$. Then the path of equivalence classes $[(B(t), v(t)]$ admits a limit $[B_o, v_o] \in \mathcal{R}(Y)$. This defines a continuous $SU(2)$-equivariant map 
\begin{equation}
\tilde{\partial}_+: \tilde{\M}_{\sigma}(Z) \longrightarrow \mathcal{R}(Y).
\end{equation}
\end{thm}

We shall first prove the existence of the asymptotic map 
\[
\partial_+: \M_{\sigma}(Z) \to \chi(Y)
\]
in the unbased moduli space, which follows the strategy in \cite{MMR}. The main point is to prove that any perturbed gradient flowline $B(t)$ has finite length. The idea of the proof goes as follows. We first show that the length of $B(t)$ is bounded in terms of the energy of the ASD connection under the extra assumption that $B(t)$ is in standard form with respect to a flat connection $\Gamma$ for all $t$. Then we argue that one can choose a small  slice neighborhood $V_{\Gamma}$ for each flat connection $\Gamma$ so that whenever $B(t)$ enters $V_{\Gamma}$ for some $t$ sufficiently large, $B(t)$ cannot exit a slightly larger slice neighborhood $U_{\Gamma}$ henceforth. In the end, we can apply Uhlenbeck’s weak compactness to deduce that when the energy of $B(t)$ is small enough, there exists a moment $t_0$ when the perturbed gradient flowline $B(t_0)$ enters the small slice neighborhood $V_{\Gamma}$ for some $\Gamma$ up to gauge transformations. Since the length of a flowline is invariant under gauge transformations, the finiteness is thus proved. 

\begin{dfn}
Let $\Gamma$ be a smooth flat connection on $E'$. Denote by $\mu_{\Gamma}$ the smallest nonzero absolute value of eigenvalues of the restricted Hessian $*d_{\Gamma}|_{\ker d^*_{\Gamma}}$.
\end{dfn}

Since $\chi(Y)$ is compact, and $\mu_{\Gamma}$ is gauge invariant, we can fix a finite number $\mu > 0$ satisfying 
\begin{equation}
\mu \geq \max \{ \mu_{\Gamma}: \Gamma \text{ is a smooth flat connection} \}.
\end{equation}
Now we start in earnest to give a length estimate for perturbed gradient flowlines.

\begin{lem}[{\cite[Proposition 4.2.1]{MMR}}]\label{l5.3}
Let $\Gamma$  be a smooth flat connection on $E'$. Then there exists a local slice neighborhood $U_{\Gamma}$ of $\Gamma$, and constant $\theta \in (0, 1/2]$ so that for any connection $B \in U_{\Gamma}$ one has 
\begin{equation}\label{e5.3}
\bigl\lvert\cs(B)-\cs(\Gamma)\bigr\rvert^{1- \theta} \leq \sqrt{2} \cdot \bigl\| \grad_{\Gamma} \cs_{\Gamma}(B) \bigr\|_{L^2}. 
\end{equation}
\end{lem}

\begin{rem}\label{r5.4}
Compared with the original statement in \cite[Proposition 4.2.1]{MMR}, we have an extra factor $\sqrt{2}$ in (\ref{e5.3}). If one replaces $\| \grad_{\Gamma} \cs_{\Gamma}(B) \|_{L^2}$ by $\| \grad_{\Gamma} \cs_{\Gamma}(B) \|_{\Gamma}$, then there is no $\sqrt{2}$ in front. This is caused by our choice of $U_{\Gamma}$ in \autoref{l4.2}. This estimate is known as the Łojasiewicz type inequality whose infinite dimensional version was originally proved by Leon Simon \cite{S83}. 
\end{rem}

\begin{lem}\label{l5.5}
Let $\Gamma$, $U_{\Gamma}$, and $\theta$ be as in \autoref{l5.3}. Let $A$ be a perturbed instanton on $E$ whose restriction to $[t_1, t_2] \times Y$ is in standard form with respect to $\Gamma$. Suppose
\begin{equation}\label{cpd}
\bigl\| p_A(t) \bigr\|_{L^2} \leq \alpha \cdot \bigl\| \grad_{\Gamma} \cs_{\Gamma}(B(t)) \bigr\|_{L^2},
\end{equation}
for some constant $\alpha \in (0, 1/2)$. Then 
\[
\int_{t_1}^{t_2} \| \dot{B}(t) \|_{L^2} dt \leq {\sqrt{2} \over {\theta^2 \sqrt{1- \alpha^2}}} \cdot \Bigl\lvert \cs(B(t_1)) - \cs(B(t_2)) \Bigr\rvert^{\theta}. 
\]
\end{lem}

\begin{proof}
The inequality (\ref{cpd}) implies that 
\[
\| p_A(t) \|_{\Gamma} \leq 2\alpha \| \grad_{\Gamma} \cs_{\Gamma}(B(t)) \|_{\Gamma},
\]
from which we get 
\begin{equation*}
\begin{split}
\| \dot{B}(t) \|^2_{\Gamma} + \| \grad_{\Gamma} \cs_{\Gamma} (B(t)) \|^2_{\Gamma} & = -2 \langle \dot{B}(t), \grad_{\Gamma} \cs_{\Gamma}(B(t)) \rangle_{\Gamma} + \| p_A(t) \|^2_{\Gamma} \\ 
& \leq  -2 \left\langle \dot{B}(t), \grad_{\Gamma} \cs_{\Gamma}(B(t)) \right\rangle_{\Gamma} \\
& + 4\alpha^2 \| \grad_{\Gamma} \cs_{\Gamma}(B(t)) \|^2_{\Gamma}. 
\end{split}
\end{equation*}
Now the inequality between the arithmetic and geometric means implies that 
\begin{equation*}
\begin{split}
2\sqrt{1- 4\alpha^2} \|\dot{B}(t)\|_{\Gamma} \|\grad_{\Gamma} \cs_{\Gamma} (B(t)) \|_{\Gamma} & \leq \| \dot{B}(t) \|_{\Gamma}^2 + (1 -4\alpha^2) \| \grad_{\Gamma} \cs_{\Gamma} (B(t)) \|_{\Gamma}^2 \\
& \leq -2 \left\langle \dot{B}(t), \grad_{\Gamma} \cs_{\Gamma}(B(t)) \right\rangle_{\Gamma} \\
& = -2 {d \over dt} \cs(B(t)). 
\end{split}
\end{equation*}
In particular we conclude that the Chern--Simons functional $\cs$ is non-increasing along a flowline solving (\ref{pas}) under our assumption. 

Let us assume that $\cs(B(t_1)) >\cs(B(t_2) > \cs(\Gamma)$. The other cases can be proved similarly. Then applying \autoref{l5.3} (also \autoref{r5.4}) we get 
\begin{equation*}
\begin{split}
-{d \over dt}  (\cs(B(t)) - \cs(\Gamma))^{\theta} & \geq \theta \sqrt{1-4\alpha^2}  |\cs(B(t))  - \cs(\Gamma)|^{\theta -1} \|\dot{B}(t)\|_{\Gamma} \cdot \|\grad_{\Gamma} \cs_{\Gamma} (B(t)) \|_{\Gamma} \\
& \geq \theta\sqrt{1-4\alpha^2} \|\dot{B}(t)\|_{\Gamma} \geq \frac{\theta \sqrt{1-4\alpha^2}}{\sqrt{2}} \|\dot{B}(t)\|_{L^2}.
\end{split}
\end{equation*}
Thus integrating both sides we get 
\begin{equation*}
\begin{split}
\int_{t_1}^{t_2} \|\dot{B}(t)\|_{L^2} dt & \leq {\sqrt{2} \over {\theta \sqrt{1-4\alpha^2}}} \cdot \Big( \big(\cs(B(t_1)) -\cs(\Gamma)\big)^{\theta} - \big(\cs(B(t_2)) -\cs(\Gamma) \big)^{\theta} \Big) \\ 
& \leq {\sqrt{2} \over {\theta^2 \sqrt{1-4\alpha^2}}} \cdot \Bigl\lvert \cs(B(t_1)) - \cs(B(t_2)) \Bigr\rvert^{\theta}. 
\end{split}
\end{equation*}
\end{proof}

When the assumption (\ref{cpd}) is violated, the Chern--Simons functional does not necessarily decrease along the flowline $B(t)$ corresponding to a perturbed ASD connection $A$. Nevertheless, one can still establish a monotonicity result after adding a term of exponential decay to the Chern--Simons functional. 

\begin{lem}\label{l4.6}
Let $\Gamma$ and $U_{\Gamma}$ be as in \autoref{l5.3}, and $A$ a perturbed ASD connection on $E$ whose restriction $A|_{I \times Y} = B(t) \times \beta(t)dt$ is in standard form with respect to $\Gamma$ and $B(t) \in U_{\Gamma}$ for all $t \in I$. Then the function 
\[
\cs_A(t):= \cs(B(t)) + {\mathfrak{c}^2_0 \|\pmb{\omega}\|^2_W \over {2\mu}} e^{-2\mu t}
\]
is non-increasing over $I$, where $\mathfrak{c}_0$ is the constant in (\ref{e4.11}). 
\end{lem}

\begin{proof}
We compute that 
\[
- {d \over {dt}} \cs_A(t)  = -\left\langle \dot{B}(t), \grad_{\Gamma} \cs_{\Gamma}(B(t)) \right\rangle_{\Gamma} + \mathfrak{c}^2_0\bigl\|\pmb{\omega}\bigr\|^2_W e^{-2\mu t}. 
\]
(\ref{e4.11}) tells us that 
\[
- {d \over {dt}} \cs_A(t) 
\geq - \left\langle \dot{B}(t), \grad_{\Gamma} \cs_{\Gamma}(B(t)) \right\rangle_{\Gamma} + \frac{1}{2}\bigl\|p_A(t)\bigr\|^2_{\Gamma}.
\]
Appealing to (\ref{pas}), the right-hand side of the above inequality is
\[
\frac{1}{2}\left(\| \dot{B}(t) \|^2_{\Gamma} + \| \grad_{\Gamma} \cs_{\Gamma} (B(t)) \|^2_{\Gamma}\right)
\]
which is non-negative. 
\end{proof}

Now we give an estimate of the length of a perturbed flowline $B(t)$, which is based on discussion with Cliff Taubes. 

\begin{prop}\label{p5.7}
Let $\Gamma$, $U_{\Gamma}$, and $\theta$ be chosen to satisfy the assumption in \autoref{l4.4} and \autoref{l5.3}. Let $A$ be a perturbed instanton on $E$ of the form in (\ref{pas}) such that $B(t) \in U_{\Gamma}$ for $t \in [0, \infty)$. Then one can find $T_1 > 0$ such that 
\begin{equation}
\int_{T_1+1}^{\infty} \|\dot{B}(t)\|_{L^2} dt \leq \mathfrak{c}_3 \left(\|F_A\|_{L^2((T_1, \infty) \times Y)} + e^{-\mu\theta T_1}\right),
\end{equation}
where $\mathfrak{c}_3$ only depends on $\Gamma$, $U_{\Gamma}$, $\|\sigma\|_{\mathcal{P}}$, $\mu$, and $\theta$.
\end{prop}

\begin{proof}
Since $\|F_A\|_{L^2(Z)}$ is finite, \autoref{l4.5} implies that one can find $T_1 > 0$ so that
\begin{equation}\label{e5.6}
\|F_{B(t)} \|_{L^2} \leq c. \left(\|F_A \|_{L^2((T_1, \infty) \times Y)} + e^{-\mu t}\right), \text{ when } t \geq T_1+1.
\end{equation}

Now consider the following two cases:
\begin{enumerate}[label=\protect\circled{\arabic*}]
\item $\| p_A(t) \|_{L^2} \leq {1 \over 4} \| \grad_{\Gamma} \cs_{\Gamma}(B(t)) \|_{L^2}$.
\item $\| p_A(t) \|_{L^2} \geq {1 \over 4} \| \grad_{\Gamma} \cs_{\Gamma}(B(t)) \|_{L^2}$.
\end{enumerate}
If $\circled{1}$ holds for all $t \in [T_1, \infty)$, then \autoref{l5.5} tells us that for any $T' > T_1$ we have 
\[
\begin{split}
\int_{T_1}^{T'} \|\dot{B}(t)\|_{L^2} dt & \leq c. |\cs(B(T_1) - \cs(B(T'))|^{\theta} \\
& \leq c. \left(\int_{T_1}^{T'} \int_Y \tr(F_A \wedge F_A) \right)^{\theta}\\
& = c. \left(\int_{T_1}^{T'} \int_Y |F_A|^2 +2\tr(F_A^+ \wedge F_A^+) \right)^{\theta} \\
& \leq c. \left(\|F_A\|^{\theta}_{L^2((T_1, T') \times Y)} + e^{-2\mu \theta T_1}\right).
\end{split}
\]
Letting $T’ \to \infty$, we conclude that 
\[
\int_{T_1+1}^{\infty} \|\dot{B}(t)\|_{L^2} dt \leq c. \left(\|F_A\|^{\theta}_{L^2((T_1, \infty) \times Y)} + e^{-2\mu \theta T_1}\right).
\]

If $\circled{2}$ holds for all $t \in [T_1, \infty)$, then the exponential decay on both $p_A(t)$ and $\grad_{\Gamma} \cs_{\Gamma}(B(t))$ implies the result. 

For the remaining case, we may assume $\circled{2}$ holds at $t=T_1$. Let $[a_0, b_0], ..., [a_n, b_n], ...$ be a sequence of intervals with integer-valued end points such that 
\begin{enumerate}[label=(\roman*)]
\item $a_0 > T_1$, and $a_i > b_{i-1}$ for $i \geq 1$. 
\item $\circled{1}$ holds for all $t \in [a_i, b_i]$, $i \geq 0$. 
\item For any $k$ with $b_i \leq k < a_{i+1}$, there exists $t_k \in [k, k+1]$ such that $\circled{2}$ holds at $t=t_k$. 
\end{enumerate}
We need to estimate the the length of $B(t)$ over $[a_i, b_i]$ and $[b_i, a_{i+1}]$ separately. \\
\noindent{\bf \em {Step 1. }} Let us first consider the case over $[a_i, b_i]$ where $\circled{1}$ always holds. We choose 
\[
t_{a_i} = \max \{ t \in [a_i - 1 , a_i] : \circled{2} \text{ holds at } t \}
\]
and 
\[
t_{b_i} = \min \{t \in [b_i, b_i+1]: \circled{2} \text{ holds at } t\}. 
\]
From \autoref{l5.3} and \autoref{l5.5} we know that 
\[
\begin{split}
\int_{t_{a_i}}^{t_{b_i}} \|\dot{B}(t) \|_{L^2} dt & \leq c. \left(| \cs(B(a_i) - \cs(\Gamma) |^{\theta} + |\cs(B(b_i)) - \cs(\Gamma)|^{\theta} \right) \\
& \leq c. \Bigl(\| \grad_{\Gamma} \cs_{\Gamma}(B(t_{a_i})) \|^{\theta \over {1-\theta}} \\
&  + \| \grad_{\Gamma} \cs_{\Gamma}(B(t_{b_i}) \|^{\theta \over {1-\theta}} \Bigr) \\
& \leq c. \left(e^{- \mu t_{a_i} \cdot {\theta \over{1- \theta}}}+ e^{- \mu t_{b_i} \cdot {\theta \over{1- \theta}}}\right) \\
& \leq c.  e^{- {\mu \theta \over {1-\theta}} \cdot  (a_i -1)}
\end{split}
\]

\noindent{\bf \em {Step 2. }} Next we consider the case over $[b_i, a_{i+1}]$. Let $k$ be an integer in $[b_i, a_{i+1})$, and $t_k \in [k, k+1]$ such that $\circled{2}$ holds at $t=t_k$. Applying Hölder's inequality we get 
\[
\begin{split}
\int_{t_k}^{k+2} \|\dot{B}(t)\|_{L^2} dt & \leq \sqrt{2}\int_{t_k}^{k+2} \|\dot{B}(t)\|_{\Gamma} dt \\
&\leq \sqrt{2 (k+2-t_k)} \left(\int_{t_k}^{k+2} \| \dot{B}(t)\|^2_{\Gamma} dt \right)^{\frac{1}{2}} \\
& \leq 2 \left( \int_{t_k}^{k+2} \langle \dot{B}(t), -\grad_{\Gamma} \cs_{\Gamma}(B(t)) + p_A(t) \rangle_{\Gamma} dt\right)^{\frac{1}{2}} \\
& \leq 2 \bigg(|\cs(B(t_k)) - \cs(B(k+2))|^{\frac{1}{2}} \\
& + \left|\int_{t_k}^{k+2} \langle -\grad_{\Gamma} \cs_{\Gamma}(B(t)) + p_A(t), p_A(t) \rangle_{\Gamma} dt \right|^{\frac{1}{2}} \bigg).
\end{split}
\]
The estimate (ii) in \autoref{l4.3} implies that for $B \in U_{\Gamma}$ 
\[
\| \grad_{\Gamma} \cs_{\Gamma}(B) \|_{L^2} \leq 2\|F_B\|_{L^2}.
\]
Continuing with the computation above, we get 
\[
\begin{split}
\int_{t_k}^{k+2} \|\dot{B}(t)\|_{L^2} dt & \leq c. \bigg(|\cs(B(t_k)) - \cs(\Gamma)|^{\frac{1}{2}} + |\cs(B(k+2)) - \cs(\Gamma)|^{\frac{1}{2}} \\
& \hspace{18mm} + e^{-\mu t_k} \left(\int_{t_k}^{k+2} \|F_{B(t)}\|^2_{L^2} dt\right)^{\frac{1}{2}} + e^{-\mu t_k} \bigg) \\
&\leq c. \Big(\|\grad_{\Gamma} \cs_{\Gamma}(B(t_k))\|^{1 \over {2- 2 \theta}}_{L^2} + e^{- \mu t_k} \|F_A\|_{L^2((T_1, \infty) \times Y)} + e^{- \mu t_k}  \Big) \\
&\leq c.\Big(e^{-{\mu \over {2- 2 \theta}} t_k } + e^{- \mu t_k} \|F_A\|_{L^2((T_1, \infty) \times Y)} + e^{-\mu t_k}\Big).
\end{split}
\] 
In the first inequality, we used \autoref{l5.3} to bound the difference $|\cs(B(t_k)) - \cs(\Gamma)|$ and \autoref{l4.6} to bound the difference $|\cs(B(t_k)) - \cs(B(k+2))|$. In the second inequality, we used (\ref{e5.6}) to bound $\|F_{B(t)}\|$ via $\|F_A\|_{L^2((T_1, \infty) \times Y)}$. 

Combining Step 1 and Step 2 we conclude that if neither of $\circled{1}$ nor $\circled{2}$ holds for all $t \in [0, \infty)$, the length of the perturbed gradient flowline after $T_1$ is bounded by
\[
\begin{split}
\int_{T_1}^{\infty} \|\dot{B}(t) \| & \leq c. \sum_{n=N}^{\infty} \bigg( e^{-{\mu \theta \over {1- \theta}} n} + e^{-{\mu \over {2- 2 \theta}} n }  + e^{-\mu n} + e^{- \mu n} \|F_A\|_{L^2((T_1, \infty) \times Y)}\bigg)\\
& \leq c. \left ( \|F_A\|_{L^2((T_1, \infty) \times Y)} + e^{- \mu \theta T_1} \right),
\end{split}
\]
where the constant only depends on $\Gamma$, $U_{\Gamma}$, $\|\sigma\|_{\mathcal{P}}$, $\mu$, and $\theta$.
\end{proof}

Now we have completed the first step in the sketched proof of length finiteness. The second step is to provide a refined choice of local slice neighborhood for $\Gamma$ where a gauge-fixing argument can be made to achieve the assumption of \autoref{p4.7}. As we mentioned in \autoref{r4.6}, the failure of getting an $L^2_1$-bound causes a flaw in the proof of \cite[Theorem 4.3.1]{MMR}. To fix their argument, we shall start by working over neighborhoods $U_{\Gamma}$ consisting of connections of weaker regularity, then apply the bootstrapping argument to compensate the loss of regularity for perturbed gradient flowlines. To this end, we first make a simple observation.

\begin{lem}\label{l5.8}
Let $\Gamma$ be a smooth flat connection on $E’$. Then one can find a sufficiently small local slice $L^2_1$-neighborhood $U’_{\Gamma}$ with the following satisfied. Given any $L^2_1$-neighborhood $V’_{\Gamma} \Subset U’_{\Gamma}$ and $\epsilon > 0$, one can find $\epsilon’ > 0$ so that for any connection $B \in V’_{\Gamma}$ satisfying $\|F_B\|_{L^2} < \epsilon’$ there exists a flat connection $\Gamma’ \in U’_{\Gamma}$ with 
\[
\|B - \Gamma’\|_{L^2_1} < \epsilon. 
\]
\end{lem}

\begin{proof}
Suppose on the contrary that for any $L^2_1$- neighborhood $U’_{\Gamma}$ one can find an $L^2_1$-neighborhood $V’_{\Gamma} \Subset U’_{\Gamma}$ and a sequence $B_n$ of connections in $V’_{\Gamma}$ satisfying $\|F_{B_n}\|_{L^2} \to 0$ as $n \to \infty$, but the $L^2_1$-limit set of $\{B_n\}$ does not contain any flat connections in $U’_{\Gamma}$. 

We write $b_n = B_n - \Gamma$. Since $\|b_n \|_{L^2_1}$ is bounded and the embedding $L^2_1 \hookrightarrow L^2$ is compact, we can take a subsequence, still denoted by $B_n$, that converges strongly in $L^2$ and weakly in $L^2_1$ to a $L^2_1$-connection $\Gamma’  \in \overline{V}’_{\Gamma} \subset U’_{\Gamma}$. Let’s write $b_n’ = B_n - \Gamma’$ and $\gamma = \Gamma’ - \Gamma$. Then we have 
\[
d_{\Gamma} b’_n = F_{B_n} - b’_n \wedge b’_n - [\gamma, b’_n]. 
\]
By choosing $U_{\Gamma}$ sufficiently small, the elliptic estimate for $d^*_{\Gamma} + d_{\Gamma}$ implies that 
\begin{equation}\label{e4.8}
\| b’_n \|_{L^2_1} \leq c. \|F_{B_n}\|_{L^2}, 
\end{equation}
where the constant only depends on $U’_{\Gamma}$. Since $\|F_{B_n}\|_{L^2}$ converges to $0$, we see that $B_n$ converges strongly in $L^2_1$ to $\Gamma’$, which also implies that $\Gamma'$ is flat. This is a contradiction. 
\end{proof}

\begin{prop}\label{p5.9}
Let $\Gamma$ be a smooth flat connection on $E’$. Then one can find a pair of local slice neighborhoods $V_{\Gamma} \Subset U_{\Gamma}$, a sufficiently small constant $\epsilon_2 > 0$, and a sufficiently large constant $T_2 > 0$ so that the following hold. Suppose $A$ is a $\sigma$-perturbed ASD $L^2_{k, loc}$-connection of finite energy over $Z$ satisfying 
\begin{enumerate}
\item[\upshape (i)] $\|\sigma\|_{\mathcal{P}} \leq 1$;
\item[\upshape (ii)] $\|F_A\|_{L^2([T_2, \infty) \times Y)} < \epsilon_2$; 
\item[\upshape (iii)] $A|_{\{T_2\} \times Y} = B(T_2)$ is $L^2_{l+1}$ gauge equivalent to a connection in $V_{\Gamma}$.
\end{enumerate}
Then one can find a $L^2_{k+1, loc}$ gauge transformation $g$ on $[T_2, \infty) \times Y$ so that $g \cdot A$ is in standard form with respect to $\Gamma$ and $(g \cdot A)|_{\{t \} \times Y} \in U_{\Gamma}$  for all $t \in [T_2, \infty)$. 
\end{prop}

\begin{proof}
The proof consists of several steps. The idea is to derive the result for local slice $L^2_1$-neighborhoods $V’_{\Gamma} \Subset U’_{\Gamma}$ using arguments in the proof of \cite[Theorem 4.3.1]{MMR}, then apply \autoref{l4.11} to improve the regularity.  

\vspace{3mm}

\noindent{\bf \em {Step 1. }} Note that the proof of \autoref{p4.7} together with its companion lemmas makes no use of higher norms $L^2_j$ with $j \geq 1$. So we may choose a local slice $L^2_1$-neighborhood of $U’_{\Gamma}$ of $\Gamma$ where the conclusion of \autoref{p4.7} still applies. Give $\epsilon > 0$, we let 
\[
\begin{split}
U’_{\Gamma}(\epsilon):= \biggl\{ B \in U’_{\Gamma} : \exists & \text{ a flat connection } \Gamma’ \in U’_{\Gamma} \\
& \text{ satisfying } \|B - \Gamma’\|_{L^2_1},  \|\Gamma’ - \Gamma\|_{L^2_1} < \epsilon \biggr\}. 
\end{split}
\]
We fix $\epsilon_a > 0$ so that $U’_{\Gamma}(\epsilon_a) \Subset U’_{\Gamma}$. We let $V’_{\Gamma} = \{ B \in U’_{\Gamma}(\epsilon_a): \|B - \Gamma\|_{L^2_1} < \epsilon_b\}$ for some $\epsilon_b > 0$ to be determined later. For now, we only require that $V’_{\Gamma} \Subset U’_{\Gamma}(\epsilon_a)$. 

\vspace{2mm}

\noindent{\bf \em {Step 2. }} Fix $\epsilon_c > 0$ and $T_c > 0$ to be determined later. Let $A$ be a $\sigma_{\pmb{\omega}}$-perturbed ASD $L^2_{k, loc}$-connection on $E$ so that 
\[
\|\sigma\|_{\mathcal{P}} \leq 1, \quad \|F_A\|_{L^2([T_c-1, \infty) \times Y)} < \epsilon_c, \quad A|_{\{T_c\} \times Y} = B(T_c) \in V’_{\Gamma}.
\]
By the argument of \autoref{l4.10}, one can apply $L^2_{k+1, loc}$ gauge transformations on $A$ so that $A|_{\{t\} \times Y} \in U’_{\Gamma}$ for $t \in (T_c-\epsilon_d, T_c+\epsilon_d)$. Iterating this procedure, we can find a maximal moment $T_m \in (T_c, \infty)$ so that $A|_{\{t\} \times Y} \in U’_{\Gamma}$ for all $t \in (T_c, T_m)$ after an $L^2_{k+1, loc}$ gauge transformation. 

\vspace{2mm}

\noindent{\bf \em {Step 3. }} We claim that one can choose $\epsilon_b$, $\epsilon_c$, and $T_c$ appropriately so that $T_m = \infty$. Suppose $T_m < \infty$. Then we can find $T_d \in (T_c, T_m)$ so that $B(T_d) \in U’_{\Gamma} \backslash \overline{U}’_{\Gamma}(\epsilon_a)$. Then \autoref{p4.7} implies that 
\begin{equation}
\begin{split}
\|B(T_d) - B(T_c)\|_{L^2} & \leq \int_{T_c}^{T_d} \|\dot{B}(t)\|_{L^2} \\
& \leq \frak{c}_3 \left(\|F_A\|_{L^2((T_c-\epsilon_d,\infty) \times Y)} + e^{-\mu \theta (T_c-\epsilon_d)}  \right). 
\end{split}
\end{equation}
Here $T_c - \epsilon_d$ plays the role of $T_1$ in \autoref{p4.7}, which will cause the dependence of $\mathfrak{c}_3$ on $\epsilon_d$. Since $\epsilon_d$ is fixed and can be made to be independent of $B(T_c)$ by choosing $V’_{\Gamma}$ sufficiently small relative to $U’_{\Gamma}$, $\mathfrak{c}_3$ will not depend on $\epsilon_b$, $\epsilon_c$, nor $T_c$. Applying \autoref{l4.5} to $(T_c-\epsilon_d, \infty) \times Y$, we get
\begin{equation}\label{e5.10}
\|F_{B(T_d)}\|_{L^2} \leq \frak{c}_1 \left(\|F_A\|_{L^2((T_c-\epsilon_d, \infty) \times Y)} + e^{-\mu T_c} \right). 
\end{equation}
\autoref{l5.8} tells us we can choose appropriate $\epsilon_c$ and $T_c$ so that there exists a flat connection $\Gamma’ \in U’_{\Gamma}$ satisfying 
\begin{equation}
\|B(T_d) - \Gamma’\|_{L^2_1} \leq \epsilon_a. 
\end{equation}
Since $B(T_c) \in V’_{\Gamma}$, we know that $\|B(T_c) - \Gamma\|_{L^2_1} \leq \epsilon_b$. Thus 
\begin{equation}
\begin{split}
\|\Gamma’ - \Gamma\|_{L^2} & \leq \|\Gamma’ - B(T_d) \|_{L^2} + \|B(T_c) - \Gamma\|_{L^2} + \|B(T_d) - B(T_c)\|_{L^2} \\
&\leq \epsilon_a + \epsilon_b + \frak{c}_3 \left( \|F_A\|_{L^2((T_c - \epsilon_d, \infty) \times Y)} + e^{-\mu\theta(T_c-\epsilon_d)} \right).
\end{split}
\end{equation}
Since $\Gamma’$ is a flat connection, the elliptic estimate of $d^*_{\Gamma} + d_{\Gamma}$ implies that 
\begin{equation}\label{e5.13} 
\begin{split}
\| \Gamma’ - \Gamma\|_{L^2_1} & \leq c. \|\Gamma’ - \Gamma\|_{L^2} \\
& \leq c. \bigg(\epsilon_a + \epsilon_b +\mathfrak{c}_3 \left( \|F_A\|_{L^2((T_c - \epsilon_d, \infty) \times Y)} + e^{-\mu\theta(T_c-\epsilon_d)} \right) \bigg).
\end{split}
\end{equation}
Thus we can choose $\epsilon_b$, $\epsilon_e$ small and $T_c$ large so that the right-hand side of (\ref{e5.13}) becomes smaller than $\epsilon_a$, which implies that $B(T_d) \in U’_{\Gamma}(\epsilon_a)$. This is a contradiction. 

\vspace{2mm}

\noindent{\bf \em {Step 4. }} Let’s take 
\[
V_{\Gamma} = \left\{ B \in \mathcal{S}_{\Gamma}: \|B - \Gamma\|_{L^2_l} < \epsilon_b \right\} \Subset U_{\Gamma} = \left\{ B \in \mathcal{S}_{\Gamma}: \|B - \Gamma\|_{L^2_l} < \epsilon_e \right\},
\]
where $\epsilon_b$ is the constant used to define $V’_{\Gamma}$ and $\epsilon_e$ is a constant to be determined later. From the discussion above, we know that given a perturbed ASD $L^2_{k, loc}$-connection $A$ with $A|_{\{T_c \} \times Y} \in V’_{\Gamma}$, one can find an $L^2_{k+1, loc}$ gauge transformation $g$ so that $(g \cdot A)|_{\{t\} \times Y} \in U’_{\Gamma}$ for all $t \in [T_c, \infty)$. Abusing notations, we simply write $A$ for the transformed connection $g \cdot A$. Let’s write $A=B(t) + \beta(t)dt$ for $t \in (T_c, \infty)$. Due to the uniform invertibility of $d^*_{\Gamma}d_B$ on $(\ker \Delta_{\Gamma})^{\perp}$ with respect to $B$ near $\Gamma$, the gradient flow equation (\ref{e4.7}) implies that 
\begin{equation}
\|\beta(t)\|_{L^2} \leq \|\beta(t)\|_{L^2_1} \leq c. \left( \|F_{B(t)}\|_{L^2} + e^{-\mu t} \right).  
\end{equation}
Thus given $T \in [T_c, \infty)$, we have
\[
\|A - A_{\Gamma}\|_{L^2((T-\epsilon_d/2, T+\epsilon_d/2) \times Y)} \leq c. \bigg( \int_{T-\epsilon_d}^{T+\epsilon_d} \|F_{B(t)}\|_{L^2} dt  + \int_{T-\epsilon_d}^{T+\epsilon_d} \|B(t) - \Gamma\|_{L^2}dt + e^{-\mu T} \bigg).
\]
By choosing $\epsilon_c$ and $U_{\Gamma}’$ smaller in the beginning, we can guarantee that the assumptions in \autoref{l4.11} are satisfied. Thus we conclude that 
\begin{equation}\label{e5.16}
\begin{split}
\|B(T) - \Gamma\|_{L^2_l} & \leq c. \|A-A_{\Gamma}\|_{L^2_k((T-\epsilon_d/2, T+\epsilon_d/2) \times Y)} \\
& \leq c. \bigg( \int_{T-\epsilon_d}^{T+\epsilon_d} \|B(t) - \Gamma\|_{L^2}dt + \int_{T-\epsilon_d}^{T+\epsilon_d} \|F_{B(t)}\|_{L^2} dt + e^{-\mu T} \bigg) \\
& \leq \mathfrak{c}_a \left( \epsilon_a + \epsilon_c + e^{-\mu T}\right).
\end{split}
\end{equation}
By setting $\epsilon_e = \mathfrak{c}_a \left( \epsilon_a + \epsilon_c + e^{-\mu T}\right) +1$, we conclude that $B(T) \in U_{\Gamma}$ for all $T \in [T_c, \infty)$. Back to the original statement, we can take $\epsilon_2 = \epsilon_c$ and $T_2 = T_c$. This completes the proof. 
\end{proof}

Now we are ready to complete the proof of \autoref{t5.1}. We first deduce the existence of the asymptotic maps for unbased moduli spaces. The proof of the based version will follow from the same argument as in \cite[Theorem 4.6.1]{MMR}. 

\begin{proof}[Proof of \autoref{t5.1}]
Let $A$ be a $\sigma$-perturbed ASD $L^2_{k, loc}$-connection on $E$ of finite energy. We write $A|_{[0, \infty) \times Y} = B(t) + \beta(t)dt$ as before. \autoref{l4.5} implies that $\|F_{B(t)}\|_{L^2} \to 0$ as $t \to \infty$. Uhlenbeck’s weak compactness (cf. \cite[Theorem A]{W04}) provides us with a sequence $T_n \to \infty$ so that $B(T_n)$ converges strongly in $L^2$ and weakly in $L^2_1$  to a flat connection $\Gamma$ after applying $L^2_2$-gauge transformations. After applying a further gauge transformation, one can guarantee that $\Gamma$ is a $C^{\infty}$ flat connection (cf. \cite{W05}). Since we start with $L^2_l$-connections $B(t)$, the gauge transformations we applied are actually in $L^2_{l+1}$ (see the proof of \autoref{p4.7}). Due to the $L^2$-convergence of $B(T_n)$, the Local Slice theorem \cite[Theorem F]{W04} implies that one can apply gauge transformations on $B(T_n)$ so that $B(T_n) \in V_{\Gamma}$ with $V_{\Gamma}$ the local slice neighborhood given by \autoref{p5.9} when $n$ is sufficiently large. Again these gauge transformations are in $L^2_{l+1}$. Thus we may fix a moment $T_o$ so that $B(T_o) \in V_{\Gamma}$ after applying an $L^2_{l+1}$ gauge transformation and $\|F_A\|_{L^2([T_o, \infty) \times Y)}$ is small enough to apply \autoref{p5.9}. Since we don’t require $\|\sigma\|_{\mathcal{P}} \leq 1$, the moment $T_o$ would depend on $\|\sigma\|_{\mathcal{P}}$ which is fixed in the beginning. Now invoking \autoref{p5.9}, we know that $A$ is in standard form with respect to  the smooth flat connection $\Gamma$ on $[T_o, \infty) \times Y$, possibly after applying a gauge transformation which, due to \autoref{p4.7}, is in class $L^2_{k+1, loc}$.

Let $U_{\Gamma}$ be the slice neighborhood given in \autoref{p5.9} so that $B(t) \in U_{\Gamma}$ when $t > T_o$. We claim that $B(t)$ converges to a flat connection $\Gamma_o \in U_{\Gamma}$ strongly in $L^2_{l+1}$. We first prove the claim for $L^2_1$-convergence. The proof of \autoref{l5.8} (cf. the derivation of (\ref{e4.8})) implies that any strongly $L^2$-convergent subsequence of $B(t)$ actually converges strongly in $L^2_1$. Suppose $\{t_n\}$ and $\{s_n\}$ are two subsequences in $[T_o, \infty)$ so that $B(t_n)$ and $B(s_n)$ converges strongly to $\Gamma_t$ and $\Gamma_s$ respectively in $L^2_1$. Then 
\[
\|\Gamma_t - \Gamma_s\|_{L^2} \leq \|\Gamma_t - B(t_n)\|_{L^2} + \|B(t_n) - B(s_n)\|_{L^2} + \|B(s_n) - \Gamma_s\|_{L^2}. 
\]
We may choose $t_n < s_n$ for all $n$ by passing to subsequences. Then \autoref{p5.9} implies that 
\[
\|B(t_n) - B(s_n)\|_{L^2} \leq \int_{t_n}^{s_n} \|\dot{B}(t)\|_{L^2} dt \to 0 \text{ as } n \to \infty. 
\]
Letting $n \to \infty$, we conclude that $\Gamma_t = \Gamma_s$. Thus the $L^2_1$ limit set of $B(t)$ consists of a single flat connection, say $\Gamma_o \in U_{\Gamma}$. Since $\Gamma$ is smooth, applying the bootstrapping argument to the equation $F_{\Gamma_o} = 0$ gives us that $\Gamma_o$ is smooth as well. Due to the strong $L^2_1$-convergence of $B(t)$, we may apply a further gauge transformation so that $A$ is in standard form with respect to $\Gamma_o$ over $[T’_o, \infty)$ for some $T’_o$ sufficiently large. Then \autoref{l4.11} can be applied to deduce the estimate (\ref{e5.16}) with respect to $\Gamma_o$, which implies that $B(t)$ converges strongly to $\Gamma_o$ in $L^2_l$. 

In \autoref{r1.8}, we mentioned that the convergence of $B(t)$ to $\Gamma_o$ is actually in $C^{\infty}$-topology. This follows from \autoref{l4.11} directly, since once $L^2_k \hookrightarrow C^0$ the bootstrapping argument using (\ref{e4.18}) would not involve a rearrangement of the quadratic term, and one can bound the $L^2_{k+1}$-norm for $k$ arbitrarily large without shrinking $U_{\Gamma_o}$ further. 

In this way, we have assigned a flat connection $[\Gamma_o] \in \chi(Y)$ to a perturbed instanton $[A] \in \M_{\sigma}(Z)$. This assignment is denoted by 
\[
\partial_+: \M_{\sigma}(Z) \longrightarrow \chi(Y),
\]
which we refer to as the asymptotic map. The continuity of $\partial_+$ follows verbatim from the argument in \cite[Page 71]{MMR}. 

Now we consider the asymptotic map for based instantons.  Let $[A, v] \in \tilde{\M}_{\sigma}(Z)$ be a based instanton. From the discussion above, we can choose a representative $(A, v)$ so that $A$ is in standard form with respect to a smooth flat connection $\Gamma$ over $[T, \infty)$ for $T$ large and $A|_{\{t\} \times Y} \to \Gamma$ strongly in $L^2_l$. Let $v(t) \in E_{(t, y_0)}$ be the path obtained by the parallel transport of $v \in E_{z_0}$ along the path $[0, \infty) \times \{y_0\}$ using the connection $A$. We can write $v(t) = \sigma(t) \cdot v$ with $\sigma: [0, \infty) \to SU(2)$ satisfying $\sigma(0) = 1$. Then $\sigma(t)$ satisfies the ODE $\dot{\sigma} + \beta_{y_0}(t) \sigma(t) = 0$, where $\beta_{y_0}(t) = \beta(t)|_{(t, y_0)}$. From this expression we see that $|\dot{\sigma}(t)| = |\beta_{y_0}(t)|$ pointwise. Since $SU(2)$ is compact, it suffices to show that the length of $\sigma(t)$ is finite to ensure the existence of its limit as $t \to \infty$. We note that 
\[
\begin{split}
\|\beta(t)\|_{C^0} & \leq c. \left( \|\beta(t) - \Theta(B(t))\|_{C^0} + \|\Theta(B(t)) \|_{L^2_{7/4}} \right) \\
& \leq c. \left(\|\sigma\|_{\mathcal{P}} e^{-\mu t} + \|b(t)\|_{L^2_{3/2}} \|\grad_{\Gamma} \cs_{\Gamma}(B(t))\|_{L^2} \right), 
\end{split}
\]
where the first inequality makes use of the Sobolev embeddings $L^2_l$, $L^2_{7/4} \hookrightarrow C^0$, and the second inequality makes use of \autoref{l4.4} and (\ref{e4.5}). From the gradient flow equation (\ref{pas}), we know that 
\[
\|\grad_{\Gamma} \cs_{\Gamma}(B(t))\|_{L^2} \leq c. \left(\|\dot{B}(t)\|_{L^2} + \|\sigma\|_{\mathcal{P}} e^{-\mu t} \right). 
\]
Since $\|b(t)\|_{L^2_{3/2}} \to 0$ as $t \to \infty$, we conclude that 
\[
\int_T^{\infty} |\dot{\sigma}(t)| dt \leq \int_T^{\infty} \|\beta(t)\|_{C^0} \leq c. \int_T^{\infty} \|\dot{B}(t)\|_{L^2} + \|\sigma\|_{\mathcal{P}} e^{-\mu t} dt,
\]
which is finite due to \autoref{p5.9}. 

Let’s denote by $\sigma_o = \lim_t \sigma(t)$ and $v_o = \sigma_o \cdot v$. Then the assignment $[\Gamma, v_o]$ to $[A, v]$ gives rise the to asymptotic map
\[
\tilde{\partial}_+: \tilde{\M}_{\sigma}(Z) \longrightarrow \mathcal{R}(Y),
\]
whose continuity follows from the same argument in \cite{MMR}. 
\end{proof}

\subsection{A Refinement of the Asymptotic Map}
From the discussion above, we know a perturbed ASD connection $A$ of finite energy on an end-cylindrical manifold is asymptotic to a limit flat connection $\Gamma$ after gauge transformations. Since the limit set $\chi(Y)$ is a finite dimensional analytic variety, a natural question to ask is whether we can use genuine gradient flowlines over a finite-dimensional object to approximate the  perturbed flowlines given by $A$. To this end, we recall the notion of ‘center manifold’ in \cite[Definition 5.1.2]{MMR}.
\begin{dfn}
Let $H=H_0 \oplus H^{\perp}_0$ be an orthogonal decomposition of a Hilbert space $H$ with $H_0$ a finite dimensional subspace. Let $U \subset H$ be a neighborhood of $0$ in $H$, and $\nu: U \to H$ a vector field over $U$. A $C^k$-center manifold for the pair $(U, \nu)$ is a submanifold $\mathcal{H} \subset H$ arising as the graph of a $C^k$-map $f: U_0 \to H^{\perp}_0$, where $U_0 \subset H_0$ is a neighborhood of $0$ in $H_0$ satisfying the following conditions:\\
(1) $\mathcal{H} \subset U$, and $T_0 \mathcal{H} = H_0$,\\
(2) $\nu_{(x, f(x))} \in T_{(x, f(x))} \mathcal{H}$ for any $x \in U_0$,\\
(3) $\Crit (\nu) \cap U' \subset \mathcal{H}$, where $U' \subset U$ is a smaller open neighborhood of $0$ in $H$. Here $Crit(\nu)$ is a set of critical points of the vector field $\nu$. 
\end{dfn}
Roughly speaking a center manifold is a finite-dimensional submanifold that is locally preserved by the flow of $\nu$ and contains all nearby critical points. Now we let $\Gamma$ be a smooth flat connection on $Y$ with a local slice neighborhood $U_{\Gamma}$ satisfying \autoref{l4.3}. We take $U \subset \mathcal{K}_{\Gamma}$ so that $U_{\Gamma}=\{\Gamma\} + U$. The deformation complex of the space of flat connections at $\Gamma$ is given by 
\[
L^2_{l+1}(Y, \SU(2)) \xrightarrow{d_{\Gamma}} L^2_l(Y, T^*Y \otimes \SU(2)) \xrightarrow{d_{\Gamma}} L_{l-1}^2(Y, \Lambda^2 T^*Y \otimes \SU(2)).
\]
We identify $H^1(Y; \ad \Gamma)=\ker d_{\Gamma} \cap \ker d^*_{\Gamma}$ as a subspace in $\mathcal{K}_{\Gamma}$. Let $H^{\perp}_{\Gamma} \subset \mathcal{K}_{\Gamma}$ be the $L^2$-orthogonal complement of $H^1(Y; \ad \Gamma)$. Corollary 5.1.4 in \cite{MMR} ensures the existence of a $\Stab(\Gamma)$-invariant $C^2$-center manifold $\mathcal{H}_{\Gamma}$ for the pair $(U, -\grad_{\Gamma} cs_{\Gamma})$. We write $W^s_{\Gamma} \subset \mathcal{H}_{\Gamma}$ for the stable set of $-\grad_{\Gamma} cs_{\Gamma}$ on the center manifold, i.e. for any $B \in W^s_{\Gamma}$ the flowline of $-\grad_{\Gamma} cs_{\Gamma}$ starting at $B$ converges to some point in $\mathcal{H}_{\Gamma}$. The most important property of the center manifold is that any perturbed ASD connection on the end can be approximated exponentially closely by a gradient flowline on the center manifold after a sufficiently long time.  

\begin{prop}\label{EPA}
Let $\mathcal{H}_{\Gamma}$ be a center manifold of $\Gamma$ with respect to $(U, -\grad_{\Gamma} cs_{\Gamma})$. Then one can find a local slice neighborhood $V_{\Gamma} \Subset U_{\Gamma}$ of $\Gamma$, and positive constants $T_{\Gamma}, \epsilon_{\Gamma}, K_{\Gamma} > 0$ so that the following hold. Let $A$ be a $\sigma$-perturbed ASD $L^2_{k, loc}$-connection of finite energy over $E$ with $\|\sigma\|_{\mathcal{P}} \leq 1$. Suppose the restricted connection $A|_{[T_{\Gamma}, \infty) \times Y} = B(t) + \beta(t)dt$ satisfies 
\begin{enumerate}
\item$B(t) \in V_{\Gamma}$ for all $t  \geq T_{\Gamma} -1$. 
\item $\|F_A|_{[T_{\Gamma}-1, \infty) \times Y}\|_{L^2} < \epsilon_{\Gamma}.$
\end{enumerate}
Then there is a unique downward gradient flowline, $B_{\Gamma}: [T_{\Gamma}-1, \infty) \to \mathcal{H}_{\Gamma}$, of $\cs_{\Gamma}$ on the center manifold with the following property.  The ASD connection $A^c_{\Gamma}=B_{\Gamma}(t) + \Theta(B_{\Gamma}(t)) dt$ induced from $B_{\Gamma}$ satisfies
\[
\|A - A^c_{\Gamma}\|_{L^2_k([t-1/2, t+1/2] \times Y)} < K_{\Gamma} e^{-{\mu_{\Gamma} \over 2}(t-T_{\Gamma})}, \; \forall t \geq T_{\Gamma},
\]
where $\mu_{\Gamma}$ is the smallest nonzero absolute value of eigenvalues of the Hessian $*d_{\Gamma}|_{\ker d^*_{\Gamma}}$. 
\end{prop}

\begin{proof}
This is Theorem 5.2.2 in \cite{MMR} when one adopts metric perturbations. Its proof carries through our case without any change. From the finite length of $B(t)$ and finite energy of $A$, (1) and (2) follow immediately. Over $[T_{\Gamma}-1, \infty) \times Y$, we decompose $A=\Gamma+b(t)+c(t)+\beta(t)dt$, where $b(t) \in \mathcal{H}_{\Gamma}$, $c(t) \in H_{\Gamma}^{\perp}$. \cite[Lemma 5.4.1]{MMR} tells us that 
\[
\|c(t)\|_{L^2} \leq c. e^{-{\mu_{\Gamma} \over 2}(t-T_{\Gamma})}. 
\]
From (\ref{pas}) and $\mu > {\mu_{\Gamma} \over 2}$, we conclude that 
\[
\| \dot{b}(t) + \grad_{\Gamma} cs_{\Gamma}(B(t))\|_{L^2} \leq c.  e^{-{\mu_{\Gamma} \over 2}(t-T_{\Gamma})}
\]
Then \cite[Lemma 5.3.1]{MMR} gives us the unique gradient flowline $B_{\Gamma}=\Gamma + b_{\Gamma}: [T_{\Gamma} -1, \infty) \to \mathcal{H}_{\Gamma}$ such that 
\[
\| b_{\Gamma}(t) - b(t) \|_{L^2} \leq c. e^{-{\mu_{\Gamma} \over 2}(t-T_{\Gamma})}. 
\]
From \autoref{l4.4}, it follows that 
\[
\begin{split}
\|\beta(t) - \Theta(B_{\Gamma}(t))\|_{L^2} & \leq \|\beta(t) - \Theta(B(t))\|_{L^2} + \|\Theta(B(t)) - \Theta(B_{\Gamma}(t))\|_{L^2} \\
& \leq c. e^{-{\mu_{\Gamma} \over 2}(t-T_{\Gamma})}. 
\end{split}
\]
Now the result follows from the standard bootstrapping argument by applying \autoref{l4.11}.
\end{proof}

\begin{rem}\label{rt0}
After we fix a slice neighborhood $V_{\Gamma}$ of each flat connection $\Gamma$ that satisfies \autoref{EPA}, the constants $T_{\Gamma}, \epsilon_{\Gamma}, K_{\Gamma}$ depend continuously on $[\Gamma]$ and $[A]$. The compactness of $\chi(Y)$ and $\M_{\sigma}(Z)$ (to be proved in the next subsection) implies that these parameters can be chosen uniformly. For later use, we denote the uniform bound on $T_{\Gamma}$’s by $T_{\dagger}$. 
\end{rem}
We write $\M_{\sigma}(Z, V_{\Gamma})$ for the set of equivalence classes $[A]$ of ASD connections on $Z$ for which one can pick up a representative $A$ whose restriction on the end $[T_{\dagger}, \infty) \times Y$ satisfies \autoref{EPA}. Then the assignment 
\begin{equation}
\begin{split}
Q_{\Gamma}: \M_{\sigma}(Z, V_{\Gamma}) & \longrightarrow W^s_{\Gamma} \\
[A] & \longmapsto B_{\Gamma}(T_{\dagger})
\end{split}
\end{equation}
defines a continuous map following the same argument as in \cite[Proposition 5.2.2]{MMR}. Note that any two gauge transformations transforming the restriction of $A$ on the end into a connection of standard form with respect to $\Gamma$ differ by a constant gauge transformation in $\Stab(\Gamma)$, thus the map $Q_{\Gamma}$ is well-defined. The map $Q_{\Gamma}$ refines the asymptotic map $\partial_+$ in the sense that $\partial_+$ is the composition of $Q_{\Gamma}$ with the map sending $B_{\Gamma}(T_{\dagger})$ to the limit point in $V_{\Gamma}$ following the downward gradient flowline of $\cs_{\Gamma}$. We also have the based version 
\begin{equation}
\tilde{Q}_{\Gamma}: \tilde{\M}_{\sigma}(Z, V_{\Gamma}) \to W^s_{\Gamma} \times_{\Stab{\Gamma}} E'_{y_0}.
\end{equation}

\subsection{Compactness of $\M_{\sigma}(Z)$}

In the presence of mere metric perturbations, the compactness of the moduli space of instantons over compact $4$-manifolds is well-understood due to the work of Uhlenbeck \cite{U82,U82b}. For the case of end-cylindrical $4$-manifolds, the compactification involves two parts. One is the bubbling phenomenon, namely energy concentration of a subsequence of instantons on a discrete subset of points. As pointed out in \cite{K04}, the extra subtlety caused by holonomy perturbations comes from its global effect on the part outside the geodesics balls where one runs Uhlenbeck's convergence argument. As a consequence, away from those bubbling points only $L^p_1$-convergence, $p \geq 2$, can be achieved rather than $C^{\infty}$-convergence in the metric-perturbed case. The other possibility is the escape of energy of a subsequence of instantons along the cylindrical end, which could also be regarded as an energy concentration phenomenon but occurring at ‘infinity’. Due to the exponential decay assumption of perturbations, the convergence of this type is actually in $C^{\infty}$ class. For the charge-zero moduli spaces of our concern, we are going to argue that neither of the phenomena could happen when perturbations are sufficiently small.

\begin{prop}\label{p5.13}
One can find a constant $\epsilon_3 > 0$ so that $\M_{\sigma}(Z)$ is compact if $\|\sigma\|_{\mathcal{P}} \leq \epsilon_3$. 
\end{prop}

The bubbling phenomenon can be easily excluded. Let’s take $[A] \in \M_{\sigma}(Z)$. Then $\kappa(A) = 0$ implies that 
\[
\|F^+_A\|_{L^2(Z)} = \|F^-_A\|_{L^2(Z)} \qquad \|F^+_A\|_{L^2(Z)} = \|\sigma(A)\|_{L^2(Z)}. 
\]
Combining the $C^0$-bound of the holonomy perturbation $\sigma$ in \autoref{l3.1} and \autoref{l3.5}, we get that 
\begin{equation}\label{e5.20}
\|\sigma(A)\|_{L^2(Z)} \leq \left(K_0 \sqrt{\vol(M)} + K’_0 \sqrt{\vol(Y)}e^{-\mu}/\mu\right)\|\sigma\|_{\mathcal{P}}. 
\end{equation}
Thus the energy $\mathcal{E}(A) = 2\|\sigma(A)\|_{L^2(Z)}$ admits a uniform bound by $\|\sigma\|_{\mathcal{P}}$. The occurrence of a bubbling point from a convergent sequence of instantons would occupy energy of amount $8\pi^2$. By choosing $\|\sigma\|_{\mathcal{P}}$ sufficiently small so that the right-hand side of (\ref{e5.20}) is less than $4\pi^2$, the bubbling phenomenon cannot happen. More directly, one can argue using Uhlenbeck’s gauge fixing \cite{U82}. Over each geodesics ball $D$, one can find a threshold $\epsilon_D > 0$ so that whenever $\mathcal{E}(A|_{D}) < \epsilon_D$ the $L^2_2$-norm of the connection form over $D’ \Subset D$, up to gauge transformations, admits a uniform bound by $\|\mathcal{\sigma}\|_{\mathcal{P}}$ and $\mathcal{E}(A|_D)$. Thus one can find a strongly $L^2_1$-convergent subsequence over $D$. Since $Z$ has bounded geometry, the threshold $\epsilon_D$ can be chosen to be uniform for a good collection of geodesics balls that cover $Z$. Then for any compact region of $Z$, one can apply the gauge patching argument to achieve global $L^2_1$-convergence.

Let’s write $Z_T = M \cup [0, T] \times Y$ for any $T \geq 1$. Recall a perturbation takes the form $\sigma = \sigma_{\pmb{\omega}} + \sigma_{\pmb{\pi}}$ with $\supp \sigma_{\pmb{\omega}} \subset Z_1$ and $\sigma_{\pmb{\pi}}$ defined slicewise over the end. Thus for any connection $A$ over $Z_T$ or $I \times Y$ with $I \subset [1, \infty)$ a compact interval, the perturbation $\sigma(A)$ is well-defined. Let’s write
\[
\begin{split}
\M_{\sigma}(Z_T): & =\left\{ A\in \A_k(Z_T) : F^+_A = \sigma(A) \right\}\Bign/\G_{k+1} \\
\M_{\sigma}(I \times Y): &=\left\{ A \in \A_k(I \times Y): F^+_A = \sigma(A) \right\}\Bign/\G_{k+1}
\end{split} 
\]
for the moduli spaces of perturbed instantons of class $L^2_k$ over $Z_T$ and $I \times Y$ respectively. Note that we don’t impose any assumption on the instanton charge of these moduli spaces. 

\begin{lem}\label{l5.14}
There exists a constant $\epsilon > 0$ so that the moduli spaces of perturbed instantons with energy no larger than $\epsilon$
\[
\begin{split}
\M^{\epsilon}_{\sigma}(Z_T):&= \left\{[A] \in \M_{\sigma}(Z_T): \mathcal{E}(A) \leq \epsilon \right\}\\
 \M^{\epsilon}_{\sigma}(I \times Y):&= \left\{[A] \in \M_{\sigma}(I \times Y): \mathcal{E}(A) \leq \epsilon \right\} 
\end{split}
\]
are both compact, where $T \geq 1$, and $I \subset [1, \infty)$ is a compact sub-interval. 
\end{lem}

\begin{proof}
Let $A_n$ be a sequence of perturbed ASD connections whose energy $\mathcal{E}(A_n)$ is uniformly bounded by $\epsilon_4$. The argument above shows that by choosing $\epsilon$ sufficiently small, one can find gauge transformations $g_n$ of class $L^2_{k+1}$ so that, after passing to a subsequence, $g_n \cdot A_n$ converges strongly in $L^2_1$ to an $L^2_2$ connection $A_o$. It follows from \cite[Proposition 3.5]{K04} that $A_o$ solves the perturbed ASD equation $F^+_{A_o} = \sigma(A_o)$. Choosing $\epsilon$ even smaller, we can apply the bootstrapping argument as in the proof of \autoref{l4.5} to deduce that $A_o$ actually lies in the class $L^2_k$, which shows that the limit $[A_o] \in \M^{\epsilon}_{\sigma}$. 
\end{proof}

Now it remains to exclude the occurrence of energy escape along the cylindrical end. In the non-perturbed or metric perturbed case, the energy of a downward gradient flowline is twice the drop of the Chern--Simons functional. Since the character variety $\chi(Y)$ is compact, one can obtain a minimum by which the Chern--Simons functional can possibly drop along non-constant flowlines. Once the energy of instantons is less than twice of this minimum, there is no energy escape. However for our perturbed case, there could be a small room for the Chern--Simons functional to increase along the end. It turns out that such an error cannot contribute to ‘broken’ flowlines supporting the loss of energy. 

\begin{proof}[Proof of \autoref{p5.13}]
Let $[A_n] \in \M_{\sigma}(Z)$ be a sequence of instantons perturbed by $\sigma$. We write $\partial_+([A_n]) = [\Gamma_n] \in \chi(Y)$ for the asymptotic values. Since $\chi(Y)$ is compact, passing to a subsequence, $[\Gamma_n]$ converges to $[\Gamma_o] \in \chi(Y)$ for which we can choose a smooth flat connection $\Gamma_o$ as a representative. Choosing $\|\mathcal{\sigma}\|_{\mathcal{P}}$ small, and applying \autoref{p5.9}, we can find a local slice neighborhood $U_{\Gamma_o}$ of $\Gamma_o$ and a moment $T_o > 0$ such that $A_n|_{[T_o, \infty) \times Y}$ is in standard form with respect to $\Gamma_o$ and $B_n(t):= A_n|_{\{t\} \times Y} \in U_{\Gamma_o}$ for $t \in [T_o, \infty)$ and all $n$ in the subsequence, after possibly gauge transformations. 

Now applying \autoref{l4.11} to the intervals $I_m:=[T_o+m, T_o+m+1]$ with $m \in \mathbb{N}$, the uniform $L^2_{k+1}$-bound gives us a $L^2_k$-convergent subsequence of $A_n|_{I_m \times Y}$ converging to $A^m_o$ for each $m$, which solves the perturbed ASD equation by \cite[Proposition 3.5]{K04}. Then it follows from \autoref{l5.14} and a diagonal process, we get an $SU(2)$-connection $A_o$ of class $L^2_k$ that solves the perturbed equation $F^+_{A_o} = \sigma(A_o)$ such that, after passing to subsequences, $A_n$ converges to $A_o$ in $L^2_{k, loc}$-topology. 

It remains to show that $\kappa(A_o) = 0$. To this end, we write $\Gamma’_o = \partial_+ (A_o) \in U_{\Gamma}$ for the asymptotic flat connection of $A_o$, and claim that $\Gamma_o’ = \Gamma_o$. Suppose this is not true. By shrinking $U_{\Gamma_o}$ if necessary at the start, we can arrange that $\Gamma_o$ and $\Gamma’_o$ stay in the same connected component of $U_{\Gamma_o}$ thus share the same value of the Chern--Simons functional. Then we can find neighborhoods of $V_{\Gamma_o}$ and $V_{\Gamma’_o}$ of $\Gamma_o$ and $\Gamma’_o$ respectively in $U_{\Gamma_o}$ such that for all $B \in V_{\Gamma_o}$ and $B’ \in V_{\Gamma_o’}$
\[
\text{ the $L^2_l$-distance } \dist(V_{\Gamma_o}, V_{\Gamma’_o}) \geq \epsilon_V \text{ and } |\cs(B) - \cs(B’)| < \epsilon’_V
\]
with respect to some positive constants $\epsilon_V \gg \epsilon_V’$. We write $B_o(t) :=A_o|_{\{t\} \times Y}$ and fix $T’_0$ so that $B_o(t) \in V_{\Gamma’_o}$ for all $t > T’_0$. The $L^2_{k, loc}$ convergence of $A_n$ implies that we can choose an increasing sequence $T’_{\alpha} \to \infty$, $\alpha \in \N$, and a corresponding subsequence $n_{\alpha} \to \infty$ such that for $B_{n_{\alpha}} \in V_{\Gamma’_o}$ for all $t \in [T’_{\alpha}-1, T’_{\alpha}+1]$. Moreover, the convergence of $B_n(t) \to \Gamma_n$ and $\Gamma_n \to \Gamma_o$ gives us a sequence $T_{\alpha}$ such that $B_{n_{\alpha}}(t) \in V_{\Gamma}$ for all $t > T_{\alpha} - 1$. Note that the sequence $\{T_{\alpha}\}$ must be unbounded, otherwise $B_o(t) \in V_{\Gamma}$ for any $t$ greater than the uniform bound of $\{T_{\alpha}\}$ due to the $L^2_{k,loc}$-convergence of $A_n$ to $A_o$. Passing to a further subsequence, we may assume $T_{\alpha} \to \infty$. Now we know 
\[
\begin{split}
\epsilon_V & \leq \|B_{n_{\alpha}}(T_{\alpha}) - B_{n_{\alpha}}(T’_{\alpha})\|_{L^2_l} \\
& \leq c. \left( \|B_{n_{\alpha}}(T_{\alpha}) - B_{n_{\alpha}}(T’_{\alpha})\|_{L^2} + e^{-\mu T’_{\alpha}} \right) \\
& \leq c. \left( \int_{T’_{\alpha}}^{T_{\alpha}} \|\dot{B}_{n_{\alpha}}(t)\|_{L^2} dt + e^{-\mu T’_{\alpha}} \right) \\
& \leq c. \left(\|F_{A_{n_{\alpha}}}\|_{L^2([T’_{\alpha}, T_{\alpha}] \times Y)} + e^{-\mu \theta T’_{\alpha}} \right),
\end{split}
\]
where the second inequality used the regularity estimate \autoref{l4.11}, and the fourth inequality used the length estimate \autoref{p5.7}. Now we know 
\begin{equation}\label{e5.21}
\begin{split}
\cs(B(T’_{\alpha})) - \cs(B(T_{\alpha})) & = \frac{1}{2} \int_{[T’_{\alpha}, T_{\alpha}] \times Y} \tr F_A \wedge F_A \\
& = \frac{1}{2}\|F_A\|_{L^2}- \|F^+_A\|_{L^2} \\
& \geq \frac{1}{2} \|F_A\|_{L^2} - c. e^{-\mu T’_{\alpha}} \|\sigma\|_{\mathcal{P}}. 
\end{split}
\end{equation}
Note that $|\cs(B(T_{\alpha})) - \cs(B(T’_{\alpha}))| < \epsilon’_V$ from our construction. We conclude that 
\begin{equation}\label{e5.22}
\epsilon_V \leq \mathfrak{c} \left(\epsilon_V’ + \|\sigma\|_{\mathcal{P}}e^{-\mu T’_{\alpha}} + e^{-\mu \theta T’_{\alpha}} \right),
\end{equation}
where the constant $\mathfrak{c}$ is independent of the choices of $V_{\Gamma_o}$ and $V_{\Gamma’_o}$. By shrinking $V_{\Gamma_o}$ and $V_{\Gamma’_o}$, we can arrange that $\epsilon_V > 2\mathfrak{c} \epsilon’_V$. Then (\ref{e5.22}) leads to a contradiction from the fact that $T’_{\alpha} \to \infty$. This completes the proof of the claim. 

Finally, we go back to show that $\kappa(A_o) = 0$. Equivalently this is $\mathcal{E}(A_o) = 2\|\sigma(A_o)\|_{L^2(Z)}$. Since the $L^2$-norm of the perturbation $\sigma(A)$ has uniform exponential decay along the end, the $L^2_{k, loc}$-convergence of $A_n$ implies that $\|\sigma(A_n)\|_{L^2(Z)} \to \|\sigma(A_o)\|_{L^2(Z)}$. Since $B_n(t) \to \Gamma_o$, (\ref{e5.21}) implies that for any $T > T_o$
\[
\mathcal{E}(A_n|_{[T, \infty) \times Y}) \leq c. \left(\cs(B_n(T)) - \cs(\Gamma_o) + e^{-\mu T}\|\sigma\|_{\mathcal{P}} \right),
\]
where the constant is independent of $n$. The same inequality holds for $A_o$ due to the convergence $B_o(T) \to \Gamma_o$. Thus 
\begin{equation}\label{e5.23}
\begin{split}
\Bigl\lvert\mathcal{E}(A_n) - \mathcal{E}(A_o)\Bigr\rvert & \leq \Bigl\lvert\mathcal{E}(A_n|_{Z_T}) - \mathcal{E}(A_o|_{Z_T})\Bigr\rvert \\
& + c.  \Big(\cs(B_n(T)) + \cs(B_o(T)) - 2\cs(\Gamma_o) + 2e^{-\mu T}\|\sigma\|_{\mathcal{P}} \Big)
\end{split}
\end{equation}
Now for any given $\epsilon > 0$, we may choose $T_{\epsilon}$ and $N_{\epsilon}$ so that the second term in the right-hand side of (\ref{e5.23}) is less than $\epsilon/2$ for $T=T_{\epsilon}$ and all $n > N_{\epsilon}$. Then by making $n$ larger, we can arrange the first term in the right-hand side of (\ref{e5.23}) to be less than $\epsilon/2$ as well. This proves that $\mathcal{E}(A_n) \to \mathcal{E}(A_o)$. Then the result follows from the fact that $\mathcal{E}(A_n) = 2\|\sigma(A_n)\|_{L^2(Z)}$ for all $n$. 
\end{proof}

\section{Transversality on the Irreducible Moduli Space}\label{tims}
Following the idea of Morgan--Mrowka--Ruberman \cite{MMR}, to improve the regularity of the refined map $Q_{\Gamma}$ one can first embed the moduli space $\tilde{\M}_{\sigma}(Z, V_{\Gamma})$ into a larger one which they refer to as a ‘thickened moduli space’, then prove transversality results there. 

The thickened moduli space is defined with the help of thickening data about smooth flat connections on $E'$, which we now recall from \cite{MMR}. The motivation for introducing the thickening data is to resolve the issue that the gradient vector field $\grad_{\Gamma} cs_{\Gamma}$ is incomplete on a local slice neighborhood of $\Gamma$. So one artificially truncates this vector field via a cut-off function. In this way all the local properties near $\Gamma$ are preserved, and one can apply analysis tools without worrying about the incompleteness. 

\begin{dfn}
Let $\Gamma$ be a smooth flat connection on $E'$. We choose 
\begin{enumerate}[label=(\alph*)]
\item a center manifold $\mathcal{H}_{\Gamma}$ of $\Gamma$, 
\item neighborhoods $V_{\Gamma} \subset V_{\Gamma}' \subset U_{\Gamma}$ as in \autoref{rt0}, 
\item a cut-off function $\varphi_{\Gamma}: \mathcal{H}_{\Gamma} \to [0,1]$ such that $\varphi_{\Gamma} \equiv 1$ on $V'_{\Gamma} \cap \mathcal{H}_{\Gamma}$ and $\supp \varphi_{\Gamma} \subset U'_{\Gamma} \cap \mathcal{H}_{\Gamma}$ for some $U'_{\Gamma} \subset U_{\Gamma}$.
\end{enumerate}
We refer to the triple $\mathcal{T}_{\Gamma} = (\mathcal{H}_{\Gamma}, V_{\Gamma}, \varphi_{\Gamma})$ as a set of thickening data about $\Gamma$. 
\end{dfn}

Given a thickening triple $\mathcal{T}_{\Gamma}$, we write 
\[
\mathcal{H}_{\Gamma}^{out} := \varphi_{\Gamma}^{-1}(0,1] \text{ and } \mathcal{H}_{\Gamma}^{in}:= \varphi_{\Gamma}^{-1}(1). 
\] 
We denote by 
\[
\Xi^{tr}_{\Gamma} :=- \varphi_{\Gamma} \cdot \grad_{\Gamma} cs_{\Gamma}|_{\mathcal{H}_{\Gamma}}
\]
the truncated downward gradient vector field over the center manifold $\mathcal{H}_{\Gamma}$. Then $\Xi^{tr}_{\Gamma}$ is a complete vector field over $\mathcal{H}_{\Gamma}$ despite that $\mathcal{H}_{\Gamma}$ is only defined near $\Gamma$. For each $h \in \mathcal{H}^{out}_{\Gamma}$, we let $B_h:[T_{\dagger}, \infty) \to \mathcal{H}_{\Gamma}$ be the unique flowline of $\Xi^{tr}_{\Gamma}$ such that $B_h(T_0)=h$. Now we extend the connection $B_h(t) + \Theta(B_h(t))dt$ smoothly to a connection $A_h$ on the entire manifold $Z$ so that over the compact part $A_h|_{Z \backslash [T_0, \infty) \times Y}$ depends on $h$ smoothly. To put the weighted Sobolev space into the package, we choose a weight $\delta_{\Gamma} \in (0, {\mu_{\Gamma} \over 2} )$ for each $[\Gamma] \in \chi(Y)$. 

\begin{dfn}\label{d2.17}
Given $h \in \mathcal{H}_{\Gamma}^{out}$, we fix a connection $A_h$ as above, and write
\[
\A_{k, \delta_{\Gamma} }(Z, \mathcal{T}_{\Gamma}, h):=\{ A \in \A_{k, loc}(Z) : A-A_h \in L^2_{k, \delta_{\Gamma}}(Z, T^*Z \otimes \SU(2))\}.
\] 
We denote the union by 
\[
\A_{k, \delta_{\Gamma}}(Z, \mathcal{T}_{\Gamma}) := \bigcup_{h \in \mathcal{H}_{\Gamma}^{out}} \A_{k, \delta_{\Gamma}}(Z, \mathcal{T}_{\Gamma}, h).
\]
The gauge group $\G_{k+1, \delta_{\Gamma}}(Z, \mathcal{T}_{\Gamma})$ that preserves $\A_{k, \delta_{\Gamma}}(Z, \mathcal{T}_{\Gamma})$ consists of all $L^2_{k+1, loc}$ gauge transformations $g$ such that there exists $\tau \in \Stab(\Gamma)$ satisfying 
\[
g|_{[T_0, \infty) \times Y} \circ \tau - \id \in L^2_{k+1, \delta_{\Gamma}}([T_0, \infty) \times Y, SU(2)).
\]
Finally we pick a cut-off function $\varphi: Z \to [0, 1]$ such that $\varphi|_{[T_0+1, \infty) \times Y} \equiv 1$ and $\varphi|_{[0, T_0] \times Y} \equiv 0$. 
The thickened moduli space with respect to the thickening data $\mathcal{T}_{\Gamma}$ perturbed by $\sigma \in \mathcal{P}_{\mu}$ is defined to be 
\[
\M_{\sigma}(Z, \mathcal{T}_{\Gamma}):= \Bigl\{ A \in \A_{k, \delta_{\Gamma}}(Z, \mathcal{T}_{\Gamma}) : F^+_A - \varphi F^+_{A_h} = \sigma(A), \kappa(A)=0 \Bigr\} \Bign/ \G_{k+1, \delta_{\Gamma}}(Z, \mathcal{T}_{\Gamma}). 
\]
The thickened based moduli space $\tilde{\M}_{\sigma}(Z, \mathcal{T}_{\Gamma})$ is defined similarly.
\end{dfn}

For the rest of this section, we are concerned with irreducible connections. The construction of the thickened moduli space gives us a map 
\begin{equation}
\begin{split}
P_{\Gamma}: \M^*_{\sigma}(Z, \mathcal{T}_{\Gamma}) & \longrightarrow \mathcal{H}_{\Gamma} \\
[A] & \longmapsto h,
\end{split}
\end{equation}
where $h$ is the element specified in \autoref{d2.17}. We also have the based version
\begin{equation}
\tilde{P}_{\Gamma}: \tilde{\M}^*_{\sigma}(Z, \mathcal{T}_{\Gamma}) \to \mathcal{H}_{\Gamma} \times_{\Stab{\Gamma}} E'_{y_0}. 
\end{equation}
Note that the perturbations $\sigma: \A_{k, loc} \to L^2_{k, \mu}$ are smooth maps. Following from \cite[Section 7.3]{MMR} the based thickened moduli space $\tilde{\M}^*_{\sigma}(Z, \mathcal{T}_{\Gamma})$ is a $C^2$-manifold and $\tilde{P}_{\Gamma}$ is a $C^2$-map. From \autoref{EPA}, when $\|\sigma\|_{\mathcal{P}} \leq 1$ we have an identification 
\[
\tilde{\M}^*_{\sigma}(Z, V_{\Gamma}) \simeq \tilde{P}_{\Gamma}^{-1} (W^s_{\Gamma} \cap V_{\Gamma} \times_{\Stab{\Gamma}} E'_{y_0}).
\]
Denote the embedding by $j: \tilde{\M}_{\sigma}(Z, V_{\Gamma}) \hookrightarrow \tilde{M}_{\sigma}(Z, \mathcal{T}_{\Gamma})$. We obtain the following commutative diagram: 
\begin{equation}\label{ic}
\begin{tikzcd}
 \tilde{\M}^*_{\sigma}(Z, V_{\Gamma})  \arrow[hookrightarrow]{r}{j} \ar[d, "\tilde{Q}_{\Gamma}"] & \tilde{\M}^*_{\sigma}(Z, \mathcal{T}_{\Gamma})  \ar[d, "\tilde{P}_{\Gamma}"] \\
W^s_{\Gamma}\times_{\Stab{\Gamma}} E'_{y_0}  \ar[hookrightarrow]{r} & \mathcal{H}_{\Gamma} \times_{\Stab{\Gamma}} E'_{y_0}
\end{tikzcd} 
\end{equation}

As mentioned above, we introduce the thickened moduli space mainly to establish the transversality result. The following result is a variance of \cite[Theorem 9.0.1]{MMR} in our case. 

\begin{prop}\label{TSC}
Let $\mathfrak{S} \subset \mathcal{H}_{\Gamma} \times_{\Stab{\Gamma}} E'_{y_0}$ be the union of a finite set of smooth submanifolds. Then with respect to a generic perturbation $\sigma$, $\tilde{P}_{\Gamma}$ is transverse to $\mathfrak{S}$. Moreover the dimension of the based thickened irreducible moduli space $\tilde{\M}^*_{\sigma}(Z, \mathcal{T}_{\Gamma})$  is given by 
\begin{equation}\label{dimm}
-{3 \over 2}(\chi(Z)+\sigma(Z))+{h^1_{\Gamma}- h^0_{\Gamma} \over 2}+{\rho(\Gamma) \over 2}+ 3,
\end{equation}
where $h^i_{\Gamma}=\dim H^i(Y; \ad \Gamma)$, $i=0, 1$, $\rho(\Gamma)$ is the Atiyah--Patodi--Singer $\rho$-invariant of the odd signature operator twisted by $\Gamma$ in \cite{APS2}, $\chi(Z)$ is the Euler characteristic, and $\sigma(Z)$ is the signature. 
\end{prop}

\begin{proof}
It only remains to show the transversality of $\tilde{P}_{\Gamma}$. The computation of the formal dimension is the same as that in \cite[Chapter 8]{MMR}. Consider the map 
\[
\begin{split}
\mathcal{F}: \mathcal{P}_{\mu} \times \A^*_{k, \delta_{\Gamma}}(Z, \mathcal{T}_{\Gamma}) &\longmapsto L^2_{k-1, \delta_{\Gamma}}(Z, \Lambda^+\otimes \SU(2)) \\
(\sigma, A) & \longmapsto F^+_A - \varphi F^+_{A_h} - \sigma(A).
\end{split}
\]
Denote by $\mathcal{F}_h:=\mathcal{F}|_{\mathcal{P}_{\mu} \times \A_{k, \mu}(Z, \mathcal{T}_{\Gamma}, h)}$ the restricted map. Then the differential of $\mathcal{F}_h$ is 
\[
D\mathcal{F}_h|_{(\sigma, A)}(\pmb{\omega}, \pmb{\pi}, a) = d^+_Aa - D\sigma|_A a - \sigma_{\pmb{\omega}}(A) - \sigma_{\pmb{\pi}}(A),
\]
where $a \in L^2_{k,\mu}(T^*Z \otimes \SU(2))$, $\pmb{\omega} \in \mathcal{P}_1$, and $\pmb{\pi} \in \mathcal{P}_2$. Due to the unique continuation of instantons over the cylinder (c.f \cite[Section 4.3.4]{DK90} and \cite[Section 7]{KM07}), the irreducibility of $A$ implies that $A|_M$ and $A|_{[0, \infty) \times Y}$ are both irreducible. Then \cite[Lemma 13]{K04} implies that the image of $\sigma_{\pmb{\omega}}(A)$ is dense in $L^2_{k-1}(M, \Lambda^+ \otimes \SU(2))$, and \cite[Proposition 5.17]{D02} implies that the image of $\sigma_{\pmb{\pi}}(A)$ is dense in $L^2_{k-1, \delta_{\Gamma}}([0, \infty) \times Y, \Lambda^+\otimes \SU(2))$, as we vary $\pmb{\omega}$ and $\pmb{\pi}$. On the other hand, the operator 
\[
d^*_{A, \tau\delta_{\Gamma}} \oplus d_A^+: L^2_{k-1, \delta_{\Gamma}}(Z, \SU(2)) \to  L^2_{k, \delta_{\Gamma}}(Z, \Lambda^1 \otimes \SU(2)) \oplus L^2_{k-1, \delta_{\Gamma}}(Z, \Lambda^+ \otimes \SU(2))
\]
where $d^*_{A, \tau\delta_{\Gamma}} = e^{-\tau \delta_{\Gamma}} d^*_A e^{\tau \delta_{\Gamma}}$, and $\tau: Z \to \R$ is a smooth function such that $\tau|_{\{t\} \times Y} =t$, is Fredholm except at a discrete set of $\R$ from the theory of Atiyah--Patodi--Singer \cite{APS1}. We may choose $\delta_{\Gamma}$ in the first place to make this operator Fredholm. In particular we see that the image of $d^+_A$ is closed and has finite dimensional cokernel. Thus we conclude that $\mathcal{F}_h$ is a submersion. 

Consider the map 
\[
\tilde{P}_{\Gamma}': \mathcal{P}_{\mu} \times \big( \A^*_{k, \delta_{\Gamma}}(Z, \mathcal{T}_{\Gamma}) \times_{\G} E_{z_0} \big) \longrightarrow \mathcal{H}_{\Gamma} \times_{\Stab(\Gamma)} E'_{y_0} 
\]
sending $(\sigma, [A, v])$ to the pair $[h, v']$, where $v'$ is the limit of $v$ under the parallel transport by $A$. By the construction of $\A^*_{k, \delta_{\Gamma}}(Z, \mathcal{T}_{\Gamma})$, the map $\tilde{P}_{\Gamma}'$ is a submersion. Since $\mathcal{F}$ is gauge equivariant and $\mathcal{F}_h$ is a submersion, we conclude that the restricted map
\[
\tilde{P}'_{\Gamma}: \mathcal{F}^{-1}(0) \times_{\G} E_{z_0} \longmapsto \mathcal{H}_{\Gamma} \times_{\Stab(\Gamma)} E'_{y_0}
\]
is also a submersion. Denote by $\Pi: \mathcal{P}_{\mu} \times \big( \A^*_{k, \delta_{\Gamma}}(Z, \mathcal{T}_{\Gamma}) \times_{\G} E_{z_0} \big) \to \mathcal{P}_{\mu}$ the projection onto the first factor. Then the Sard--Smale theorem tells us that $\tilde{P}_{\Gamma}=\tilde{P}'_{\Gamma}|_{\Pi^{-1}(\sigma)}$ is transverse to a given smooth submanifold for a generic perturbation $\sigma$. 
\end{proof}

Before extracting more information of the asymptotic map from \autoref{TSC}, Let us first recall the Kuranishi obstruction map at a flat connection $\Gamma$. 
\begin{thm}[{\cite[Theorem 12.1.1]{MMR}}]\label{Kur}
Let $\Gamma$ be a $C^{\infty}$ flat connection on $Y$. Then there exists a $\Stab(\Gamma)$-invariant neighborhood $V$ of $0$  in $H^1(Y; \ad \Gamma)$, a $\Stab(\Gamma)$-invariant neighborhood $U$ of $\Gamma$ in $\mathcal{S}_{\Gamma}$, and $C^{\infty}$ $\Stab(\Gamma)$-equivariant maps  
\begin{equation}
\mathfrak{p}_{\Gamma}: V  \to U \text{ and } 
\mathfrak{o}_{\Gamma}: V  \to H^2(Y; \ad \Gamma) 
\end{equation}
satisfying 
\begin{enumerate}
\item[\upshape (i)] $\mathfrak{p}_{\Gamma}$ is an embedding whose differential at $0$ coincides with the inclusion map $H^1(Y; \ad \Gamma) \hookrightarrow \mathcal{K}_{\Gamma}$. 
\item[\upshape (ii)] The restriction of $\mathfrak{p}_{\Gamma}|_{\mathfrak{o}_{\Gamma}^{-1}(0)}$ is a homeomorphism onto the space of flat connections in $U$. 
\end{enumerate}
\end{thm}

\begin{rem}\label{RKur}
Roughly speaking, the proof of \autoref{Kur} makes use of the implicit function theorem to obtain a map 
\begin{equation}
\mathfrak{q}_{\Gamma}: V \to \im d^*_{\Gamma} \subset L^2_1(T^*Y \otimes \SU(2)),
\end{equation}
which is characterized by the fact that 
\[
\Pi'_{\Gamma} F_{\Gamma+b+\mathfrak{q}_{\Gamma}(b)}=0,
\]
where $\Pi'_{\Gamma}: L^2_{l-1}(\Lambda^2T^*Y \otimes \SU(2)) \to \im d_{\Gamma}$ is the $L^2$-orthogonal projection. Then 
\[
\mathfrak{p}_{\Gamma}(b)=\Gamma+b+\mathfrak{q}_{\Gamma}(b) \text{ and }
\mathfrak{o}_{\Gamma}(b)=\Pi_{\Gamma} F_{\mathfrak{p}_{\Gamma}(b)}, 
\]
where $\Pi_{\Gamma}: L^2_{l-1}(\Lambda^2T^*Y \otimes \SU(2)) \to \ker d_{\Gamma}$ is the $L^2$-orthogonal projection. The zero set of map $\mathfrak{o}_{\Gamma}$ provides a local structure of the character variety $\chi(Y)$ near $[\Gamma]$. 
\end{rem}

\begin{dfn}\label{smooth}
For a flat connection $\Gamma$ on $Y$, the map $\mathfrak{o}_{\Gamma}$ in \autoref{Kur} is called the Kuranishi obstruction map. $[\Gamma] \in \chi(Y)$ is said to be a smooth point if the Kuranishi map of $\Gamma$ vanishes on $V_{\Gamma}$, i.e. $\mathfrak{o}_{\Gamma}\equiv 0$, otherwise a singular point. 
\end{dfn}

\begin{prop}[{\cite[Corollary 9.3.1]{MMR}}]\label{TIS}
Let $[\Gamma] \in \chi(Y)$ be a smooth point. Then there exists a center manifold $\mathcal{H}_{\Gamma}$ consisting of flat connections. Moreover, the asymptotic map 
\begin{equation}
\tilde{\partial}_+: \tilde{M}_{\sigma}(Z, V_{\Gamma}) \to  \mathcal{R}(Y)
\end{equation}
is $C^2$ so that given a submanifold $\mathfrak{S} \subset \mathcal{R}(Y)$, $\tilde{\partial}_+$ is transverse to $\mathfrak{S}$ with respect to a generic perturbation $\sigma$. 
\end{prop}

\begin{proof}
When $[\Gamma] \in \chi(Y)$ is a smooth point, one can take a center manifold $\mathcal{H}_{\Gamma}$ to be the graph of the map $\mathfrak{q}_{\Gamma}$. Indeed in this case $\mathcal{H}_{\Gamma}$ consists of flat connections which are the critical points of $\cs_{\Gamma}$ near $\Gamma$, and are preserved by the gradient flow of $\cs_{\Gamma}$. Since every point on the center manifold $\mathcal{H}_{\Gamma}$ is a critical point, all gradient flowlines on $\mathcal{H}_{\Gamma}$ are constant. Thus the map $\tilde{Q}_{\Gamma}$ coincides with the asymptotic map $\partial^0_+$. Then result now follows from \autoref{TSC}. 
\end{proof}

Now we restrict our attention to the case when $Y=T^3$. For each singular point in $\chi(T^3)$, Gompf--Mrowka \cite{GM} have constructed a center manifold. We recall their construction below and use it to prove that there are no irreducible ASD connections asymptotic to those singular points.  

The character variety $\chi(T^3)$ is identified as a copy of the quotient $T^3/ \sim$, where $\sim$ is given by the hypoelliptic involution consisting of $8$ fixed points corresponding to central connections . When $[\Gamma] \in \chi(Y)$ is a non-central connection, the first homology $H^1(T^3; \ad \Gamma)$ is computed as 
\begin{equation}
H^1(T^3; \ad \Gamma) \cong \mathcal{H}^1(T^3) \otimes  H^0(T^3; \ad \Gamma), 
\end{equation}
where $\mathcal{H}^1(T^3)$ is the space of harmonic $1$-forms on $T^3$, $ H^0(T^3; \ad \Gamma) \cong i\R$ is the Lie algebra of the stabilizer of $\Gamma$. Thus each $b \in H^1(T^3; \ad \Gamma)$ gives a flat connection $\Gamma+b$ due to $b \wedge b =0$. This implies that $\mathfrak{q}_{\Gamma}(b)=0$, and then the Kuranishi map $\mathfrak{o}_{\Gamma}(b) =0$ with $b$ in a small neighborhood $V_{\Gamma}$ of $0$. We conclude that $\Gamma$ is a smooth point in the sense of \autoref{smooth}. When $[\Gamma] \in \chi(Y)$ is a central connection, we have 
\begin{equation}
H^1(T^3; \ad \Gamma) \cong \mathcal{H}^1(T^3) \otimes \SU(2). 
\end{equation}
Now we fix an orthonormal frame $\{ e^1, e^2, e^3\}$ of $\mathcal{H}^1(T^3)$ with respect to the product metric. Then each $b \in H^1(T^3; \ad \Gamma)$ has the form 
\[
b=\sum_i e^i \otimes X_i, \; X_i \in \SU(2). 
\]
Thus the curvature of $\Gamma+b$ has the form
\begin{equation}
F_{\Gamma+b}=\frac{1}{2} \sum_{i, j} e^i \wedge e^j \otimes [X_i, X_j]. 
\end{equation}
In particular $F_{\Gamma+b} \in \mathcal{H}^2(T^3)$. From \autoref{RKur} we conclude that
$\mathfrak{q}_{\Gamma}(b)=0$. Thus the Kuranishi map is 
\[
\mathfrak{o}_{\Gamma}(b)=\frac{1}{2} \sum_{i, j} e^i \wedge e^j \otimes [X_i, X_j].
\]

\begin{prop}[{\cite[Proposition 15.2]{GM}}]\label{STAB}
Let $\Gamma$ be a smooth central flat connection on $E' \to T^3$. Then $H^1(T^3; \ad \Gamma)$ is a center manifold of $\Gamma$. Moreover the stable manifold $W^0_{\Gamma}$ of the origin is given by 
\[
\left\{\sum_i e^i \otimes X_i \in \mathcal{H}^1(T^3) \otimes \SU(2): \|X_i\|=\|X_j\|, \langle X_i, X_j \rangle=0, \langle X_1, [X_2, X_3] \rangle \leq 0 \right\}.
\]
\end{prop}

\begin{cor}\label{asms}
Let $\Gamma$ be a central flat connection on the trivial $SU(2)$-bundle $E'$ over $T^3$. Then for a generic perturbation $\sigma$, one has 
\[
\partial_+^{-1}([\Gamma]) \cap \M^*_{\sigma}(Z) =\varnothing,
\]
where $Z=M \cup [0, \infty) \times T^3$ is a homology $D^2 \times T^2$. 
\end{cor}

\begin{proof}
From the commutative diagram (\ref{ic}) and \autoref{TSC} we know that the map
\[
\tilde{P}_{\Gamma}: \tilde{\M}_{\sigma}^*(Z, \mathcal{T}_{\Gamma}) \to \mathcal{H}_{\Gamma} \times_{\Stab{\Gamma}} E'_{y_0}
\]
is transverse to the stable set $W^0_{\Gamma} \times_{\Stab{\Gamma}} E'_{y_0}$ for a generic perturbation $\sigma$. Moreover $\dim \tilde{\M}^*_{\sigma}(Z, \mathcal{T}_{\Gamma}) = 6$ from  (\ref{dimm}), and the stratified space $W^0_{\Gamma}$ has codimension $4$ in $\mathcal{H}_{\Gamma}$ from \autoref{STAB}. Thus $\tilde{\partial}_{+}^{-1}([\Gamma]) \cap \tilde{\M}^*_{\sigma}(Z)$ lies in a $2$-dimensional $C^2$-manifold $\tilde{P}^{-1}_{\Gamma}(W^0_{\Gamma} \times_{\Stab{\Gamma}} E'_{y_0})$. Since $\partial_+^{-1}([\Gamma]) \cap \M^*_{\sigma}(Z)$ is the quotient of the free smooth $SO(3)$-action on $\tilde{\partial}_{+}^{-1}([\Gamma]) \cap \tilde{\M}^*_{\sigma}(Z)$, we conclude that it has to be empty due to dimension counting. 
\end{proof}

So far we have a complete description of the irreducible moduli space $\M^*_{\sigma}(Z)$ for a generic perturbation $\sigma$:
\begin{enumerate}
\item[\upshape (i)] $\M^*_{\sigma}(Z)$ is a smooth oriented $1$-manifold.
\item[\upshape (ii)] $\M^*_{\sigma}(Z)$ misses all central connections in $\chi(T^3)$ under the asymptotic map. 
\item[\upshape (iii)] $\partial_+: \M^*_{\sigma}(Z) \to \chi(T^3)$ is transverse to a any sub-complex in $\chi(T^3)$ (the set of perturbations depends on the choice of the sub-complex). 
\end{enumerate}

\section{The Reducible Locus}\label{trl}

In this section, we study the structure of the reducible locus $\M^{\Red}_{\sigma}(Z)$ with respect to a generic perturbation. We continue to assume that $Y=T^3$ and $H_*(Z;\Z) \cong H_*(D^2 \times T^2; \Z)$. The first part is to give a global description of $\M^{\Red}_{\sigma}(Z)$. The second part is to give a local description of the moduli space near points that are approached by a sequence of irreducible instantons.  

Before diving into the details, we would like to comment on the choice of the weight $\delta_{\Gamma}$. There are two things we need to take care of. The first is to ensure the moduli space lies in the thickened moduli space. This is achieved by choosing $\delta_{\Gamma} \in (0, \mu_{\Gamma}/2)$, where $\mu_{\Gamma}$ is the smallest nonzero absolute value of eigenvalues of the restricted Hessian $*d_{\Gamma}|_{\ker d^*_{\Gamma}}$. The second is to ensure the deformation complex at an instanton $[A] \in \M_{\sigma}(Z)$ is Fredholm. Later we will see It is equivalent to the Fredholmness of the following complex:
\begin{equation}\tag{$F_{\delta_{\Gamma}, \sigma}$}\label{ssu2tc}
L^2_{k+1, \delta_{\Gamma}}(Z, \SU(2))  \xrightarrow{-d_A} L^2_{k, \delta_{\Gamma}}(Z, \Lambda^1 \otimes \SU(2))  \xrightarrow{d^+_{A, \sigma}} L^2_{k-1, \delta_{\Gamma}}(Z, \Lambda^+ \otimes \SU(2)),
\end{equation}
where $d^+_{A, \sigma}=d^+_A - D \sigma|_A$. We may ignore the perturbation part, since the Freholmness is preserved under small or compact perturbations. According to \cite[Lemma 8.3.1]{MMR} the complex ($F_{\delta_{\Gamma}}$) is Fredholm if and only if ${\delta_{\Gamma} \over 2}$ is not an eigenvalue of $-*d_{\Gamma}|_{\im d^*_{\Gamma}|_{\Omega^2}}$. Since $\im d^*_{\Gamma}|_{\Omega^2} \subset \ker d^*_{\Gamma}|_{\Omega^1}$, we may simply consider the smallest absolute value of eigenvalues of $*d_{\Gamma}|_{\ker d^*_{\Gamma}}$. Note that when $\Gamma$ is not a central connection, we have $H^1(T^3, \ad \Gamma) =3$ (cf. \cite[Lemma 14.2]{GM}). Thus only when approaching the central connections can the smallest absolute value of eigenvalues approach $0$. Let us fix a neighborhood $\mathcal{O}_{c}$ of the central connections in $\chi(T^3)$. Then we may choose a uniform weight $\delta > 0$ for all instantons $[A]$ with $\partial_+([A]) \in \chi(T^3) \backslash \mathcal{O}_c$. 

\subsection{The Global Picture}

Recall that an $SU(2)$-connection $A$ on $E$ is reducible if $A$ preserves a splitting $E=L \oplus L^*$ for some line bundle $L$. We write $A=A_L \oplus A^*_L$ corresponding to the splitting of $E$. The holonomy perturbation decomposes correspondingly as 
\[
\sigma(A):= 
\begin{pmatrix}
\sigma_L(A_L) & 0 \\
0 & \sigma_{L}(A_L^*) 
\end{pmatrix}
\in \Omega^+(Z, \SU(2)).
\]
Note that $\sigma_L(A_L)=-\sigma_L(A^*_L)$. Thus the perturbed ASD equation $F^+_A=\sigma(A)$ is equivalent to
\begin{equation}\label{e7.1}
F^+_{A_L}= \sigma_L(A_L),
\end{equation}
where $\sigma_L(A_L) \in \Omega^+(Z, i\R)$. 

\begin{lem}\label{rrl}
Let $[A] \in \M^{\Red}_{\sigma}(Z)$ be a reducible $\sigma$-perturbed instanton. Then one can choose a representative $A$ which preserves the splitting $E=\underline{\C} \oplus \underline{\C}$.
\end{lem}

\begin{proof}
As in the proof of \autoref{t5.1}, one can choose a representative $A$ such that $A|_{[T_1, \infty) \times T^3}$ is in standard form with respect to a flat connection $\Gamma$ on $E'$. Let $E=L \oplus L^*$ be the decomposition that $A$ preserves. Let us write $A=B+\beta dt$ and  $A_L=B_L+ \beta_L dt$. Then we have 
\[
F_{A_L} = F_{B_L} + dt \wedge (\dot{B}_L - d_{B_L} \beta_L). 
\]
We denote by $T_s \subset \{s\} \times T^3$ the $2$-torus representing a generator $1_Z \in H_2(Z; \Z) \cong \Z$. Then the Chern--Weil formula gives us that 
\[
c_1(L) \cdot  1_Z= {1 \over {2\pi i}} \int_{T_s} F_{A_L} =   {1 \over {2\pi i}} \int_{T_s} F_{B_L(s)} .
\]
Due to the finite energy of $A$, \autoref{l4.11} tells us that for $T \gg 0$ one has 
\[
\|F_{B(t)} \|^2_{C^0} \leq c. (\|F_A \|^2_{L^2([T, \infty) \times Y)} + \|\sigma\|_{\mathcal{P}} e^{-2\mu t}), \text{ when } t \geq T.
\]
Since $\|F_{B_L(s)}\|_{C^0} \leq \|F_{B(s)}\|_{C^0}$, we conclude that 
\[
c_1(L) \cdot 1_Z = \lim_{s \to \infty} {1 \over {2\pi i}} \int_{T_s} F_{B_L(s)}  =0.
\]
Thus $c_1(L)=0 \in H^2(Z; \Z)$ meaning $L$ is equivalent to the trivial bundle $\underline{\C}$. 
\end{proof}

\autoref{rrl} leads us to consider the moduli space of ASD $U(1)$-connections on the trivial bundle $\underline{\C}$. We write $\A_{k, loc}^{U(1)}(Z)$ for the space of $L^2_{k, loc}$ $U(1)$-connections on the trivial line bundle $\underline{\C}$ of $Z$. Fixing the product connection as the reference connection, $\A_{k, loc}^{U(1)}$ is identified with $L^2_{k, loc}(Z, T^*Z \otimes i\R)$. We write $\G^{U(1)}_{k+1,loc}$ for the $L^2_{k+1, loc}$ gauge transformations of the $U(1)$-bundle $\underline{\C}$. Then $\G^{U(1)}_{k+1,loc}$ is given by $L^2_{k+1, loc}(Z, U(1))$.

\begin{dfn}
The moduli space of perturbed anti-self-dual $U(1)$-connections is defined to be
\[
\M^{U(1)}_{\sigma}(Z):= \{ A_L \in \A_{k, loc}^{U(1)}: F^+_{A_L} = \sigma_L(A_L), \int_Z |F_{A_L}|^2 < \infty \} /\G^{U(1)}_{k+1,loc}.
\]
\end{dfn}
If we write $A_L = d + a_L$ with $a_L \in L^2_{k, loc}(Z, T^*Z \otimes i\R)$, the induced dual connection has the form $A^*_L= d-a_L$. Thus $F_{A_L} = -F_{A^*_L}$. Combining with the fact that $\sigma_L(A_L) = -\sigma_L(A^*_L)$, we get an involution $\tau$ on $\M^{U(1)}_{\sigma}(Z)$ given by $\tau([A]) = [A^*]$. 

\begin{lem}\label{pci}
The quotient of $\M^{U(1)}_{\sigma}(Z)$ under $\tau$ is the reducible locus $\M^{\Red}_{\sigma}(Z)$. Moreover the set of fixed points of $\tau$ consists of classes of flat connections whose holonomy groups lie in $\{-1, 1\} \subset U(1)$. 
\end{lem}

\begin{proof}
Note that the map $A_L \mapsto A_L \oplus A^*_L$ descends to a surjective map $\M^{U(1)}_{\sigma}(Z) \to \M^{\Red}_{\sigma}(Z)$ by \autoref{rrl}. We also note the Weyl group of $SU(2)$ is $\Z/2$ generated by the matrix representative
\[
\eta=
\begin{pmatrix}
0 & -1 \\
1 & 0
\end{pmatrix}.
\]
Thus any $SU(2)$ gauge transformation of $E$ preserving the splitting $E= \underline{\C} \oplus \underline{\C}$ is either a $U(1)$ gauge transformation or a $U(1)$ gauge transformation multiplied by $\eta$. The effect of multiplying $\eta$ is applying the involution $\tau$. This identifies $\M^{U(1)}_{\sigma}(Z) /\tau = \M^{\Red}_{\sigma}(Z)$. 

Let $[A_L] \in \M_{\sigma}^{U(1)}$ be a fixed point of $\tau$. Then there exists $u \in \G^{U(1)}$ such that $A^*_L = u \cdot A_L$. Let us write $A_L=d + a_L$. Then this is equivalent to 
\[
-a_L = a_L - u^{-1}du.
\]
Since $u^{-1}du$ is a closed $1$-form, we conclude that $d a_L=0$, which implies that $A_L$ is a flat connection. Note that for any loop $\Gamma$ in $Z$, $\Hol_{\gamma} A^*_L$ equals to the conjugate of $\Hol_{\gamma} A_L$. However, since $A_L$ and $A^*_L$ are gauge equivalent, their holonomy coincides with each other. We conclude that the holonomy of $A_L$ has to be real in $U(1)$, which lies in $\{-1, 1\}$. 
\end{proof}

\begin{dfn}
We say a $U(1)$-connection $A_L$ is a central connection if the holonomy group of $A$ lies in $\{\pm 1\} \subset U(1)$. Otherwise we say $A_L$ is non-central. We denote by $\M^{U(1),*}_{\sigma}(Z)$ the non-central part of the moduli space $\M^{U(1)}_{\sigma}(Z)$. 
\end{dfn}

The analysis of the structure of the $U(1)$-moduli space $\M^{U(1)}_{\sigma}(Z)$ is essentially simpler than that of the $SU(2)$-case due to the commutativity of the group $U(1)$. We first point out that one does not need to consider the thickened moduli space, and the center manifold looks as simple as one would hope. To see this, Let us consider the deformation complex of the space of flat $U(1)$ connections over $T^3$ at $\Gamma_L$:
\begin{equation}\label{3cx1}
L^2_l(T^3, i\R) \xrightarrow{-d} L^2_{l-1}(T^3, \Lambda^1 \otimes i\R) \xrightarrow{*d} L^2_{l-2}(T^3, \Lambda^1 \otimes i\R).
\end{equation}
The $U(1)$-version Chern--Simons functional is
\[
\cs^{U(1)}(B_L) = -\frac{1}{2} \int_{T^3} b_L \wedge db_L,
\]
where $B_L= \nabla_L + b_L$, $\nabla_L$ is the product connection, and $b_L \in L^2_l(T^3, \Lambda^1 \otimes i\R)$. The gradient of $\cs^{U(1)}$ is given by 
\[
\grad \cs^{U(1)}|_{B_L} = *d b_L. 
\]
We denote by $H^1_{\Gamma_L}:=\ker d \cap \ker d^* \subset L^2_l(T^3, \Lambda^1 \otimes i\R)$.   We claim that $\Gamma_L + H^1_{\Gamma_L}$ is the center manifold for the pair $(H^1_{\Gamma_L}, -\grad \cs^{U(1)})$, i.e. the center manifold is the graph of the zero map. Indeed, $\Gamma_L + H^1_{\Gamma_L}$ consists of all critical points of $\grad \cs^{U(1)}$ in $\ker d^*$ and is preserved by the gradient flowlines. Moreover, the gradient vector field $\grad \cs^{U(1)}$ is already complete over $H^1_{\Gamma_L}$. So there is no need to consider the thickened moduli space. The argument in \autoref{EPA} implies that all connections in $\M^{U(1)}_{\sigma}(Z)$ have exponential decay to their asymptotic value. We denote by $\mathcal{R}^{U(1)}(T^3):=\Hom(\pi_1(T^3), U(1))$ the space of $U(1)$-representations of $\pi_1(T^3)$, and write the asymptotic map as 
\[
\partial_+: \M^{U(1)}_{\sigma}(Z) \longrightarrow \mathcal{R}^{U(1)}(T^3). 
\]
Since the center manifold is $C^{\infty}$, the asymptotic map is $C^{\infty}$ as well. If we only consider the non-central stratum $\M^{U(1), *}_{\sigma}(Z)$ where the holonomy perturbations are nonzero, the argument of \autoref{TSC} implies that 
\[
\partial_+: \M^{U(1), *}_{\sigma}(Z) \longrightarrow \mathcal{R}^{U(1)}(T^3)
\]
is transverse to any given submanifold with respect to generic perturbations. The computation of the dimension is given by considering the deformation complex at $[A_L] \in \M^{U(1),*}_{\sigma}(Z)$ as in \cite[Chapter 8]{MMR}:
\[
d=-\big(\chi(Z) + \sigma(Z) \big) + {h^1 + h^0 \over 2} + {\rho(\Gamma_L) \over 2},
\]
where $\Gamma_L$ is the asymptotic value of $A_L$, $h^i$ is the dimension of the $i$-the homology of the deformation complex (\ref{3cx1}) at $\Gamma_L$ as above, and $\rho(\Gamma_L)$ is the $\rho$-invariant of the odd signature operator twisted by $\Gamma_L$.  It was mentioned in \cite[Section 16]{GM} that $\rho(\Gamma) = 0$. Since $h^1=3, h^0=1$, we conclude that $d=2$. Combining with the transversality result, we see that the image $\partial_+(\M_{\sigma}^{U(1),*}(Z))$ misses the points in $\mathcal{R}^{U(1)}(T^3)$ whose holonomy groups lie in $\{ \pm 1\}$. We summarize the above discussion as follows. 

\begin{prop}
Let $Z$ be a Riemannian smooth manifold with a cylindrical end modeled on $[0, \infty) \times T^3$ satisfying $H_*(Z; \Z) \cong H_*(D^2 \times T^2; \Z)$. Then given a sub-manifold $\mathfrak{S} \subset \mathcal{R}^{U(1)}(T^3)$, with respect to a generic perturbation $\sigma \in \mathcal{P}_{\mu}$ we have the following description for the non-central moduli space $\M^{U(1),*}_{\sigma}(Z)$.
\begin{enumerate}
\item[\upshape (i)] $\M^{U(1),*}_{\sigma}(Z)$ is an oriented smooth $2$-manifold. 
\item[\upshape (ii)] The asymptotic map $\partial_+: \M^{U(1),*}_{\sigma}(Z) \to \mathcal{R}^{U(1)}(T^3)$ is smooth and transverse to $\mathfrak{S}$. 
\item[\upshape(iii)] The image $\partial_+(\M_{\sigma}^{U(1),*}(Z))$ misses all connections in $\mathcal{R}^{U(1)}(T^3)$ whose holonomy groups lie in $\{ \pm 1\}$. 
\end{enumerate}
\end{prop}

Note that when there are no perturbations, the $U(1)$-moduli space $\M^{U(1)}(Z)$ is the space of gauge equivalence classes of flat $U(1)$-connections on $Z$. The holonomy map identifies $\M^{U(1)}(Z)$ with the space $\mathcal{R}^{U(1)}(Z)$ consisting of $U(1)$-representations of $\pi_1(Z)$, thus is identified as a $2$-torus. Our next step is to show that this feature is preserved under small generic perturbations. 

Note that all connections in $\M^{U(1)}_{\sigma}(Z)$ has exponential decay after a fixed time $T_0$ given by \autoref{rt0}. Moreover the decay rate is given by ${\mu \over 2}$, where $\mu$ is the smallest absolute value of eigenvalues of $*d|_{\ker d^*}$. We fix a constant $\delta$ satisfying $0< \delta < {\mu \over 2}$ as the weight.  Now we can narrow down the ambient connection space to be 
\[
\begin{split}
\A^{U(1)}_{k, \delta}(Z):=\{A_L \in \A^{U(1)}_{k, loc}(Z): & \exists b_L \in \mathcal{H}^1(T^3; i\R) \text{ such that } \\
& A_L- \varphi b_L \in L^2_{k, \delta}(Z, \Lambda^1 \otimes i\R) \},
\end{split}
\]
where $\varphi: Z \to \R$ is the cut-off function in \autoref{d2.17}. The gauge group that preserves $\A^{U(1)}_{k, \delta}(Z)$ is 
\[
\begin{split}
\G_{k+1, \delta}:=\{ u \in \G_{k+1, loc}(Z):& u|_{[T_0 ,\infty) \times T^3} =u_0 \cdot e^{\xi}, \text{ where } u_0 \in S^1 \\
& \xi \in L^2_{k+1, \delta}([T_0, \infty) \times T^3, i\R) \}.
\end{split}
\]
Let $[A_L] \in \M^{U(1)}_{\sigma}(Z)$. The Lie algebra of the gauge group $\G^{U(1)}_{k+1, \delta}(Z)$ is 
\begin{equation}\label{e7.3.1}
\hat{L}^2_{k+1, \delta}(Z, i\R):=\{ \xi \in L^2_{k+1, loc}(Z, i\R): d\xi \in L^2_{k, \delta}(Z, i\R) \}.
\end{equation}\label{e7.3.2}
The tangent space $T_{A_L} \A^{U(1)}_{k, \delta}(Z)$ is 
\begin{equation}
\hat{L}^2_{k, \delta}(Z, \Lambda^1 \otimes i\R):= \{ a_L + \varphi b_L : a_L \in L^2_{k, \delta}(Z, \Lambda^1 \otimes i \R), b_L \in \mathcal{H}^1(T^3; i\R) \}.
\end{equation}
Finally the deformation complex at $A_L$ is 
\begin{equation}\tag{$E^{U(1)}_{\delta, \sigma}$}\label{u1tc}
\hat{L}^2_{k+1, \delta}(Z, i\R) \xrightarrow{-d} \hat{L}^2_{k, \delta}(Z, \Lambda^1 \otimes i\R) \xrightarrow{d^+_{\sigma}} L^2_{k-1, \delta}(Z, \Lambda^+ \otimes i\R),
\end{equation}
where $d^+_{\sigma}:=d^+ - D \sigma_L|_{A_L}$ is the linearization of the perturbed ASD equation at $A_L$. Sitting inside of the complex $E^{U(1)}_{\delta, \sigma}$ is a sub-complex:
\begin{equation}\tag{$F^{U(1)}_{\delta, \sigma}$}\label{u1sc}
L^2_{k+1, \delta}(Z, i\R) \xrightarrow{-d} L^2_{k, \delta}(Z, \Lambda^1 \otimes i\R) \xrightarrow{d^+_{\sigma}} L^2_{k-1, \delta}(Z, \Lambda^+ \otimes i\R).
\end{equation}
When $\sigma=0$, the quotient of $E^{U(1)}_{\delta, \sigma}$ by $F^{U(1)}_{\delta, \sigma}$ is identified as
\[
i\R \xrightarrow{0} \mathcal{H}^1(T^3; i\R) \to 0.
\]
As mentioned in the beginning of this section, the choice of $\delta$ also ensures that the complex $E^{U(1)}_{\delta, \sigma}$ is Fredholm with respect to small or compact perturbations. Now let $[A_L] \in \M^{U(1)}_{\sigma}(Z)$ be a connection with holonomy in $\{\pm 1\}$. In this case, the holonomy of $A = A_L \oplus A^*_L$ is traceless, from which we know $D \sigma_L|_{A_L}=0$. \cite[Proposition 3.12]{M1} identifies $H^2(E^{U(1)}_{\delta}) \cong \hat{H}^+_c(Z; i\R)$, where $\hat{H}^+_c(Z; i\R)$ is the image of $H^+_c(Z; i\R)$ in $H^2(Z; i\R)$ under the inclusion map. Due to the fact that $b^+(Z) =0$, we conclude that $H^2(E^{U(1)}_{\delta})=0$. From this fact we learn two things:
\begin{enumerate}[label=(\alph*)]
\item Each central class $[A_L]$ is a smooth point in $\M^{U(1)}_{\sigma}(Z)$. 
\item The non-perturbed $U(1)$-moduli space $\M^{U(1)}(Z)$ is regular, i.e. It is smoothly cut-out by the defining equation. 
\end{enumerate}

\begin{prop}\label{TRS}
For a generic small perturbation $\sigma$, $\M^{U(1)}_{\sigma}(Z)$ is diffeomorphic to a $2$-torus. 
\end{prop}

\begin{proof}
The proof is similar to the transversality result as in \autoref{TSC}. Let $\sigma \in \mathcal{P}_\mu$ be a  generic perturbation so that $\M^{U(1)}_{\sigma}(Z)$ is regular. Pick a path $\sigma_t$ from $0$ to $\sigma$ in $\mathcal{P}_{\mu}$. Now we consider the map 
\[
\begin{split}
\mathcal{F}^{U(1)}: \mathcal{P}_{\mu} \times \ker d^*_{\delta} &\longrightarrow L^2_{k-1, \delta}(Z, \Lambda^+ \otimes i\R) \\ 
(\sigma, a_L) & \longmapsto d^+_{\sigma} a_L,
\end{split}
\]
where $d^*_{\delta} = e^{-\delta\tau} d^* e^{\delta \tau}: L^2_{k, \delta}(T^*Z \otimes i\R) \to L^2_{k-1, \delta}(Z, i\R)$ is the formal $L^2_{\delta}$-adjoint of $d$. Our discussion above implies that $\mathcal{F}^{U(1)}$ is a submersion. We denote by $\mathcal{Z}^{U(1)} = (\mathcal{F}^{U(1)})^{-1}(0)$. By construction, $\sigma_0$ and $\sigma_1$ are two regular values of the projection map $\pi: \mathcal{Z}^{U(1)}  \to \mathcal{P}_{\mu}$. We approximate the path $\sigma_{t}$ relative to boundary by a generic path $\sigma'_{t}$ transverse to the map $\pi:  \mathcal{Z}^{U(1)}  \to \mathcal{P}_{\mu}$. Then the union 
\[
\mathcal{Z}^{U(1)}_I:=\bigcup_{t \in [0, 1]} \pi^{-1} (\sigma'_t) \cap \mathcal{Z}^{U(1)}
\]
is a cobordism from $\M^{U(1)}(Z)$ to $\M^{U(1)}_{\sigma_1}(Z)$. Since $\sigma_0=0$ is a regular value of $\pi|_{\mathcal{Z}^{U(1)}}$, we conclude that whenever $\|D \sigma_L\|$ is small, $\sigma$ is a regular value as well. Since $\|D\sigma_L\| \leq c. \|\sigma\|_{\mathcal{P}}$, we can choose choose $\|\sigma\|_{\mathcal{P}}$ sufficiently small so that each point $\sigma'_t$ in the path is a regular value of $\pi|_{\mathcal{Z}^{U(1)}}$. Thus the cobordism $\mathcal{Z}^{U(1)}_I$ is a product. This shows that $\M^{U(1)}_{\sigma_{\pmb{\omega}}}(Z)$ is diffeomorphic to $\M^{U(1)}(Z)$ which is a $2$-torus. 
\end{proof}

\begin{cor}\label{RPI}
Given a generic perturbation $\sigma$ whose norm $\|\sigma\|_{\mathcal{P}}$ is sufficiently small, the reducible locus $\M^{\Red}_{\sigma}(Z)$ is identified as a pillowcase, i.e. the quotient of $T^2$ by the hypoelliptic involution. 
\end{cor}

\begin{proof}
\autoref{pci} tells us that $\M^{\Red}_{\sigma}(Z)$ is the quotient of $\M^{U(1)}_{\sigma}(Z)$ under an involution whose fixed point set consists of flat connection on $Z$ with holonomy group inside $\{\pm 1\}$. Since $|b^1(Z; \Z/2)|=4$, there are four of them. Moreover each of them are smooth in $\M^{U(1)}_{\sigma}(Z)$. Now we know $\M^{U(1)}_{\sigma}(Z)$ is a $2$-torus. 
\end{proof}

\subsection{The Kuranishi Picture}

Now we analyze the Kuranishi picture about a reducible instanton $[A]$ inside the moduli space $\M_{\sigma}(Z)$. Let $[A] \in \M^{\Red}_{\sigma}(Z)$ take the form $A=A_L \oplus A_L^*$ with respect to a reduction $E=\underline{\C} \oplus \underline{\C}$. 

\begin{lem}\label{iic}
Any central instanton $[A] \in \M_{\sigma}(Z)$ is isolated from the irreducible locus $\M_{\sigma}^*(Z)$ with respect to small perturbations $\sigma \in \mathcal{P}_{\mu}$. 
\end{lem}

\begin{proof}
When $A$ is central, we know $\sigma(A)=0$ for any perturbation $\sigma \in \mathcal{P}_{\mu}$. Thus \autoref{l2.3} tells us $A$ is actually flat. Then it suffices to prove the result in the non-perturbed case. 

Note that the non-perturbed moduli space $\M(Z)$ is identified with the space of gauge equivalent classes of flat connections. Following the same argument as in the paragraph above \autoref{STAB}, we see the Kuranishi obstruction map at $[A]$ in the $\M(Z)$ is given by 
\begin{equation}
\begin{split}
\mathfrak{o}_A: H^1(Z; \SU(2)) & \longrightarrow H^2(Z; \SU(2)) \\
e^1 \otimes X_1 + e^2 \otimes X_2 & \longmapsto e^1 \wedge e^2 \otimes [X_1, X_2],
\end{split}
\end{equation}
where $\{e^1, e^2\}$ is an orthonormal frame of $\mathcal{H}^1(Z)$. Since the stabilizer of $A$ is $SU(2)$, a neighborhood of $[A]$ in $\M(Z)$ is identified with a neighborhood in the $SU(2)$-quotient $\mathfrak{o}^{-1}_A(0) / SU(2) \simeq \R^2 / (\Z/2)$. This proves that a neighborhood of $[A]$ in $\M(Z)$ is the same as a neighborhood of $[A]$ in the reducible locus $\M^{\Red}(Z)$. Thus $[A]$ is isolated from the irreducible locus.
\end{proof}

Given a generic small perturbation $\sigma$, any instanton $[A] \in \M_{\sigma}(Z)$ asymptotic to a central connection in $\chi(T^3)$ has to be central itself. Moreover \autoref{iic} tells us that the central instantons $[A] \in \M_{\sigma}(Z)$ is isolated from the irreducible locus. Thus we can choose the neighborhood $\mathcal{O}_c$ of central connections in $\chi(T^3)$ such that all irreducible instantons $[A] \in \M^*_{\sigma}(Z)$  has their asymptotic values outside $\mathcal{O}_c$. This in turn enables us to pick a weight $\delta > 0$ uniformly for all irreducible instantons in the Fredholm package. 

Now let $[A] \in \M^{\Red}_{\sigma}(Z)$ be a non-central reducible instanton satisfying $\partial_+([A]) \notin \mathcal{O}_c$. We may write $A=d+a$ with $a \in \hat{L}^2_{k, \delta}(Z, \Lambda^1 \otimes \SU(2))$ and $A_L=d+a_L$ with $a_L \in \hat{L}^2_{k, \delta}(Z,  \Lambda^1 \otimes i\R)$ (recall the space $\hat{L}^2_{k, \delta}$ was defined in (\ref{e7.3.2})). The perturbed deformation complex at $[A]$ is 
\begin{equation}\tag{$E_{\delta, \sigma}$}\label{su2t}
\hat{L}^2_{k+1, \delta}(Z, \SU(2))  \xrightarrow{-d_A} \hat{L}^2_{k, \delta}(Z, \Lambda^1 \otimes \SU(2)) \xrightarrow{d^+_{A, \sigma}} L^2_{k-1, \delta}(Z, \Lambda^+ \otimes \SU(2)),
\end{equation}
where $d^+_{A, \sigma}=d^+_A - D \sigma|_A$. With respect to the isomorphism 
\[
\begin{split}
i\R \oplus \C & \longrightarrow \SU(2) \\
(v, z) & \longmapsto
\begin{pmatrix}
v & z \\
-\bar{z} & -v 
\end{pmatrix}
\end{split}
\]
the induced connection on $\SU(2)$-forms $\Omega^j(Z, \SU(2))$ splits as $d \oplus A_{\C}$ with $A_{\C} = A^{\otimes2}_L$, the holonomy perturbation splits as $\sigma_L \oplus \sigma_{\C}$, and the deformation complex $E_{\delta, \sigma}$ splits as the direct sum of the following two complexes: 
\begin{equation}\tag{$E^{U(1)}_{\delta, \sigma}$}\label{su2tu}
\hat{L}^2_{k+1, \delta}(Z, i\R) \xrightarrow{-d} \hat{L}^2_{k, \delta}(Z, \Lambda^1\otimes i\R) \xrightarrow{d^+_{\sigma}} L^2_{k-1, \delta}(Z, \Lambda^+ \otimes i\R)
\end{equation}
and 
\begin{equation}\tag{$E^{\C}_{\delta, \sigma}$}\label{su2tc}
L^2_{k+1, \delta}(Z, \C) \xrightarrow{-d_{A_{\C}}} L^2_{k, \delta}(Z, \Lambda^1 \otimes \C) \xrightarrow{d^+_{A_{\C}, \sigma}} L^2_{k-1, \delta}(Z, \Lambda^+ \otimes \C), 
\end{equation}
where $d^+_{A_{\C}, \sigma} = d^+_{A_{\C}} - D \sigma_{\C}|_{A_{\C}}$. Since we are working with weighted Sobolev spaces, the homology of the complex can be identified as 
\[
H^0_A(E_{\delta , \sigma})= \ker d_A, \; H^1_A(E_{\delta, \sigma})= \ker d^+_{A, \sigma} \cap \ker d^*_{A, \delta}, H^2_A(E_{\delta, \sigma}) = \ker d^{+, *}_{A, \sigma, \delta},
\]
where 
\[
d^*_{A, \delta} = e^{-\delta \tau} d^*_A e^{\delta \tau}, \; d^{+, *}_{A, \sigma, \delta} = e^{-\delta \tau} d^{+, *}_{A, \sigma} e^{\delta\tau}
\]
are the $L^2_{\delta}$-adjoints. We further note that the complex $E_{\delta, \sigma}$ is $U(1)$-equivariant, where the $U(1)$-action on $\SU(2)$-valued forms are induced by the action on the Lie algebra $\SU(2)$:
\[
e^{i\theta} \cdot (v, z) = (v, e^{i2\theta}z). 
\]
Earlier in \autoref{Kur} we have considered the Kuranishi obstruction map at a flat connection on a $3$-manifold to study the local structure. The same strategy can be applied to the four dimensional case as well.

Following \cite[Theorem 12.1.1]{MMR} there exists a $U(1)$-invariant neighborhood $V_A$ of $0$ in $H^1_A(E_{\delta, \sigma})$ together with a $U(1)$-equivariant map
\[
\mathfrak{o}_A : V_A \to H^2_A(E_{\delta, \sigma})
\]
such that the $U(1)$-quotient of $\mathfrak{o}^{-1}_A(0)$ is isomorphic to a neighborhood of $[A] \in \M_{\sigma}(Z)$ as a stratified space. In particular when $H^1_A(E^{\C}_{\delta, \sigma})=H^2(E^{\C}_{\delta, \sigma})=0$ at $A$, we have $H^2_A(E_{\delta, \sigma})=0$. We see that 
\[
\mathfrak{o}_A^{-1}(0) = H^1_A(E^{U(1)}_{\delta, \sigma}) \simeq H^2(Z; i\R) = i\R \oplus i\R.
\]
Thus $[A]$ is isolated from the irreducible part $\M^*_{\sigma}(Z)$. The next proposition shows this situation fits with all but finitely many reducible instantons $[A] \in \M^{\Red}_{\sigma}(Z)$.

\begin{prop}\label{gph}
With respect to a small generic perturbation $\sigma \in \mathcal{P}_{\mu}$, for all but finitely many non-central reducible instantons $[A] \in \M^{\Red}_{\sigma}(Z)$ satisfying $\partial_+([A]) \notin \mathcal{O}_c$ one has 
\[
H^1_A(E^{\C}_{\delta, \sigma}) = 0.
\]
Moreover $H^1_A(E^{\C}_{\delta, \sigma}) \cong \C$ for the finitely many exceptional reducibles.
\end{prop}

\begin{proof}
Given $A_L \in \A^{U(1)}_{k, \delta}(Z)$, we write $A_{\C}:=A_L^{\otimes 2}$ for the connection on the trivial line bundle $\underline{\C} \to Z$. Let $\bar{A}=\bar{A}_{L} \oplus \bar{A}^*_{L}$ be a non-central flat connection satisfying $\partial_+([\bar{A}]) \notin \mathcal{O}_c$ and $H^1_{\bar{A}}(E^{\C}_{\delta}) \neq 0$. We write $\mathcal{H}^1:=H^1_{\bar{A}}(E^{\C}_{\delta})$, $\mathcal{H}^2:=H^2_{\bar{A}}(E^{\C}_{\delta})$. Note that $i \R = H^0(Z; \ad \bar{A}) = H^0(Z; i\R) \oplus H^0_{\bar{A}}(E^{\C}_{\delta})$. Thus $H^0_{\bar{A}}(E^{\C}_{\delta})=0$. Denote by $\Pi: L^2_{k-1}(Z, \Lambda^+ \otimes \C)  \to \im d^+_{\bar{A}_{\C}}$ the orthogonal projection to the image of $d^+_{\bar{A}_{\C}}$. Note that 
\[
\ind E^{\C}_{\delta} = \ind E_{\delta} - \ind E^{U(1)}_{\delta} = 1 -1 =0
\]
for all $A \in \M^{\Red}(Z)$ with $\partial_+([A]) \notin \mathcal{O}_c$ . We conclude $\dim_{\C} \mathcal{H}^1= \dim_{\C} \mathcal{H}^2$. Let us consider the map
\begin{equation}
\begin{split}
\eta: \mathcal{P}_{\mu} \times \A^{U(1)}_{k, \delta}(Z) \times \ker d^*_{\bar{A}, \delta} & \longrightarrow \im d^+_{\bar{A}_{\C}} \\
(\sigma, A_L, b) & \longmapsto \Pi (d^+_{A_{\C}, \sigma} b). 
\end{split}
\end{equation}
The differential of $\eta$ at $(0, \bar{A}_L, b)$ on the third component is given by 
\[
D\eta|_{(0, \bar{A}_L, b)} (0, 0, \beta) = \Pi (d^+_{\bar{A}_{\C}} b), 
\]
which is surjective. By the implicit function theorem we can find a neighborhood $U \times V \subset \mathcal{P}_{\mu} \times \ker d^*_{\delta}$ of $(0, 0)$ and a map $h: U \times V \times \mathcal{H}^1 \to \ker d^*_{\bar{A}, \delta}$ such that for all  $(\sigma, a_L, b) \in U \times V \times \mathcal{H}^1$ one has 
\[
\eta(\sigma, \bar{A}_L + a_L, b + h(\sigma, a_L, b)) = 0.
\]
In particular $d^+_{A_{\C}, \sigma}(b + h(\sigma, a_L, b)) \in \mathcal{H}^2$. Writing $A_L = \bar{A}_L + a_L$. This leads us to a map 
\begin{equation}
\begin{split}
\xi: U \times V & \longrightarrow L^2_{k-1, \delta}(Z, \Lambda^+\otimes i\R) \times \Hom_{\C}(\mathcal{H}^1, \mathcal{H}^2) \\
(\sigma, a_L) & \longmapsto \big(F^+_{A_L} - \sigma_L(A_L), b \mapsto d^+_{A_{\C}, \sigma}(b + h(\sigma, a_L, b)) \big)
\end{split}
\end{equation}
We write $\xi = \xi_1 \times \xi_2$ for its decomposition into the two factors in its range. The argument in \autoref{TSC} implies that $\xi_1$ is a submersion. Since the loops $q_{\alpha}$ constructed in the holonomy perturbation are dense at each point, Proposition 65 in \cite{H94} implies that we only need to vary finitely many components in $\pmb{\omega}=\{\omega_{\alpha}\}$ to ensure that $\xi_2$ is a submersion. Thus we conclude the map $\xi$ is a submersion. Let $S_i \subset \Hom_{\C}(\mathcal{H}^1, \mathcal{H}^2)$ be the stratum consisting of linear maps of complex dimension-$i$ kernel (one may consult Koschorke's work \cite{K70} on the stratification of Fredholm maps between Hilbert spaces while here we are using a finite-dimensional version). Then the projection map $\pi: \xi^{-1}( \{0 \} \times S_i) \to U$ is Fredholm of real index $2-2i^2$. By the Sard--Smale theorem, for a generic perturbation $\sigma \in U$ only the top two strata $S_0$ and $S_1$ survive in the image of $\xi|_{\{\sigma\} \times V}$, which corresponds to $[A] \in \M^{\Red}_{\sigma}(Z)$ having $\dim_{\C} H^1_A(E^{\C}_{\delta, \sigma})=0, 1$ respectively. Moreover the set of connections $A$ attaining $\dim_{\C} H^1_A(E^{\C}_{\delta, \sigma}) = 1$ is discrete in $V$. 

Since $\M^{\Red}(Z)$ is compact, we only need to run the argument finitely many times to get generic small perturbations $\sigma \in \mathcal{P}_{\mu}$ so that for all but finitely many reducible instantons $[A] \in \M^{\Red}_{\sigma}(Z)$, the corresponding cohomology satisfies $H^1_A(E^{\C}_{\delta, \sigma}) = 0$, while $H^1_A(E^{\C}_{\delta, \sigma}) = \C$ for the finite exceptions.
\end{proof}

\begin{rem}\label{gph1}
The only reason we impose the condition that $\partial_+{[A]} \notin \mathcal{O}_c$ is to ensure the complex $E_{\delta, \sigma}$ is Fredholm. Since we know all central instantons are isolated from the irreducible part $\M^*(Z)$, any reducible instanton $[A] \in \M^{\Red}(Z)$ with $\partial_+{[A]} \in \mathcal{O}_c$ is also isolated from $\M^*(Z)$ once we choose $\mathcal{O}_c$ small enough. In the language of representation variety, this is equivalent to the vanishing of the twisted cohomology $H^1(Z; \C_{A_{\C}}) =0$. This property is also preserved under small perturbations. However in the perturbed case we need to vary the weight $\delta$ to define the cohomology. 
\end{rem}

With the help of \autoref{gph}, we have the following description of a neighborhood of the reducible locus $\M^{\Red}_{\sigma}(Z)$ in the total moduli space $\M_{\sigma}(Z)$.

\begin{prop}\label{nbr}
Given a small generic perturbation $\sigma \in \mathcal{P}_{\mu}$, all but finitely many reducible instantons $[A] \in \M^{\Red}_{\sigma}(Z)$ are isolated from the irreducible moduli space $\M^*_{\sigma}(Z)$. Moreover any reducible instanton $[A]$ not isolated from $\M^*_{\sigma}(Z)$ is non-central, and has a neighborhood $U_{[A]}$ in $\M_{\sigma}(Z)$ such that $U_{[A]} \cap \M^*_{\sigma}(Z) \simeq [0, \epsilon)$ for some $\epsilon >0$. 
\end{prop}

\begin{proof}
Let $[A] \in \M_{\sigma}(Z)$ have $H^1_A(E^{\C}_{\delta, \sigma})=\C$. From \autoref{gph} and \autoref{gph1} there are only finitely many such instantons, all of which are non-central. The irreducible part of a neighborhood $U_{[A]}$ is identified with the $U(1)$-quotient of $\mathfrak{o}_A^{-1}(0) \cap V_A$, where $V_A \subset H^1_A(E_{\delta, \sigma})$ is a neighborhood of the origin. We identify $H^1_A(E_{\delta, \sigma}) \cong i\R \oplus i\R \oplus \C$, and $H^2(E_{\delta, \sigma}) \cong \C$ so that the $U(1)$-action are given respectively by 
\[
e^{i\theta} \cdot (x_1, x_2, z) = (a, b, e^{2i\theta}z), \; \; e^{i\theta} \cdot w = e^{2i\theta}w.
\]
To get a better understanding of what $\mathfrak{o}_A$ looks like, we recall its construction as follows. One first considers the map 
\begin{equation}
\begin{split}
\ker d^*_{A, \delta} & \longrightarrow \im d^+_{A, \sigma} \\ 
a & \longmapsto \Pi(F^+_{A+a} - \sigma(A+a)),
\end{split}
\end{equation}
where $\Pi: L^2_{k-1, \delta}(\Lambda^+T^*Z \otimes \SU(2)) \to \im d^+_{A, \sigma}$ is the $L^2_{\delta}$ orthogonal projection onto the image of $d^+_{A, \sigma}$. Since this map is a submersion, the implicit function theorem gives us a function $\mathfrak{q}_A: V_A \to \im d^{+, *}_{A, \sigma, \delta}$ so that $A+a+\mathfrak{q}_A(a) \in H^2(E_{\delta, \sigma})$. Then we let 
\[
\mathfrak{o}_A(a):=F^+_{A+a+\mathfrak{q}_A(a)} - \sigma(A+a+\mathfrak{q}_A(a)). 
\]
Note that $\mathfrak{o}_A$ is analytic and vanishes at least up to second order by the virtue of its construction. Thus the $U(1)$-equivariance forces the Kuranishi map to take the following form 
\[
\mathfrak{o}_A(x_1, x_2, z) =f(x_1, x_2, |z|) \cdot z. 
\] 
where $f: i\R \oplus i\R \oplus \C \to \C$ vanishes at least up to first order. We further write 
\[
f(x_1, x_2, |z|) := \sum_{i \geq 0} f_i(x_1, x_2)|z|^i. 
\]
We note that $\mathfrak{q}_A$ vanishes at least to second order. Thus the second order term of $\mathfrak{o}_A$ at $0$ is given by $Dd^+_{A, \sigma}|_0$, which is non-vanishing due to the transversality in \autoref{gph}. So up to an orientation-preserving change of coordinates, we may take $f_0(x_1, x_2)=x_1 \pm ix_2$. Since the complex $E^{\C}_{\delta, \sigma}$ is complex linear, we know that $\mathfrak{o}_A(0, 0, z) = f(0, 0, |z|) \cdot z$ is complex linear. Thus $f_i(0, 0)=0$ for all $i \geq 1$. It now follows that the zero set of $\mathfrak{o}_A$ is given by 
\[
\mathfrak{o}_A^{-1}(0) = \{x_1=x_2=0\} \cup \{z=0\}.
\]
We then conclude that the normal part is identified with $\C / U(1) \simeq [0, \infty)$. 
\end{proof}

\begin{dfn}\label{bfp}
Any reducible instanton $[A] \in \M^{\Red}_{\sigma}(Z)$ in \autoref{gph} satisfying $H^1_A(E^{\C}_{\delta, \sigma}) = \C$ is called a bifurcation point of $\M^{\Red}_{\sigma}(Z)$. 
\end{dfn}

\subsection{The Orientation}
At the end of this section, we discuss how we orient the perturbed moduli spaces. Formally an orientation of the moduli space $\M_{\sigma}(Z)$ is a trivialization of the determinant line of the index bundle associated with the deformation complex parametrized by connections in $\M_{\sigma}(Z)$. As we noted above, one cannot choose a uniform weight $\delta$ so that the deformation complex $E_{\delta, \sigma}$ is Fredholm for all instantons $[A] \in \M_{\sigma}(Z)$. So we only orient the portion of the moduli space that makes $E_{\delta, \sigma}$ Fredholm.  

Choose a weight $\delta$ and a neighborhood $\mathcal{O}_c \subset \chi(T^3)$ of the central connections as before. We write 
\[
\M_{\sigma}(Z, \mathcal{O}_c^c):=\{ [A] \in \M_{\sigma}(Z) : \partial_+[A] \notin \mathcal{O}_c\}
\]
for the portion of the moduli space consisting of instantons $[A]$ that are not asymptotic to an element in $\mathcal{O}_c$. We first orient the unperturbed reducible locus $\M^{\Red}(Z)$ as follows. Note the unperturbed deformation complex $E^{U(1)}_{\delta, \sigma}$ is independent of $[A] \in \M^{\Red}(Z)$, thus the corresponding index bundle is trivialized automatically once we fix a trivialization at a single point. According to \cite[Proposition 3.12]{M1}, its determinant is identified with 
\[
\det \Ind(E^{U(1)}_{\delta}) = \Lambda^{\max} H^0(Z; i\R)^* \otimes \Lambda^{\max} H^1(Z; i\R). 
\] 
Following the path $(E^{U(1)}_{\delta, t\sigma})$, $t \in [0, 1]$, the orientation on $\Ind(E^{U(1)}_{\delta})$ can be transported to orient the perturbed index bundle $\Ind (E^{U(1)}_{\delta, \sigma})$ for all small perturbations $\sigma$. The $SU(2)$ deformation complex $E_{\delta, \sigma}$ splits into the direct sum of two complexes $E^{U(1)}_{\delta, \sigma}$ and $E^{\C}_{\delta, \sigma}$ at a reducible instanton $[A] \in \M^{\Red}_{\sigma}(Z, \mathcal{O}_c^c)$. The complex structure on $E^{\C}_{\delta, \sigma}$ provides us with a canonical orientation. Combining with an orientation on $E^{\C}_{\delta, \sigma}$, we get a trivialization of $\det \Ind(E_{\delta, \sigma})$ on the reducible locus $\M^{\Red}_{\sigma}(Z, \mathcal{O}_c^c)$.

Now we discuss how we orient the irreducible moduli space $\M^*_{\sigma}(Z)$. Let $[A] \in M^*_{\sigma}(Z)$. Recall that we only allow small perturbations $\sigma$ so that $\partial_+[A] \notin \mathcal{O}_c$. Given any $[\Gamma] \notin \mathcal{O}_c$, \cite[Section 8.8]{MMR} tells us that the determinant line of $\Ind(-d^*_{A, \delta} \oplus d^+_A)$ over the moduli space $\B_{\delta}(Z, [\Gamma])$ consisting of gauge equivalence classes of connections on $Z$ that are asymptotic to $[\Gamma]$ exponentially with rate $-\delta$ is trivial. Since the complement $\mathcal{O}^c_c = \chi(Y) \backslash \mathcal{O}_c$ is simply-connected, the corresponding determinant line over $\B_{\delta}(Z, \mathcal{O}^c_c)$ is trivial as well. Since $\M_{\sigma}(Z, \mathcal{O}^c_c) \subset \B_{\delta}(Z, \mathcal{O}^c_c)$, then the trivialization of the determinant line over $\M^{\Red}(Z, \mathcal{O}^c_c)$ gives rise to a trivialization of the determinant line over $\M^*_{\sigma}(Z) \subset \B_{\delta}(Z, \mathcal{O}^c_c)$

To conclude, an orientation of $H^0(Z; i\R)^* \oplus H^1(Z; i\R)$ induces an orientation on $\M_{\sigma}(Z, \mathcal{O}_c^c)$ for all small perturbations $\sigma$. The orientation of $Z$ induces an orientation on $H^0(Z; i\R)$. We fix an orientation on $H^1(Z; i\R) \cong i\R \oplus i\R$ which is referred to as a homology orientation. Moreover to each bifurcation point $[A] \in \M^{\Red}_{\sigma}(Z)$ we assign a sign as follows.

\begin{dfn}
Let $[A] \in \M^{\Red}_{\sigma}(Z)$ be a bifurcation point as in \autoref{gph}. We assign $+1$ (resp. $-1$) to $[A]$ if $f_0: i\R \oplus i\R \to \C$ is orientation-preserving (resp. orientation-reversing), where $f_0$ is given by the Kuranishi obstruction map $\mathfrak{o}_A$ in the proof of \autoref{gph}. 
\end{dfn}

\begin{rem}
Since the local structure the moduli space near a bifurcation $[A]$ is modeled on $\mathfrak{o}_A^{-1}(0)$, the `$+1$' assignment describes the case when the path of irreducible instantons is pointing away from $[A]$, and the `$-1$' assignment corresponds to the case when the path is pointing into $[A]$. 
\end{rem}

\section{The Surgery Formula}\label{sffo}

\subsection{The Set-Up}

We start by giving an explicit description of the surgery operation. Let $X$ be an admissible integral homology $S^1 \times S^3$, and $\mathcal{T} \hookrightarrow X$ an embedded torus inducing a surjective map on first homology, i.e. the map $H_1(\mathcal{T}; \Z) \to H_1(X; \Z)$ given by the inclusion is surjective. We fix a generator $1_X \in H^1(X; \Z)$ serving as a homology orientation. We fix a framing of $\mathcal{T}$ by choosing an identification $\nu(\mathcal{T}) \cong D^2 \times T^2$. We write 
\[
\mu=\partial D^2 \times \{pt.\} \times \{pt.\}, \lambda=\{pt.\} \times S^1 \times \{pt.\}, \gamma=\{pt.\} \times \{pt.\} \times S^1. 
\]
Let us denote by $M:=\overline{X \backslash \nu(\mathcal{T})}$ the closure of the complement of the tubular neighborhood. Then we have $H_*(M; \Z) \cong H_*(D^2 \times T^2; \Z)$. We require the framing is chosen so that $[\lambda]$ generates $\ker \big(H_1(\partial M; \Z) \to H_1(M; \Z) \big)$ and $1_X \cdot [\gamma]=1$. Under this choice, the isotopy class of $\mu$ and $\lambda$ are fixed, but there is still ambiguity in choosing $\gamma$ which we will allow. Since the diffeomorphism type of the surgered manifold $X_{p,q}=M \cup_{\varphi_{p,q}} D^2 \times T^2$ is determined by the isotopy class of $[\varphi_{p,q}(\mu)] = p[\mu] + q[\lambda]$, the surgery operation is well-defined despite the framing ambiguity.  

Note that only when $p=1$ can the $(p,q)$-surgered manifold have the same homology as that of $S^1 \times S^3$. For simplicity, we write 
\begin{equation}
X_q=X_{1, q} \text{ and } \varphi_q=\varphi_{1, q}, \quad q \neq 0.
\end{equation}
We also write $N=D^2 \times T^2$ and identify $T^3=\partial M=-\partial N$. In this way,  $X_q = M \cup_{\varphi_q} N$ for $q \neq 0$. Since the gluing map $\varphi_q$ preserves $[\gamma]$ for all $q$, we abuse the notation $[\gamma] \in H_1(X_q; \Z)$ as a chosen generator. To define the Furuta--Ohta invariant over $X_q$, we need the following observation. 

\begin{lem}\label{l8.1}
With the assumptions above, $X_q$ is admissible for $q \neq 0$. 
\end{lem}
\begin{proof}
 Let $q \neq 0$. Any representation $\rho: \pi_1(X_q) \to U(1)$ is determined by the image $\rho([\gamma]) \in U(1)$ of the generator $[\gamma]$. So every representation on $X_q$ comes from one on $X$. Consider the following portion of Mayer--Vietoris sequence:
\begin{equation}
0 \to H^1(X_q; \C_{\rho}) \to H^1(M; \C_{\rho}) \oplus H^1(N; \C_{\rho}) \xrightarrow{j_q} H^1(T^3; \C_{\rho}),
\end{equation}
where $j_q(\alpha, \beta) = \alpha|_{\partial M} - \varphi^*_q(\beta|_{\partial N})$. Since $H^1(X, \C_{\rho})=0$, we conclude that $\forall \alpha \in H^1(M, \C_{\rho})$, $\beta \in H^1(N, \C_{\rho})$
\[
\alpha|_{\partial M} - \beta|_{\partial N} = 0 \Longleftrightarrow \alpha=0, \; \beta=0.
\]
Denote by $r_M: H^1(M; \C_{\rho}) \to H^1(T^3; \C_{\rho})$ and $r_N: H^1(N; \C_{\rho}) \to H^1(T^3; \C_{\rho})$ the restriction map. Note that $\im r_M \cap \im r_N  =\im r_M \cap \im \varphi^*_q \circ r_N.$ Thus $j_q(\alpha, \beta) = 0$ implies that $\exists \beta' \in H^1(N; \C_{\rho})$ such that $\beta'|_{\partial N}  = \varphi^*_q(\beta|_{\partial N})$, which further implies that $\alpha=0, \beta'=0$, thus $\beta=0$. This shows that $H^1(X_q; \C_{\rho})=0$. 
\end{proof}

Let $E=\C^2 \times X$ be a trivialized $\C^2$-bundle over an admissible integral homology $S^1 \times S^3$. We denote by $\M_{\sigma}(X)$ the moduli space of perturbed ASD $SU(2$)-connections on $E$. The vanishing of the twisted first homology ensures that the reducible locus $\M^{\Red}_{\sigma}(X)$ is isolated from the irreducible locus $\M^*_{\sigma}(X)$ for all small perturbations. Moreover the irreducible locus $\M^*_{\sigma}(X)$ is an oriented compact $0$-manifold. The Furuta--Ohta invariant \cite{RS1} is defined to be the signed count of irreducible instantons under a generic small perturbation:
\[
\lambda_{FO}(X):= {1 \over 4} \#\M^*_{\sigma}(X).
\]

Recall over the $(0, 1)$-surgered manifold $X_{0,1}$, we are considering the $SO(3)$-bundle $E_0$ characterized by 
\[
p_1(E_0)=0 \quad \text{ and } \quad w_2(E_0) = \omega_{\mathcal{T}},
\]
where $\omega_{\mathcal{T}} \in H^2(X_{0,1}; \Z/2)$ is Poincaré dual to the core $\{0\} \times T^2$ of the gluing $D^2 \times T^2$. 

\begin{lem}\label{l8.2}
Every (unperturbed) ASD connection on $E_0$ is flat and irreducible. 
\end{lem}

\begin{proof}
The Chern--Weil formula tells us any connection $A$ on $E_0$ satisfies $\|F^+_A\|_{L^2} = \|F^-_A\|_{L^2}$. So every ASD connection $E_0$ is flat. 

Let us write $\mathcal{T}_0 \subset X_{0, 1}$ for the core $\{0\} \times T^2$ in the gluing copy $D^2 \times T^2$. Since $w_2(E_0) = \omega_{\mathcal{T}}$ is Poincaré dual to $\mathcal{T}_0$, on the complement $X_{0,1} \backslash \mathcal{T}_0$ the bundle $E_0$ lifts to an $SU(2)$-bundle $\tilde{E}_0$. Correspondingly, each flat $SO(3)$-connection $A$ on $E_0$ lifts to a flat $SU(2)$-connection $\tilde{A}$ on $\tilde{E}_0$ whose holonomy around the meridian of $\mathcal{T}_0$ is $-1$. 

The meridian $\mu = \partial D^2 \times \{pt.\}$ is null-homologous in the complement $X_{0,1} \backslash \mathcal{T}_0$, thus bounds a surface, say $\Sigma$. Suppose the lifted connection $\tilde{A}$ is reducible. The holonomy gives a representation $\Hol_{\tilde{A}}: \pi_1(X_{0,1} \backslash \mathcal{T}_0) \to U(1)$ so that $\Hol_{\tilde{A}}(\mu) = -1$. If $g(\Sigma) = 0$, $\mu$ bounds a disk in $\Sigma$, which is impossible. If $g(\Sigma) > 0$, we can write $[\mu] = \Pi_{i=1}^g [a_i, b_i]$ for a standard basis $(a_1, b_1, ..., a_g, b_g)$ of $\pi_1(\Sigma)$. Since $U(1)$ is abelian, we must have $\Hol_{\tilde{A}}(\mu) = 1$, which gives us a contradiction. It follows that $\tilde{A}$ is irreducible, so is $A$. 
\end{proof}

\begin{cor}
The count $D^0_{\omega_{\mathcal{T}}}(X_{0, 1})$ of irreducible ASD connections on $E_0$ is independent of the choices of generic small perturbations and metrics on $X_{0,1}$. 
\end{cor}

\begin{proof}
\autoref{l8.2} implies that given a metric $g$, one can find a constant $\epsilon_g$ so that whenever $\|\sigma\|_{\mathcal{P}} < \epsilon_g$, the perturbed moduli space $\M_{\sigma}(X_{0,1}, g)$ contains no reducible instantons. Otherwise we get a sequence of reducible ASD connections in $\M_{\sigma_n}(X_{0,1}, g)$ with $\|\sigma_n\|_{\mathcal{P}} \to 0$, which converge to a unperturbed reducible ASD connection $E_0$, after taking a subsequence. 

So we use generic perturbations with norm smaller than $\epsilon_g$ to define the count $D^0_{\omega_{\mathcal{T}}}(X_{0,1})$ for $(X_{0,1}, g)$. Given a path of metrics $(g_t)_{t \in [0,1]]}$, a generic path of perturbations $\sigma_t$ with $\|\sigma_t\|_{\mathcal{P}} < \epsilon_{g_t}$ gives us a cobordism between $\M_{\sigma_0}(X_{0,1}, g_0)$ and $\M_{\sigma_1}(X_{0,1}, g_1)$, which establishes the independence on the choices of metrics. 
\end{proof}

Instead of proving \autoref{surq} directly, we prove the following special case when $q=1$. 

\begin{thm}\label{sur1}
After fixing appropriate homology orientations, one has 
\[
\lambda_{FO}(X_1) = \lambda_{FO}(X) + D^0_{w_{\mathcal{T}}}(X_{0,1}).
\]
\end{thm}

We explain why the special case is sufficient. Let us denote by $\mathcal{T}_q \subset X_q$ the image of the core $\{0\} \times T^2 \subset D^2 \times T^2$ in $X_q$ after the surgery performed. Then with respect to the framing of $\mathcal{T}_q$ given by the gluing copy $D^2 \times T^2$, $(1,1)$-surgery along $\mathcal{T}_q$ results in $X_{q+1}$, and $(0,1)$-surgery along $\mathcal{T}_q$ results in $X_{0,1}$. Thus \autoref{surq} is derived by applying \autoref{sur1} repeatedly.  

\subsection{Neck-Stretching and Gluing}

The proof of the surgery formula is based on a neck-stretching argument which we now set up. Let $X$ be an admissible integral homology $S^1 \times S^3$ which decomposes into two pieces $X = M \cup N$ along a $3$-torus oriented in the way that $T^3 = \partial M = -\partial N$. We further assume that the integral homology of $M$ and $N$ are both isomorphic to $H_*(T^2 \times D^2; \Z)$. We identify a tubular neighborhood of $T^3$ in $X$ as $(-1, 1) \times T^3$. Let $h$ be a flat metric on $T^3$. We pick a metric $g$ on $X$ so that 
\[
g|_{(-1, 1) \times T^3} = dt^2 + h.
\]
Given $T > 0$, we stretch the neck $(-1, 1) \times T^3$ of $X$ to obtain $(X_T, g_T)$:
\[
X_T = M \cup [-T, T] \times T^3 \cup N,
\]
where $\{-T\} \times T^3$ is identified with $\partial M$, and $\{T\} \times T^3$ is identified with $\partial N$. We shall write $M_T = M \cup [0, T] \times T^3$ and $N_T = [-T, 0] \times T^3 \cup N$. Then it’s natural to identify $X_T = M_T \cup N_T$. The geometric limits are denoted by $M_o:=M \cup [0, \infty) \times T^3$, $N_o:=(- \infty, 0] \times T^3 \cup N$, and $X_o = M_o \cup N_o$. 

The holonomy perturbations $\sigma_T$ on $X_T$ are constructed as follows. Each perturbation consists of three parts $\sigma_T =\bar{\sigma}_M + \bar{\sigma}_T + \bar{\sigma}_N$. $\bar{\sigma}_M$ and $\bar{\sigma}_N$ are perturbations of the first type considered in \autoref{s3}. So we have $\supp \bar{\sigma}_M \subset M \cup [0, 1] \times T^3$ and $\supp \bar{\sigma}_N \subset [-1, 0] \times T^3 \cup N$. The middle term arises as cylinder functions as the perturbations over then end: 
\begin{equation}
\bar{\sigma}_T(A) := \delta_T(t) \cdot \left(dt \wedge \sum_{\alpha} \pi_{\alpha} \grad \tau_{\alpha}(A|_{\{t\} \times Y}) \right)^+,
\end{equation}
where $\delta_T: (-T-1, T+1) \to \R$ is a smooth function satisfying 
\[
\delta_T|_{(-T-1, -T-1+\epsilon]} \equiv 0, \quad \delta_T|_{[-T, -\epsilon]} = e^{-\mu (t+T)}, \quad \delta_T(t) = \delta_T(-t),
\]
where $\epsilon \ll 1$ is a fixed positive number. After a normalization, the limit of $\delta_T$, as $T \to \infty$, consists of two parts $\delta_o:= \delta_+ + \delta_-$, where $\delta_+: (-1, \infty) \to \R$ is the function we used in \autoref{d3.6}, and $\delta_-: (-\infty, 1) \to \R$ is given by $\delta_-(t) = \delta_+(-t)$. In this way, we get a perturbation over the end of $M_o$:
\[
\bar{\sigma}_+(A):= \delta_+(t) \cdot \left(dt \wedge \sum_{\alpha} \pi_{\alpha} \grad \tau_{\alpha}(A|_{\{t\} \times Y}) \right)^+
\]
Similarly, we get a perturbation on $N_o$, denoted by $\bar{\sigma}_-$. The $L^2_{k,loc}$-limit of $\sigma_T$ then becomes 
\[
\sigma_o:= \sigma_M + \sigma_N = (\bar{\sigma}_M + \bar{\sigma}_+) + (\bar{\sigma}_- + \bar{\sigma}_N).
\]
When $T=0$, we simply write $\sigma$ for the perturbation over the original manifold $X$. After fixing the smooth functions $\delta_T$, $\sigma$ determines $\sigma_T$ for each $T \in [0, \infty)$, and vice versa. Since $\sigma_M \in \mathcal{P}_{\mu}(M_o)$ and $\sigma_N \in \mathcal{P}_{\mu}(N_o)$, we can regard a perturbation $\sigma$ over $X$ as given by the fiber product of $\sigma_M$ and $\sigma_N$ if the $\pmb{\pi}$-coefficients corresponding to the cylindrical end part of $\sigma_M$ and $\sigma_N$ coincide. We denote the space of perturbations over $X$ arising this way by $\mathcal{P}_{\mu}(X)$, equipped with the norm $\|\sigma\|_{\mathcal{P}} = \|\sigma_M\|_{\mathcal{P}} + \|\sigma_N\|_{\mathcal{P}}$. Then the transversality result in \cite{K04} implies that one can use small generic perturbations in $\mathcal{P}_{\mu}(X)$ to define the Furuta--Ohta invariant of $X$. 

With perturbations settled down, we discuss the convergence of instantons with respect to the neck-stretching process. We write 
\begin{equation}
\M_{\sigma_o}(X_o):= \M_{\sigma_M}(M_o) \times_{\partial} \M_{\sigma_N}(N_o)
\end{equation}
for the fiber product of charge-zero instantons on $M_o$ and $N_o$ under the asymptotic maps $\partial_+: \M_{\sigma_M}(M_o) \to \chi(Y)$ and $\partial_-: \M_{\sigma_N}(N_o) \to \chi(Y)$. The reducible locus is defined as 
\begin{equation}
\M^{\Red}_{\sigma_o}(X_o):= \M^{\Red}_{\sigma_M}(M_o) \times_{\partial} \M^{\Red}_{\sigma_N}(N_o).
\end{equation}
The complement is the irreducible part $\M^*_{\sigma_o}(X_o):=\M_{\sigma_o}(X_o) \backslash \M^{\Red}_{\sigma_o}(X_o)$. We introduce the following notion of convergence. 

\begin{dfn}
Let $\{T_n\}$ be a divergent increasing sequence in $[0, \infty)$. We say a sequence of instantons $[A_n] \in \M_{\sigma_{T_n}}(X_{T_n})$ converges to an instanton pair $([A_M], [A_N]) \in \M_{\sigma_o}(X_o)$ if the following hold. 
\begin{enumerate}
\item One can choose representatives $A_n$, $A_M$, and $A_N$ so that 
\[
A_n|_{M_{T_n}} \xrightarrow{L^2_{k, loc}} A_M, \qquad A_n|_{N_{T_n}} \xrightarrow{L^2_{k, loc}} A_N.
\]
\item For any $\epsilon > 0$, there exists $T_{\epsilon} > 0$ and $N_{\epsilon} \in \mathbb{N}$ so that for any $n > N_{\epsilon}$ one can choose representatives $A_n$ so that 
\[
\| A_n|_{[T_{\epsilon} - T_n, T_n - T_{\epsilon}] \times Y} - A_{\Gamma}\|_{L^2_k([T_{\epsilon} - T_n, T_n - T_{\epsilon}] \times Y)} < \epsilon
\]
for some $[\Gamma] \in \chi(Y)$ independent of $n$. 
\end{enumerate}
\end{dfn}

With respect to such a convergence notion, we get a topology over the family of moduli spaces:
\[
\M_{I, \sigma}(X):= \bigcup_{T \in [0, \infty)} \M_{\sigma_T}(X_T) \cup \M_{\sigma_o}(X_o),
\]
which decomposes into the irreducible and reducible strata correspondingly:
\[
\M^*_{I, \sigma}(X):= \bigcup_{T \in [0, \infty)} \M^*_{\sigma_T}(X_T) \cup \M^*_{\sigma_o}(X_o),
\]
\[
\M^{\Red}_{I, \sigma}(X):= \bigcup_{T \in [0, \infty)} \M^{\Red}_{\sigma_T}(X_T) \cup \M^{\Red}_{\sigma_o}(X_o).
\]

\begin{thm}\label{t8.3}
With respect to small generic perturbations $\sigma \in \mathcal{P}_{\mu}$, the irreducible stratum $\M^*_{I, \sigma}(X)$ forms a smooth orientable compact $1$-manifold. Moreover, 
\[
\M^*_{\sigma_o}(X_o) = \M^*_{\sigma_M}(M_o) \times_{\partial} \M^{\Red}_{\sigma_N}(N_o) \cup \M^{\Red}_{\sigma_M}(M_o) \times_{\partial} \M^*_{\sigma_N}(N_o). 
\]
\end{thm}

\begin{proof}
Appealing to \autoref{TIS} and \autoref{nbr}, we can find generic perturbations so that the asymptotic maps $\partial_+: \M_{\sigma_M}(M_o) \to \chi(Y)$ and $\partial_-: \M_{\sigma_N}(N_o) \to \chi(Y)$ are transverse to each other away from a prefixed small neighborhood of the central connections. 

We claim that $\M_{I, \sigma}(X)$ is compact. For the non-perturbed case, one may consult \cite[Chapter 6]{MMR}. To see why, let $[A_n] \in \M_{\sigma_{T_n}}(X_{T_n})$ be a sequence in $\M_{I, \sigma}(X)$. If the sequence $\{T_n\}$ is bounded above, \autoref{l5.14} implies that one can find a convergent subsequence of $[A_n]$. Otherwise, we can pick up a divergent increasing subsequence, still denoted by $\{T_n\}$. Since we work with the trivial bundle over $X_{T_n}$, the instanton charges of $\{[A_n]\}$ are all zero. Then the compactness reduces to the case in \autoref{p5.13}. 

Since $X$ is admissible, the intersection $\partial^+(\M^{\Red}(M_o)) \cap \partial^-(\M^{\Red}(N_o))$ of unperturbed moduli spaces, which corresponds to the reducible connections on $X$, is isolated from the image of bifurcation points from $\M_{\sigma_M}(M_o)$ and $\M_{\sigma_N}(N_o)$ (recall bifurcation points are defined in \autoref{bfp}). Since small perturbations preserve this property, it follows that the irreducible stratum $\M^*_{\sigma_o}(X_o)$ is compact. Since the dimension of both irreducible strata $\M^*_{\sigma_M}(M_o)$ and $\M^*_{\sigma_N}(N_o)$ is $1$ (see \autoref{TSC}) while $\dim \chi(Y) =3$, transversality of $\partial_+$ and $\partial_-$ implies that one factor of $\M^*_{\sigma_o}(X_o)$ is irreducible and the other is reducible. 

The reason why $\M^*_{I, \sigma}(X)$ is an orientable $1$-manifold follows from the standard gluing result in Yang--Mills theory (cf. \cite{M88}, \cite[Theorem I.2.2]{MM93}, \cite[Proposition 5]{BD95}). More precisely, one can fix a precompact subset $U \subset \chi(Y)$ away from central connections so that $\partial_+(\M^*_{\sigma_M}(M_o))$ and $\partial_-(\M^*_{\sigma_N}(N_o))$ are both contained in $U$ with respect to any small perturbations $\sigma$. Choosing generic $\sigma$, one can arrange that $\partial_+$ and $\partial_-$ intersects transversely inside $U$ when restricted to the irreducible strata. Then \cite[Theorem I.2.2]{MM93} and \cite[Proposition 5]{BD95} say that for all sufficiently large $T$, there is a gluing map 
\[
\mathfrak{gl}_T: \M^*_{\sigma_o}(X_o, U) \longrightarrow \M^*_{\sigma_T}(X_T)
\]
which gives rise to an orientation-preserving diffeomorphism, where the domain $\M^*_{\sigma_o}(X_o, U)$ is the subspace of $\M^*_{\sigma_o}(X_o)$ obtained by taking fiber product over $U$. Then the argument in \cite[Theorem 19.2.8]{KM07} implies that such diffeomorphisms fit together to make the $\infty$-end of $\M^*_{I, \sigma}(X)$ the boundary of a $1$-manifold. 
\end{proof}

\subsection{The Proof of \autoref{sur1}}

Recall $(X, g)$ is an admissible integral homology $S^1 \times S^3$ decomposed as $X=M \cup_{T^3} N$, where $N= \nu(\mathcal{T})$ is a tubular neighborhood of an embedded torus $\mathcal{T} \subset X$. We have framed $N \cong D^2 \times T^2$ with a basis $\{\mu, \lambda, \gamma\}$ on $\partial N = -T^3$. The surgered manifolds $X_1$ and  $X_{0,1}$ are given respectively by 
\[
X_1 = M \cup_{\varphi_1} N \text{ and } X_{0,1}= M \cup_{\varphi_0} N, 
\]
where under the basis $\{ \mu, \lambda, \gamma \}$ on $-T^3$ we have 
\[
\varphi_1=
\begin{pmatrix}
1 & 0 & 0 \\
1& 1 & 0 \\
0 & 0 & 1
\end{pmatrix} \text{ and }
\varphi_{0, 1}=
\begin{pmatrix}
0 & -1 & 0 \\
1 & 0 & 0 \\
0 & 0 & 1
\end{pmatrix}. 
\]
We put a metric $g_1$ and $g_0$ on $N$ such that $\varphi_1^*(g_1|_{\partial N}) =\varphi_{0,1}^*(g_0|_{\partial N}) = h$. \autoref{str} tells us that $\partial_+: \M^*_{\sigma}(M_o) \to \chi(T^3)$ is transverse to any given submanifold in $\chi(T^3)$ and misses the singular points in $\chi(T^3)$ for small perturbations. Note that $\pi_1(N)$ is abelian. Thus the unperturbed moduli space consists of only reducible instantons. Let us write 
\[
\M_1:=\M^*_{\sigma}(M_o) \text{ and } \M_2:=\M^{\Red}(N_o).
\] 
It follows from \autoref{t8.3} that 
\[
\begin{split}
\# \M^*_{\sigma}(X) & = \# \big( \partial_+(\M_1) \cap \partial_-(\M_2) \big) \\
 \# \M^*_{\sigma}(X_1) & = \# \big( \partial_+(\M_1) \cap \varphi^*_1 \circ \partial_-(\M_2) \big),
\end{split}
\]
where $\varphi^*_1: \chi(T^3) \to \chi(T^3)$ is the map induced by $\varphi_1$. To compare the difference we now put coordinates on the character variety $\chi(T^3)$. 

Let us first consider the $U(1)$-character variety $\mathcal{R}^{U(1)}(T^3)$ which is a double cover of $\chi(T^3)$. Recall that we have fixed a basis $\{ \mu, \lambda, \gamma\}$ for $\partial M = T^3$. To any $U(1)$-connection $A_L$ we assign a coordinate $(x(A_L), y(A_L), z(A_L))$ given by 
\[
x(A_L) = {1 \over {2\pi i }} \int_{\mu} a_L, \; y(A_L) = {1 \over {2\pi i }} \int_{\lambda} a_L, \; z(A_L) = {1 \over {2\pi i }} \int_{\gamma} a_L, 
\]
where $a_L = A - d \in \Omega^1(T^3, i\R)$ is the difference between $A_L$ and the product connection. Modulo $U(1)$-gauge transformations, the coordinates $x, y, z$ take values in $\R/\Z$. Then the holonomies of $A_L$ around $\mu, \lambda, \gamma$ are given respectively by 
\[
\Hol_{\mu}A_L = e^{-2\pi i x}, \;\Hol_{\lambda}A_L = e^{-2\pi i y}, \; \Hol_{\gamma}A_L = e^{-2\pi i z}. 
\]
If we restrict to the fundamental cube 
\[
\mathcal{C}_{T^3}:=\left\{(x, y, z): x, y, z\in [-1/2, 1/2]\right\}, 
\]
the $SU(2)$-character variety $\chi(T^3)$ is identified with the quotient of $\mathcal{C}_{T^3}$ under the equivalence relations $(x, y, z) \sim (-x, -y, -z)$, $(-1/2, y, z) \sim (1/2, y, z)$, $(x, -1/2, z) \sim (x, 1/2, z)$, and $(x, y, -1/2) \sim (x, y, 1/2)$. We shall restrict further to the following portion of the fundamental cube involved in the proof:
\[
\mathcal{C}^o_{T^3}:= \big\{(x, y, z): x \in [-1/2, 0], y \in [0, 1/2], z\in [-1/2, 1/2] \big\}. 
\]
We note that $\mathcal{C}^o_{T^3}$ does not include all information of $\mathcal{C}_{T^3}/\sim$, since the portion satisfying $x \in [-1/2, 0]$ and $y \in [-1/2, 0]$ is not identified in $\mathcal{C}^o_{T^3}$. Nevertheless, $\mathcal{C}^o_{T^3}$ has contained the information of all the reducible strata that are involved in the proof. 

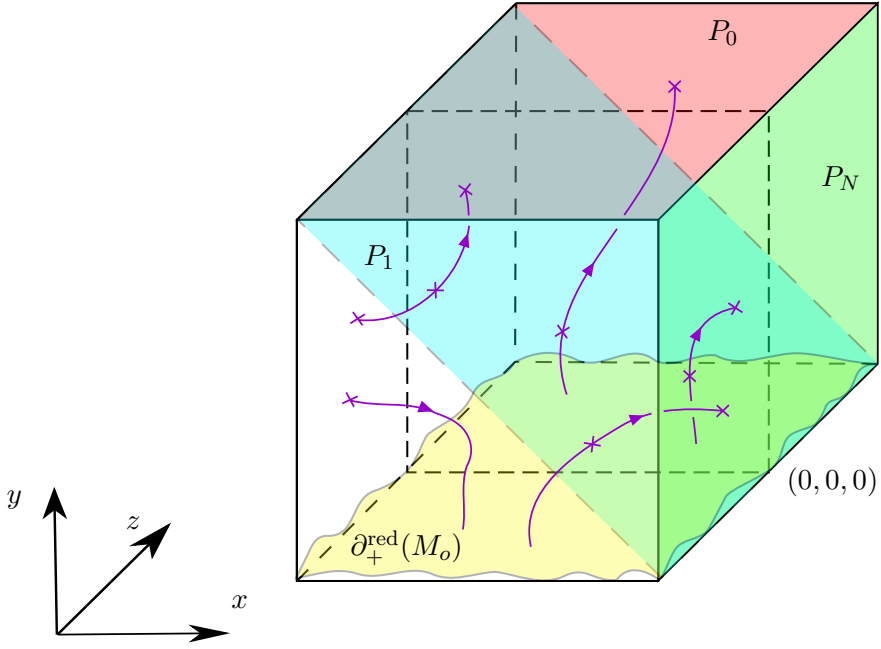
\begin{figure}[h]
\centering
\begin{picture}(200, 220)
\put(-80, -20){\includegraphics[width=0.8\textwidth]{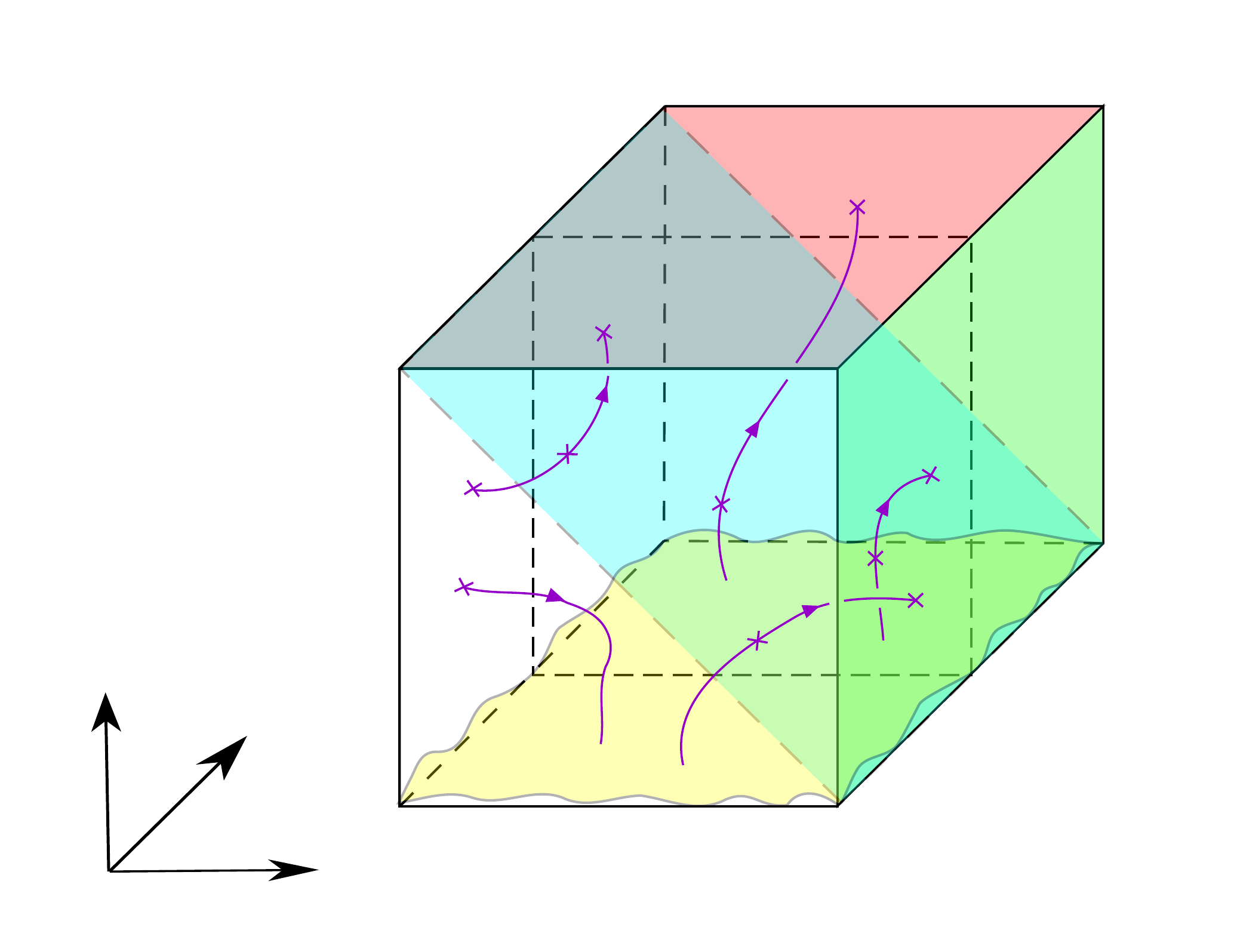}}
\put(-70, 35){$y$}
\put(-37, 28){$z$}
\put(-13, 7){$x$}
\put(170, 190){$P_0$}
\put(190, 150){$P_N$}
\put(45, 123){$P_1$}
\put(45, 27){$\partial_+^{\Red}(M_o)$}
\put(180, 40){$(0, 0, 0)$}
\end{picture}
\caption{The Cube Portion $\mathcal{C}^o_{T^3}$}
\end{figure}

Then the equivalence relations above only identify points on the lower-strata of $\mathcal{C}^o_{T^3}$, i.e. strata of dimension less than $3$ consisting of the faces, edges, and vertices of the cube. Then the image $\partial_-(\M_2)$ is given by the quotient of the plane 
\begin{equation}\label{pn}
P_N:= \left\{(0, y, z) : y \in [0, 1/2], z \in [-1/2, 1/2]\right\}
\end{equation}
whose quotient $[P_N] \subset \chi(T^3)$ is a pillowcase. The image $\partial_+(\M^{\Red}(M_o))$ of the unperturbed reducible locus on the manifold $M_o$ is given by the quotient of the plane
\begin{equation}\label{pm}
P_M:=\left\{(x, 0, z): x \in [-1/2, 0], z \in [-1/2, 1/2] \right\}
\end{equation}
whose quotient $[P_M]$ is also a pillowcase. The image $\varphi_1^* \circ \partial_-(\M_2)$ is given by the quotient of the plane
\[
P_1:=\left\{ (x, -x, z) : x \in [-1/2, 0], z \in [-1/2, 1/2] \right\}
\]
whose quotient $[P_1] \subset \chi(T^3)$ is a cylinder $[0, 1] \times S^1$. Finally we consider a parallel copy of $P_M$ given by
\[
P_0:= \left\{ (x, 1/2, z) : x \in [-1/2, 0], z \in [-1/2, 1/2] \right\}
\]
whose quotient $[P_0] \subset \chi(T^3)$ is again a pillowcase. We then orient the portion of the fundamental cube $\mathcal{C}^o_{T^3}$ by $dx \wedge dy \wedge dz$. It is straightforward to see that the equivalence relation defined on the faces of $\mathcal{C}^o_{T^3}$ is orientation-preserving. Thus all the top strata of the quotient of the planes $P_N$, $P_M$, $P_1$, and $P_0$ are oriented by the orientation induced from that of $\mathcal{C}^o_{T^3}$. Let us consider a solid 
\[
V:=\left\{ (x, y, z) : x+y \geq 0, x \in [-1/2, 0], y \in [0, 1/2], z \in [-1/2, 1/2] \right\}. 
\]
Then the quotient $[V]$ is enclosed by $-[P_1]$, $[P_N]$, and $[P_0]$ in $\chi(T^3)$. Note that 
\[
V \cap P_M =\left\{(0, 0, z): z \in [-1/2, 1/2] \right\}. 
\]
Let $[A] \in \M^{\Red}(M_o)$ be a non-central instanton such that $\partial_+([A]) \in [V]\subset \chi(T^3)$. The admissibility of $X$ implies that $H^1(M_o; \ad A_{\C}) =0$ since $[A]$ comes from a reducible instanton on $X$. Thus $[V]$ avoids the asymptotic values of the bifurcation points in $\M_{\sigma}(M_o)$ with respect to small perturbations. By choosing generic perturbations making $\partial_+$ transverse to $[V]$, we conclude that
\[
\begin{split}
\# \M_{\sigma}^*(X_1) - \# \M_{\sigma}^*(X)  & = \# \big(\partial_+(\M_1) \cap [P_1] \big) - \# \big(\partial_+(\M_1) \cap [P_N] \big) \\
& = \# \big(\partial_+(\M_1) \cap [P_0] \big) =\# \partial_+^{-1}([P_0]). 
\end{split}
\]
Now the proof has been reduced to the following result.

\begin{prop}
Continuing with notations above, one has 
\[
\# \partial_+^{-1}([P_0]) = 4 D^0_{w_{\mathcal{T}}}(X_{0,1}),
\]
where $D^0_{w_{\mathcal{T}}}(X_{0,1})$ counts the irreducible anti-self-dual $SO(3)$-instantons on the $\R^3$-bundle $E_0 \to X_{0,1}$ characterized by 
\[
p_1(E_0) = 0, \; w_2(E_0)=\PD[\mathcal{T}_0] \in H^2(X_{0,1}; \Z/2),
\]
and $\mathcal{T}_0 \subset X_{0,1}$ is the core torus of the gluing $D^2 \times T^2$. 
\end{prop}

\begin{proof}
From its construction, $\# \partial_+^{-1}([P_0])$ counts irreducible $SU(2)$-instantons on $M_o$ whose asymptotic holonomy around $\lambda$ is $\Diag(-1) \in SU(2)$. Passing to the adjoint bundle, every $SU(2)$-connection gives rise to an $SO(3)$-connection. Let $[A] \in \partial_+^{-1}([P_0])$. Then $A$ corresponds to a perturbed anti-self-dual $SO(3)$-connection $A'$ whose asymptotic holonomy around $\lambda$ is the identify. Since performing $0$-surgery amounts to gluing $D^2 \times T^2$ by sending the meridian $\partial D^2 \times \{pt.\} \times \{pt.\}$ to $\lambda$, the gluing theorem for $SO(3)$-instantons glues $[A']$ to an $SO(3)$-instanton $[A'_0]$ on $E_0$. The fact that $A'_0$ fails to lift to an $SU(2)$-connection forces $w_2(E_0) = \PD[\mathcal{T}_0]$. 

In this way, we have identified $\# \partial_+^{-1}([P_0])$ with the counting of irreducible $SO(3)$-instantons over $X_{0,1}$ up to $SU(2)$-gauge transformations. On the other hand, $ D^0_{w_{\mathcal{T}}}(X_{0,1})$ counts irreducible $SO(3)$-instantons up to $SO(3)$-gauge transformations. To compare these groups $\G^{SU(2)}(X_{0,1})$ and $\G^{SO(3)}(X_{0,1})$, we note that there exists an exact sequence \cite[Page 66]{AMY95}:
\[
\Map(X_{0,1}, \Z/2) \to \G^{SU(2)}(X_{0,1}) \xrightarrow{\alpha} \G^{SO(3)}(X_{0,1}) \to H^1(X_{0,1}; \Z/2) \to 1. 
\]
Since $H^1(X_{0,1}; \Z/2) \simeq \Z/2 \oplus \Z/2$, we conclude that $[\G^{SO(3)}(X_{0,1}): \im \alpha] = 4$, which accounts for the factor ‘4’ in the statement. 
\end{proof}

\subsection{An Application to Finite Order Diffeomorphisms}

This subsection is devoted to the proof of \autoref{fod} using the surgery formula. We briefly recall the set-up. Let $\mathcal{K} \subset Y$ be a knot in an integral homology $S^1 \times S^3$. We write $\Sigma_n$ for the $n$-fold cyclic cover of $Y$ branched along $\mathcal{K}$, and $\tau_n: \Sigma_n \to \Sigma_n$ for the covering translation. We assume $\Sigma_n$ is a rational homology sphere, and denote by $X_n$ the mapping torus of $\Sigma_n$ under the covering translation $\tau_n$. 

\begin{proof}[Proof of \autoref{fod}]
The argument is exactly the same as in the case of the Casson--Seiberg--Witten invariant \cite[Proposition 1.2]{M1}. We denote by $\mathcal{T} \subset X_n$ the mapping torus of the branching set $\tilde{\mathcal{K}} \subset \Sigma_n$, and $X'_n$ the manifold resulted from performing $(1, 1)$-surgery of $X_n$ along $\mathcal{T}$. Lemma 7.1 in \cite{M1} tells us that the restriction of the covering translation $\tau_n$ on the knot complement extends to a free self-diffeomorphism $\tau_n': \Sigma'_n \to \Sigma'_n$ with $\Sigma'_n$ the manifold given by performing $1$-surgery of $\Sigma_n$ along $\tilde{\mathcal{K}}$. Moreover $X'_n$ is the mapping torus of $\Sigma'_n$ under $\tau'_n$. From Corollary 7.7 in \cite{RS1}, we have 
\[
\lambda_{FO}(X'_n) = n \lambda(Y) + {1 \over 8} \sum_{m=0}^{n-1} \sign^{m/n}(Y, \mathcal{K}) + \frac{1}{2} \Delta''_{\mathcal{K}}(1).
\]
Let us denote by $X^0_n$ the manifold obtained by performing $(0, 1)$-surgery of $X_n$ along $\mathcal{T}$. In the proof of \cite[Proposition 1.2]{M1}, it has been shown that $X^0_n =S^1 \times Y_0(\mathcal{K})$. Combining the surgery formula and \autoref{0d0s}, we get 
\[
\lambda_{FO}(X_n) = n \lambda(Y) + {1 \over 8} \sum_{m=0}^{n-1} \sign^{m/n}(Y, \mathcal{K}).
\]
\end{proof}

\section{The Excision Formula}\label{fsfo}

\subsection{The Set-Up}
We start with a more explicit description of the excision operation. Let $(X_1, \mathcal{T}_1)$ and $(X_2, \mathcal{T}_2)$ be two pairs of admissible homology $S^1 \times S^3$ with an essentially embedded torus. We choose an identification $\nu(\mathcal{T}_i) \cong D^2 \times T^2$ for a tubular neighborhood of $\mathcal{T}_i$ as in \autoref{sffo} so that we get a basis $\{\mu_i, \lambda_i, \gamma_i \}$ of $\partial \nu(\mathcal{T}_i)=-\partial M_i$ for each $i$, where $M_i = X_i \backslash \nu(\mathcal{T}_i)$. Let $\varphi: \partial M_2 \to \partial M_1$ be a diffeomorphism so that the manifold 
\[
X_1 \#_{\varphi} X_2 := M_1 \cup_{\varphi} M_2
\]
is an admissible homology $S^1 \times S^3$. Let $A_{\varphi}$ be the matrix representing the induced map $\varphi_*: H_1(\partial M_2; \Z) \to H_1(\partial M_1; \Z)$ under the basis $\{[\mu_i], [\lambda_i], [\gamma_i]\}$. Over $D^2 \times T^2$, we write 
\[
\mu'=\{pt.\} \times S^1 \times \{pt.\}, \lambda'=\partial D^2 \times \{pt.\} \times \{pt.\}, \gamma' = \{pt.\} \times \{pt.\} \times S^1. 
\]
We let $X_{1, \varphi}:=M_1 \cup_{\varphi_1} D^2 \times T^2$ and $X_{2, \varphi}:=D^2 \times T^2 \cup_{\varphi_2} M_2$ with the gluing map $\varphi_i$ inducing the matrix $A_{\varphi}$ on first homology groups with respect to the bases $\{[\mu_i], [\lambda_i], [\gamma_i]\}$ and $\{[\mu'], [\lambda'], [\gamma']\}$. 

Since we require $X_1 \#_{\varphi} X_2$ to be an admissible homology $S^1 \times S^3$, the form of the matrix $A_{\varphi}$ can be described more explicitly. Due to the ambiguity of the choice of $\gamma_i$, one can find a framing of $\nu(\mathcal{T}_1)$ by adding to $\gamma_1$ certain multiples of $\mu_1$ and $\lambda_1$ so that $A_{\varphi}$ has the form 
\[
A_{\varphi}=
\begin{pmatrix}
a & b & 0 \\
c & d & 0 \\
p & q & 1
\end{pmatrix}. 
\]
Since $A_{\varphi}$ is orientation-reversing, we have $\det A_{\varphi}=-1$. With the help of the Mayer-Vietoris sequence, one can show that $X_1 \#_{\varphi} X_2$ is an integral homology $S^1 \times S^3$ if and only if $\GCD(aq, b)=1$ and $(aq)^2 + b^2 \neq 0$. If we mimic the argument in \autoref{l8.1} to derive the admissibility of $X_1 \#_{\varphi} X_2$ from that of $X_1$ and $X_2$ purely on the homology level, i.e. $\im r_{M_1} \cap \im r_{M_2} = \im r_{M_1} \cap \im \varphi^* \circ r_{M_2}$, we get 
\[
b=\pm 1, \quad q=0. 
\]
We note that the diffeomorphism type of $X_{1, \varphi}$ is determined by the image $\varphi_{1, *}([\lambda']) \in H_1(\partial M_1; \Z)$. Thus in this case, $X_{1 ,\varphi}$ is obtained from $X_1$ via the $(1, d)$-surgery (since $\varphi_{1, *}([\lambda']) = [\mu] + d[\lambda]$). The same can be derived for $X_{2, \varphi}$. In general we need to know more about the topology of $X_1$ and $X_2$ to determine whether $X_1 \#_{\varphi} X_2$ is admissible with a given matrix $A_{\varphi}$. 

When we take the fiber sum of $(X_1, \mathcal{T}_1)$ and $(X_2, \mathcal{T}_2)$, the gluing map $\varphi_{\mathcal{T}}$ corresponds to the matrix:
\[
A_{\varphi_{\mathcal{T}}} = 
\begin{pmatrix}
0 & 1 & 0 \\
1 & 0 & 0 \\
0 & 0 & 1
\end{pmatrix}. 
\]
We see that $X_{i, \varphi_{\mathcal{T}}} = X_i$, $i=1, 2$, and the fiber sum $X_1 \#_{\mathcal{T}} X_2$ is an admissible integral homology $S^1 \times S^3$ as $b=1, q=0$ in this case. 

To apply the neck-stretching argument, we put metrics $g_1$ on $X_1$ and $g_2$ on $X_2$ so that when restricting to collar neighborhoods of $\partial M_1$ and $\partial M_2$ they are respectively of the form 
\[
dt^2 + \varphi^*h \text{ and } dt^2 + h.
\]
where $h$ is a flat metric on $T^3$. We also let $L$ be the length-parameter of the neck, and write various neck-stretched manifolds as in \autoref{sffo}. 

\subsection{The Proof of \autoref{exif}}

To simplify the notation, we write $N=D^2 \times T^2$. Depending on the context, $N_o$ could mean either $N \cup [0, \infty) \times T^3$ or $(-\infty, 0] \times T^3 \cup N$. Then \autoref{t8.3} tells us that  
\[
\begin{split}
\# \M^*_{\sigma_1}(X_{1, \varphi})  & = \# \big( \partial_+(\M_{\sigma_1}^*({M_{1, o}})) \cap \varphi^* \circ \partial_-(\M^{\Red}(N_o)) \big) \\
\# \M^*_{\sigma_2}(X_{2, \varphi})  & = \# \big( \partial_+(\M^{\Red}(N_o)) \cap \varphi^* \circ \partial_-(\M^*_{\sigma_2}(M_{2, o})) \big),
\end{split}
\]
and moreover
\begin{equation*}
\begin{split}
\# \M^*_{\sigma_1 \# \sigma_2} (X_1 \#_{\varphi} X_2) & = 
\# \big( \partial_+ (\M^{\Red}_{\sigma_1}(M_{1, o})) \cap \varphi^* \circ \partial_-(\M^*_{\sigma_2}(M_{2, o})) \big) \\ 
&+ \# \big( \partial_+ (\M^*_{\sigma_1}(M_{1, o})) \cap \varphi^* \circ \partial_-(\M^*_{\sigma_2}(M_{2, o})) \big) \\ 
& + \# \big( \partial_+ (\M^*_{\sigma_1}(M_{1, o})) \cap \varphi^* \circ \partial_-(\M^{\Red}_{\sigma_2}(M_{2, o})) \big). 
\end{split}
\end{equation*}
The intersection number on right-hand side of the third equation above makes sense because we can first fix a generic $\sigma_2$, then choose a generic $\sigma_1$ so that the asymptotic map $\partial_+: \M^*_{\sigma}(M_{1, o}) \to \chi(T^3)$ is transverse to $\varphi^* \circ \partial_-(\M_{\sigma_2}(M_{2, o}))$. 

As in the proof of \autoref{sur1}, we regard the character variety $\chi(T^3)$ as the quotient of the fundamental cuber $\mathcal{C}_{T^3}$ under appropriate relations. We identify the copy $T^3= \partial M_1$ with a basis given by $\{\mu_1, \lambda_1, \gamma_1\}$. Then 
\[
\partial_+(\M^{\Red}(M_{1, o})) = [P_M] \text{ and }\varphi^* \circ \partial_- (\M^{\Red}(M_{2, o})) =\varphi^* [P_M],
\]
where $P_M$ is defined in (\ref{pm}). Note that the admissibility of $X_1 \#_{\varphi} X_2$ ensures there are no bifurcation points in $\M^{\Red}(M_{1, o})$ nor $\M^{\Red}(M_{2, o})$ that are asymptotic to $[P_M] \cap \varphi^*[P_M]$, for otherwise the Mayer--Vietoris sequence would produce a $U(1)$-representation $\rho$ corresponding to a bifurcation point with the property that $H^1(X_1 \#_{\varphi} X_2, \C_{\rho}) \neq 0$. Since $\dim \M^*_{\sigma_1}(M_{1, o}) = \dim \M^*_{\sigma_2}(M_{2, o}) = 1$, the transversality of the asymptotic maps implies that 
\begin{equation}\label{fibr1}
\partial_+ (\M^*_{\sigma_1}(M_{1, o})) \cap \varphi^* \circ \partial_-(\M^*_{\sigma_2}(M_{2, o})) = \varnothing. 
\end{equation}

The proof of \autoref{TSC} gives us an isotopy from the unperturbed reducible locus $\M^{\Red}(M_{2, o})$ to the perturbed one $\M^{\Red}_{\sigma_2}(M_{2, o})$. Thus we conclude the counting 
\[
\# \big( \partial_+ (\M^*_{\sigma_1}(M_{1, o})) \cap \varphi^* \circ \partial_-(\M^{\Red}_{\sigma_2}(M_{2, o})) \big) 
\]
is equal to
\[
\# \big( \partial_+ (\M^*_{\sigma_1}(M_{1, o})) \cap \varphi^* \circ \partial_-(\M^{\Red}(M_{2, o})) \big). 
\]
Note that $\varphi^* \circ \partial_-(\M^{\Red}(M_{2, o})) = \varphi^* \circ \partial_-\M^{\Red}(N_o).$ Thus 
\begin{equation}\label{fibr2}
\# \big( \partial_+ (\M^*_{\sigma_1}(M_{1, o})) \cap \varphi^* \circ \partial_-(\M^{\Red}_{\sigma_2}(M_{2, o})) \big)  = \# \M^*_{\sigma_1}(X_{1, \varphi}).
\end{equation}
Similarly we have 
\begin{equation}\label{fibr3}
\# \big( \partial_+ (\M^{\Red}_{\sigma_1}(M_{1, o})) \cap \varphi^* \circ \partial_-(\M^*_{\sigma_2}(M_{2, o})) \big) = \# \M^*_{\sigma_2}(X_{2, \varphi}). 
\end{equation}
Combining (\ref{fibr1}), (\ref{fibr2}), and (\ref{fibr3}), we conclude that
\[
\# \M^*_{\sigma_1 \# \sigma_2} (X_1 \#_{\varphi} X_2) = \# \M^*_{\sigma_1}(X_{1, \varphi}) + \# \M^*_{\sigma_2}(X_{2, \varphi}),
\]
which finishes the proof.

\subsection{Examples}

In this subsection we compute the Furuta--Ohta invariants for two families of admissible integral homology $S^1 \times S^3$ arising from mapping tori under diffeomorphisms of infinite order. 

\begin{exm}
Let $(Y_1, \mathcal{K}_1)$ and $(Y_2, \mathcal{K}_2)$ be two pairs of integral homology sphere with an embedded knot. Fix two integers $n_1, n_2 > 1$. In what follows, $j=1$ or $j=2$. We denote by $\Sigma_j$ the cyclic $n_j$-fold cover of $Y_i$ branched along $\mathcal{K}_j$, and $\tilde{\mathcal{K}}_j$ the preimage of $\mathcal{K}_j$ in the cover $\Sigma_j$. Now we take $X_j$ to be the mapping torus of $\Sigma_j$ under the covering translation, and $\mathcal{T}_j$ the mapping torus of $\tilde{\mathcal{K}}_j$. Then the fiber sum formula tells us that the Furuta--Ohta invariant of $X_1 \#_{\mathcal{T}} X_2$ is given by 
\begin{align*}
\lambda_{FO} (X_1 \#_{\mathcal{T}} X_2) & = n_1\lambda(Y_1) + {1 \over 8}  \sum_{m_1=1}^{n_1-1}\sign^{m_1/n_1}(Y_1, \mathcal{K}_1) \\
& + n_2 \lambda(Y_2) + {1 \over 8} \sum_{m_2=1}^{n_2-1}\sign^{m_2/n_2}(Y_2, \mathcal{K}_2). 
\end{align*}
We claim that the fiber $X_1 \#_{\mathcal{T}} X_2$ is the mapping torus of the knot splicing, denoted by $\Sigma_1 \#_{\mathcal{K}} \Sigma_2$, under certain self-diffeomorphism. We denote by $\tau_j$ the covering translation on $\Sigma_j$. A $\tau_j$-invariant neighborhood of $\tilde{K}_j$ is identified with $S^1 \times D^2$ where $\tau_j$ acts as 
\[
\tau_j(e^{i\eta_j}, re^{i\theta_j})  = \left(e^{i\eta_j}, re^{i(\theta_j+{2\pi \over {n_j}})} \right).
\]
A neighborhood of $\mathcal{T}_j$ is now identified with $[0, 1] \times S^1 \times D^2 / \sim$, for which we identify with $D^2 \times T^2$ as follows: 
\begin{align*}
\nu(\mathcal{T}_j) & \longrightarrow S^1 \times S^1 \times D^2 \\
[t, e^{i\eta_j}, e^{i\theta_j}] & \longmapsto \left([t], e^{i\eta_j}, re^{i(\theta_j+t{2\pi \over {n_j}})} \right).
\end{align*}
Under the identifications above, along the mapping circle the knot complements $V_1:=\Sigma_1 \backslash \nu(\tilde{\mathcal{K}}_1)$ and $V_2:=\Sigma_2 \backslash \nu(\tilde{\mathcal{K}}_2)$ are glued at time $t$ via the map
\begin{align*}
\phi_t: \partial V_2 & \longrightarrow \partial V_1 \\
(e^{i\eta_2}, e^{i \theta_2}) & \longmapsto \left(e^{i(\theta_2+t{2\pi \over {n_2}})}, e^{i(\eta_2 - t {2\pi \over {n_1}})} \right). 
\end{align*}
We write $F_t:=V_1 \cup_{\phi_t} [0, 1]_s \times T^2 \cup_{\id} V_2$ for the fiber at time $t$. We identify $F_t$ with $F_0$ by inserting the isotopy $\phi_0^{-1} \circ \phi_{(1-s)t}$ from $\phi^{-1}_0 \circ \phi_t$ to $\id$ along $[0, 1] \times T^2$, and denote by $f_t: F_t \to F_0$ this identification. From its construction, $F_0$ is the knot splicing $\Sigma_1 \#_{\mathcal{K}} \Sigma_2$. To see how the monodromy map looks like, we note that for $x \in \{0\} \times T^2 \subset F_1$, one has 
\[
\phi_0 \circ \tau_2(x) = \tau_1 \circ \phi_1(x). 
\]
Thus $\tau_1$ and $\tau_2$ combine to a map $\tau_1 \# \tau_2: F_1 \to F_0$. So the monodromy map is given by $\tau_1 \# \tau_2 \circ f^{-1}_1: F_0 \to F_0$ whose restriction to the neck $[0, 1] \times T^2 \subset F_0$ has the form
\[
(s, e^{i\eta}, e^{i\theta}) \longmapsto \left(s, e^{i(\eta+(1-s){2\pi \over {n_1}})}, e^{i(\theta+ s{2\pi \over {n_2}})} \right).
\]
In particular the monodromy map is of infinite order. 
\end{exm}

\begin{exm}
We consider the `Dehn twist' along a torus in this example. Let $Y=Y_1 \#_{\mathcal{K}} Y_2$ be a splicing of two integral homology spheres along embedded knots. We denote by $V_i$ the knot complement in $Y_i$, and write $Y=V_1 \cup [0, 1]_s \times T^2 \cup V_2$. We use $(e^{i\eta}, e^{i\theta}) \in S^1 \times S^1$ to parametrize $T^2$ so that $S^1 \times \{pt.\}$ is null-homologous in $V_1$ and $\{pt.\} \times S^1$ is null-homologous in $V_2$. Let $p, q$ be a relatively prime pair and $c:[0, 1] \to T^2$ be a curve 
\[
c(t):=(e^{i(\eta + t2\pi p)}, e^{i(\theta+ t2\pi q)}). 
\]
The Dehn twist along $c$ is a diffeomorphism $\tau_c: Y \to Y$ whose restriction on $V_1$ and $V_2$ is identity, and on the neck $[0, 1] \times T^2$ is given by 
\[
\tau_c(s, e^{i\eta}, e^{i\theta}) = (s, e^{i (\eta + s2\pi p)}, e^{i(\theta + s2\pi q)}).
\]
Then we see that $\tau_c$ has infinite order. Let $X_c$ be the mapping torus of $Y$ under $\tau_c$. Since $Y$ is an integral homology sphere, $X_c$ is an admissible homology $S^1 \times S^3$. 

We claim that $X_c$ is given by torus excision. Let $M_1 = S^1 \times V_1$ and $M_2=S^1 \times V_2$. We regard $X_c=[0, 1]_t \times Y / \sim$, where $(0, \tau_c(y)) \sim (1, y)$. Since $\tau_c|_{[0, 1] \times T^2}$ is isotopic to identify, we can identify $[0,1]_t \times ([0,1]_s \times T^2 \cup V_2) / \sim$ with $[0, 1] \times T^3 \cup M_2$ by 
\[
(t, s, e^{i\eta}, e^{i\theta}) \longmapsto \big( s, e^{2\pi t}, e^{i(\eta+(1-t)s2\pi p)}, e^{i(\theta+ (1-t)s2\pi q)} \big). 
\]
Then $[0, 1] \times T^3 \cup M_2$ is glued to $M_1$ by 
\[
(e^{i\zeta}, e^{i\eta}, e^{i\theta}) \longmapsto (e^{i\zeta}, e^{i(\eta - p\xi )}, e^{i(\theta- p \xi)}). 
\]
In terms of the gluing matrix, the gluing map $\varphi$ is given by 
\[
A_{\varphi}=
\begin{pmatrix}
0 & 1 & 0 \\
1 & 0 & 0 \\
-p & -q & 1
\end{pmatrix}. 
\]
If we write $X_{1, \varphi}=M_1 \cup_{\varphi} D^2 \times T^2$, $X_{2, \varphi}=D^2 \times T^2 \cup_{\varphi} M_2$, then  
\[
\lambda_{FO}(X_c) = \lambda_{FO}(X_{1, \varphi}) + \lambda_{FO}(X_{2, \varphi}).
\]
Finally, we note that $X_{i, \varphi}$ is obtained from $S^1 \times Y_i$ by a torus surgery. However, the gluing map $\varphi$ is not the type we considered in \autoref{surq}, and we do not know how to compare $\lambda_{FO}(X_{i, \varphi})$ with $\lambda_{FO}(S^1 \times Y_i)$ in this case. 
\end{exm}

\bibliographystyle{alpha}
\bibliography{ReferencesLFO}
\end{document}